\documentclass[11pt,reqno,oneside]{amsart}

\usepackage{graphicx,subfigure,threeparttable}
\usepackage{amssymb}
\usepackage{epstopdf}

\usepackage[a4paper, total={6in, 9.1in}]{geometry}

\usepackage{amsmath,amsfonts,amsthm,mathrsfs,amssymb,cite}
\usepackage[usenames]{color}

\usepackage{graphics}
\usepackage{framed}

\usepackage{enumerate}
\usepackage{hyperref}
\usepackage{xcolor}
\hypersetup{
  colorlinks,
  linkcolor={blue!80!black},
  urlcolor={blue!80!black},
  citecolor={blue!80!black}
}
\usepackage[capitalize,noabbrev]{cleveref}

\numberwithin{equation}{section}

\newtheorem{Theorem}{Theorem}[section]
\newtheorem{Lemma}[Theorem]{Lemma}
\newtheorem{Proposition}[Theorem]{Proposition}
\newtheorem{Definition}[Theorem]{Definition}
\newtheorem{Corollary}[Theorem]{Corollary}

\newtheorem{Remark}[Theorem]{Remark}
\numberwithin{equation}{section}

\newtheorem{thm}{Theorem}[section]

\theoremstyle{definition}

\numberwithin{equation}{section}

 \def\p{\partial} 
\def \Vh0{\stackrel{\circ}{V}_h}

  \def\f{\frac}  

\def\p{\partial}

\newcommand{\lc}
{\mathrel{\raise2pt\hbox{${\mathop<\limits_{\raise1pt\hbox
{\mbox{$\sim$}}}}$}}}

\newcommand{\gc}
{\mathrel{\raise2pt\hbox{${\mathop>\limits_{\raise1pt\hbox{\mbox{$\sim$}}}}$}}}

\newcommand{\ec}
{\mathrel{\raise2pt\hbox{${\mathop=\limits_{\raise1pt\hbox{\mbox{$\sim$}}}}$}}}

\def\bb{\begin{equation}} \def\ee{\end{equation}}

\def\beqn{\begin{eqnarray}}  \def\eqn{\end{eqnarray}}

\def\beqnx{\begin{eqnarray*}} \def\eqnx{\end{eqnarray*}}

\def\bn{\begin{enumerate}} \def\en{\end{enumerate}}

\def\bd{\begin{description}} \def\ed{\end{description}}



\allowdisplaybreaks

\title[Surface concentration of transmission resonance]{Surface concentration of transmission eigenfunctions }

\author{Yat Tin Chow}
\address{Department of Mathematics, University of California, Riverside, CA, USA}
\email{yattinc@ucr.edu}

\author{Youjun Deng}
\address{School of Mathematics and Statistics, HNP-LAMA, Central South University, Changsha, Hunan, China
}
\email{youjundeng@csu.edu.cn}

\author{Hongyu Liu}
\address{Department of Mathematics, City University of Hong Kong, Kowloon, Hong Kong, China}
\email{hongyu.liuip@gmail.com, hongyliu@cityu.edu.hk}

\author{Mahesh Sunkula}
\address{Department of Mathematics, Purdue University, West Lafayette, USA}
\email{msunkula@purdue.edu}

\begin{document}
\maketitle

\begin{abstract}

The transmission eigenvalue problem is a type of non-elliptic and non-selfadjoint spectral problem that arises in the wave scattering theory when invisibility/transparency occurs. The transmission eigenfunctions are the interior resonant modes inside the scattering medium. We are concerned with the geometric rigidity of the transmission eigenfunctions and show that they concentrate on the boundary surface of the underlying domain in two senses. This substantiates the recent numerical discovery in \cite{CDHLW} on such an intriguing spectral phenomenon of the transmission resonance. Our argument is based on generalized Weyl's law and certain novel ergodic properties of the coupled boundary layer-potential operators which are employed to analyze the generalized transmission eigenfunctions.

\medskip



\end{abstract}

\section{Introduction}

We first introduce the time-harmonic acoustic wave scattering, which is the physical origin of the transmission eigenvalue problem in our study and moreover shall be used to motivate our mathematical analysis.

Let $D$ be an open connected and bounded domain in $\mathbb{R}^d$, $d\geq 3$, with a $\mathcal{C}^\infty$ smooth boundary
$\partial D$ and a connected complement $\mathbb{R}^d\backslash\overline{D}$. In the physical setting, $D$ signifies the support of an inhomogeneous medium scatterer located in an otherwise uniformly homogeneous space. The medium parameter is characterised by the refractive index which is normalised to be $1$ in $\mathbb{R}^d\backslash\overline{D}$ and is assumed to be $Q\in\mathbb{R}_+$ and $Q\neq 1$ in $D$. Set $V=(Q^2-1)\chi_D+0\chi_{\mathbb{R}^d\backslash\overline{D}}$, which is referred to as the scattering potential. Let $\psi_0\in\mathcal{C}^\infty(\mathbb{R}^d)$ be an impinging wave field which is an entire solution to $(\Delta+\kappa^2) \psi_0=0$ in $\mathbb{R}^d$, where $\kappa\in\mathbb{R}_+$ signifies the angular frequency of the wave. The impingement of $\psi_0$ on the scattering potential $(D, V)$, or equivalently on the scattering medium $(D, Q)$, leads to the following Helmholtz system for the total wave field $\psi\in H_{loc}^1(\mathbb{R}^d)$:
\begin{eqnarray}
\begin{cases}
\Delta \psi+\kappa^2(1+V) \psi = 0 & \text{ in }\ \mathbb{R}^d\medskip\\
( \partial_r - \mathrm{i} \kappa) (\psi - \psi_0) = \mathcal{O}( r^{- \frac{d+1}{2}})  & \text{ as }\ r \rightarrow \infty,
\end{cases}
\label{transmissionk}
\end{eqnarray}
where $\mathrm{i}:=\sqrt{-1}$ and $r:=|x|$ for $x\in\mathbb{R}^d$. The last limit in \eqref{transmissionk} is known as the Sommerfeld radiation condition which holds uniformly in the angular variable $\hat x:=x/|x|\in\mathbb{S}^{d-1}$ and characterises the outgoing nature of the scattered $\psi^s:=\psi-\psi_0$. The well-posedness of the scattering system \eqref{transmissionk} is known (cf. \cite{LSSZ,McLean}) and in particular it holds that
\begin{equation}\label{eq:f1}
\psi(x)=\psi_0(x)+\frac{e^{\mathrm{i}\kappa r}}{r^{(d-1)/2}} \psi_\infty(\hat x)+\mathcal{O}\left(\frac{1}{r^{(d+1)/2}}\right)\quad\mbox{as}\ r\rightarrow \infty.
\end{equation}
In \eqref{eq:f1}, $\psi_\infty$ is referred to as the far-field pattern which encodes the scattering information of the underlying scatterer under the probing of the incident wave $\psi_0$. An inverse problem of industrial importance is to recover $(D, V)$ by knowledge of $\psi_\infty$. It is clear that the recovery fails if $\psi_\infty\equiv 0$, namely invisibility/transparency occurs. In such a case, one has by the Rellich theorem \cite{CK_book} that $\psi=\psi_0$ in $\mathbb{R}^d\backslash\overline{D}$. Hence, if setting $u=\psi|_{D}$ and $v=\psi_0|_D$, it holds that
\begin{equation}\label{eq:trans1}
\begin{cases}
& \Delta u+\kappa^2 Q^2 u=0\hspace*{1.62cm}\mbox{in}\ \ D,\medskip\\
& \Delta v+\kappa^2 v=0\hspace*{2.15cm}\mbox{in}\ \ D,\medskip\\
& u=v, \quad\partial_\nu u=\partial_\nu v\hspace*{1.2cm}\mbox{on}\ \ \partial D,
\end{cases}
\end{equation}
where and also in what follows $\nu\in\mathbb{S}^{d-1}$ stands for the exterior unit normal to $\partial D$. That is, if invisibility/transparency occurs, the total and incident wave fields fulfil the spectral system \eqref{eq:trans1}, which is referred to as the transmission eigenvalue problem in the literature.

Let us consider the spectral study of the transmission eigenvalue problem \eqref{eq:trans1}. It is clear that $u=v\equiv 0$ are trivial solutions. If there exist nontrivial $u\in L^2(D)$ and $v\in L^2(D)$ such that $u-v\in H_0^2(D)$ and the first two equations in \eqref{eq:trans1} are fulfilled, then $\kappa$ is referred to as a transmission eigenvalue and $u, v$ are the corresponding transmission eigenfunctions. It is emphasised that in this paper, we are mainly concerned with real transmission eigenvalues, namely $\kappa\in\mathbb{R}_+$, though there exist complex transmission eigenvalues. The transmission eigenvalue problem is non-elliptic and non-selfadjoint, and this is partly evidenced by setting $w=u-v$ and verifying that
\begin{equation}\label{eq:r2}
(\Delta+\kappa^2)\big(\Delta+\kappa^2Q^2\big) w=0\quad\mbox{in}\ H_0^2(D),
\end{equation}
which is a fourth-order PDE eigenvalue problem and quadratic in $\lambda=\kappa^2$. The following connection of the transmission eigenfunctions with the scattering problem \eqref{transmissionk}--\eqref{eq:f1} shall be a useful observation for our subsequent study.
\begin{thm}[Proposition~4.2, \cite{BL2017b}]\label{thm:p1}
Suppose that $\kappa\in\mathbb{R}_+$ is a transmission eigenvalue and $u, v\in L^2(D)$ are the associated transmission eigenfunctions to \eqref{eq:trans1}. Then {\color{black} for any sufficiently small $\varepsilon>0$ (i.e. there exists $\varepsilon_0 > 0$ such that for all $\varepsilon < \varepsilon_0$,)} there exists $g_\varepsilon\in L^2(\mathbb{S}^{d-1})$ such that
\begin{equation}\label{eq:h1}
\|v_{g_\varepsilon}-v\|_{L^2(D)}<\varepsilon,\quad v_{g_{\varepsilon}}(x):=\int_{\mathbb{S}^{d-1}} \exp(\mathrm{i}\kappa x\cdot\theta) g_\varepsilon(\theta)\, ds(\theta),\ \ \forall x\in D. 
\end{equation}
Moreover, if taking $\psi_0=v_{g_\varepsilon}$ in \eqref{transmissionk}, one has $\|\psi_\infty\|_{L^2(\mathbb{S}^{d-1})}\leq C_{V, \kappa}\varepsilon$ and $\|u-\psi\|_{L^2(D)}\leq C_{V,\kappa}\varepsilon$, where $C_{V, \kappa}$ is a positive constant depending only on $V$ and $\kappa$.
\end{thm}
\noindent In the physical setting, $v_{g_\varepsilon}$ is referred to as a Herglotz wave, and Theorem~\ref{thm:p1} states that if $u, v$ are transmission eigenfunctions, they respectively correspond to the total and incident wave fields (restricted in $D$) from a nearly invisible/transparent scattering scenario.

{\color{black}The transmission eigenvalue problem was first investigated by Kirsch in 1986 \cite{Kir} and soon after it was more systematically studied by Colton and Monk in 1988 \cite{CM}. The exclusion of the transmission eigenvalues can guarantee the injectivity and dense range of the far-field operator and hence the validity of a certain reconstruction scheme for the inverse scattering problem mentioned earlier. In recent years, the study of transmission eigenvalue problems reacquired popularity due to many challenging mathematical questions that it poses and to many possible applications}. The spectral properties of the transmission eigenvalues have been extensively and intensively studied in the literature, and we refer to \cite{CHreview,CCHn, Liureview} for reviews and surveys on the existing developments on this aspect. In particular, we would like to note that generically there exist infinitely many real transmission eigenvalues which accumulate at $\infty$. Recently, several intrinsic geometric properties of the transmission eigenfunctions were discovered and investigated. In \cite{BLLW,BL2018,BL2017b}, it is shown that the transmission eigenfunctions are generically vanishing around corners or high-curvature places of $\partial D$. In \cite{CDHLW}, it is found that ``many" transmission eigenfunctions tend to localize around on $\partial D$ in the sense that the $L^2$-energies of those eigenmodes are concentrated in a neighbourhood of $\partial D$; see Fig.~1 for two typical numerical illustrations, where the transmission eigenfunctions are plotted associated with different $(D, V)$'s. It is highly intriguing to have the following observations:
\begin{enumerate}
\item It is clear that the transmission eigenfunctions are interior resonant modes which exhibit highly oscillatory patterns. Interestingly, the high oscillations of these resonant modes are localized on $\partial D$. The study on the eigenfunction concentration is a central topic in mathematical physics and spectral theory; see e.g. \cite{Zel} and the references cited therein. However, the concentration phenomenon presented here is peculiar and different from the existing ones in the literature for the classical eigenvalue problems. {\color{black} In fact, for a concrete comparison, one may consider the eigenfunctions associated with the classical Dirichlet/Neumann Laplacian, which may exhibit confinement or uniform distribution patterns depending on the underlying geometry, but generally do not always concentrate around the boundary of the underlying domain; see \cite{Wiki,Zel} for a general discussion. Hence, the boundary localization of transmission eigenfunctions represents a new spectral phenomenon.}

\item The physical intuition to explain the surface-localizing behaviour can be described as follows. According to Theorem~\ref{thm:p1}, the transmission eigenfunctions are (at least approximately) restrictions of the incident and total wave fields when invisibility/transparency occurs. Hence, in order to reach the invisibility/transparency, a `smart' way for the propagating wave is to `slide' over the surface of the scattering object, namely $(D,V)$, and return to its original path after passing through the object. This clearly gives rise to the regular pattern depicted in Fig.~1, where the wave fields inside the object clearly propagates along the surface $\partial D$. Applications to invisibility cloaking of the transmission eigenfunctions were discussed in \cite{JL,LWZ}. 

\item The surface-localization indicates that the transmission eigenfunctions carry the geometric information of the underlying scattering medium and hence is a global geometric rigidity property. This spectral property has been proposed for super-resolution wave imaging, generation of the so-called pseudo plasmon modes with a potential bio-sensing application and the artificial electromagnetic mirage effect \cite{CDHLW,DLWW}, as well as stress concentration in elasticity \cite{JLZZ2}. 

\end{enumerate}
\begin{figure}[t]
\centering
\subfigure[$u$-part, $\kappa=5.5496$]{\includegraphics[width=0.3\textwidth]{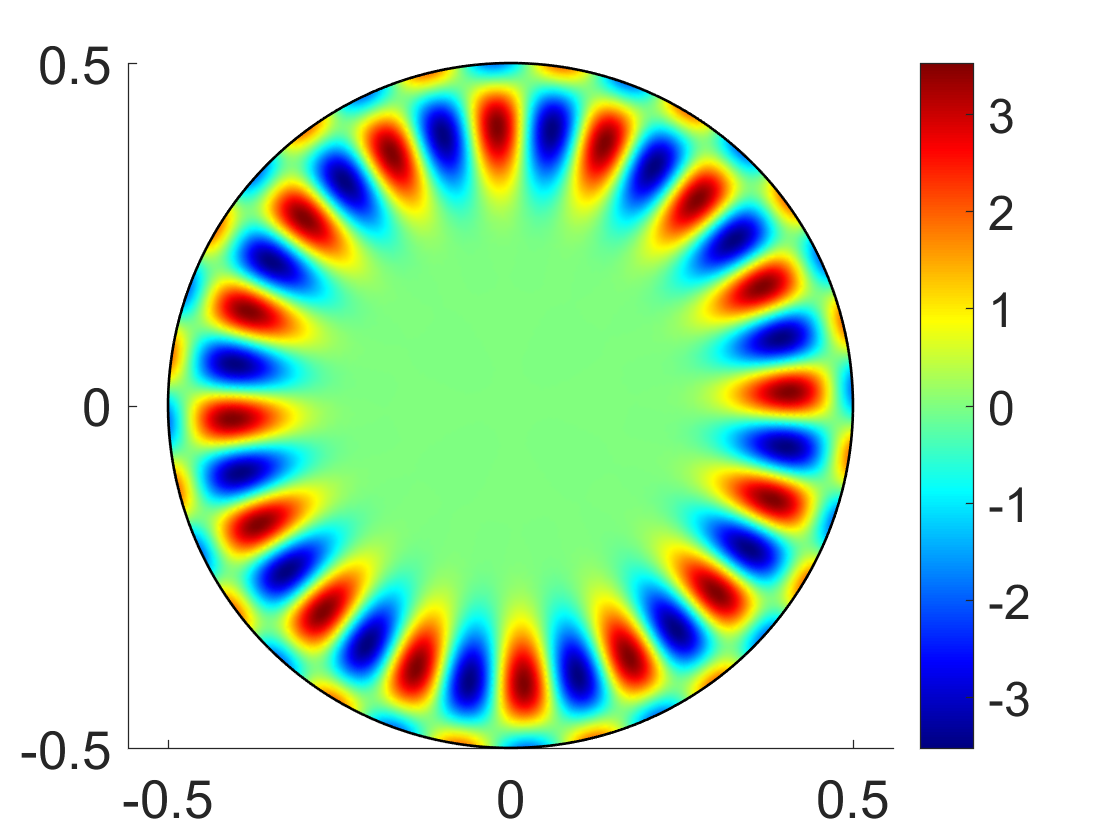}}\qquad\qquad
\subfigure[$v$-part, $\kappa=5.5496$]{\includegraphics[width=0.3\textwidth]{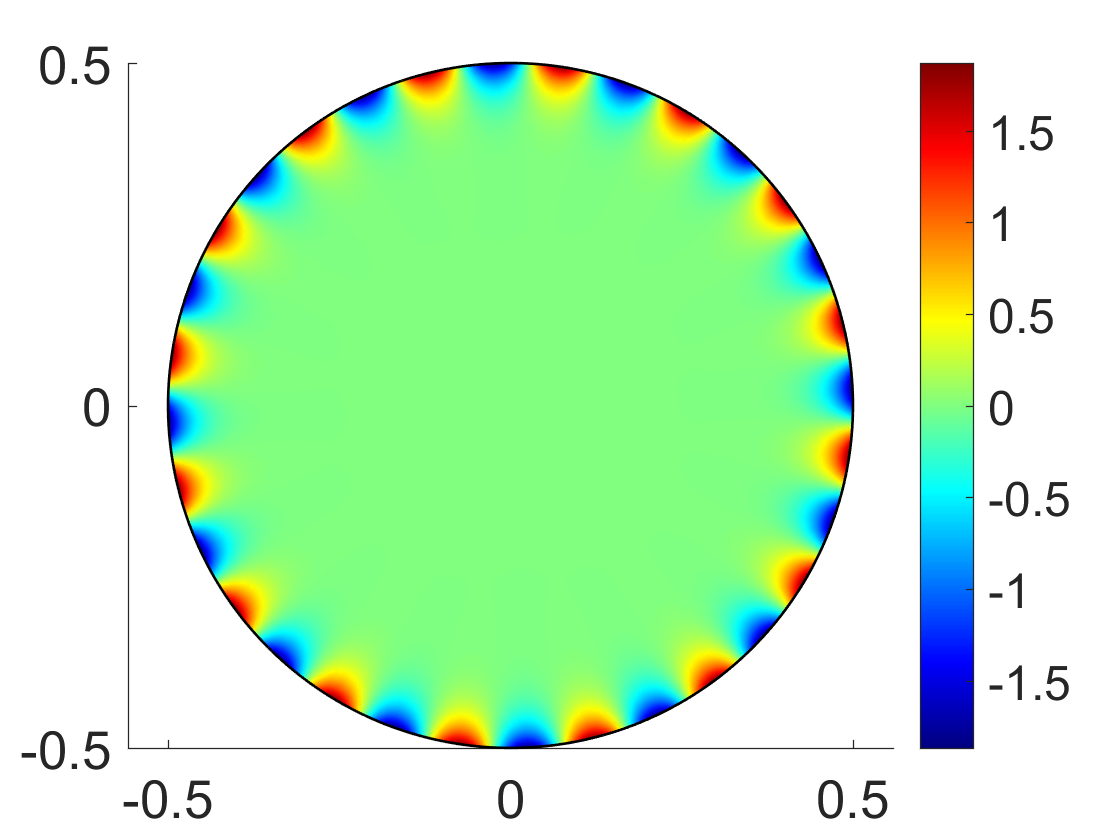}}\\
\subfigure[$v$-part, $\kappa=3.1023$]{\includegraphics[width=0.3\textwidth]{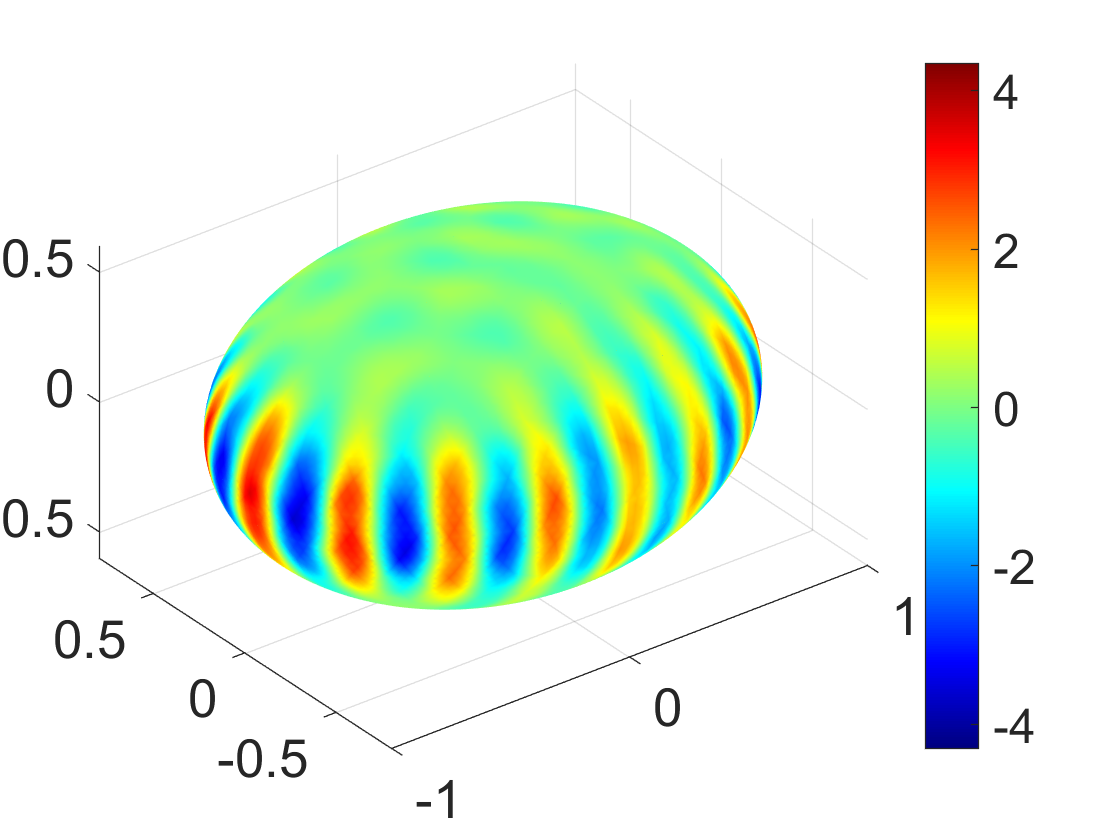}}\qquad\qquad
\subfigure[$v$-part, $\kappa=3.1023$]{\includegraphics[width=0.3\textwidth]{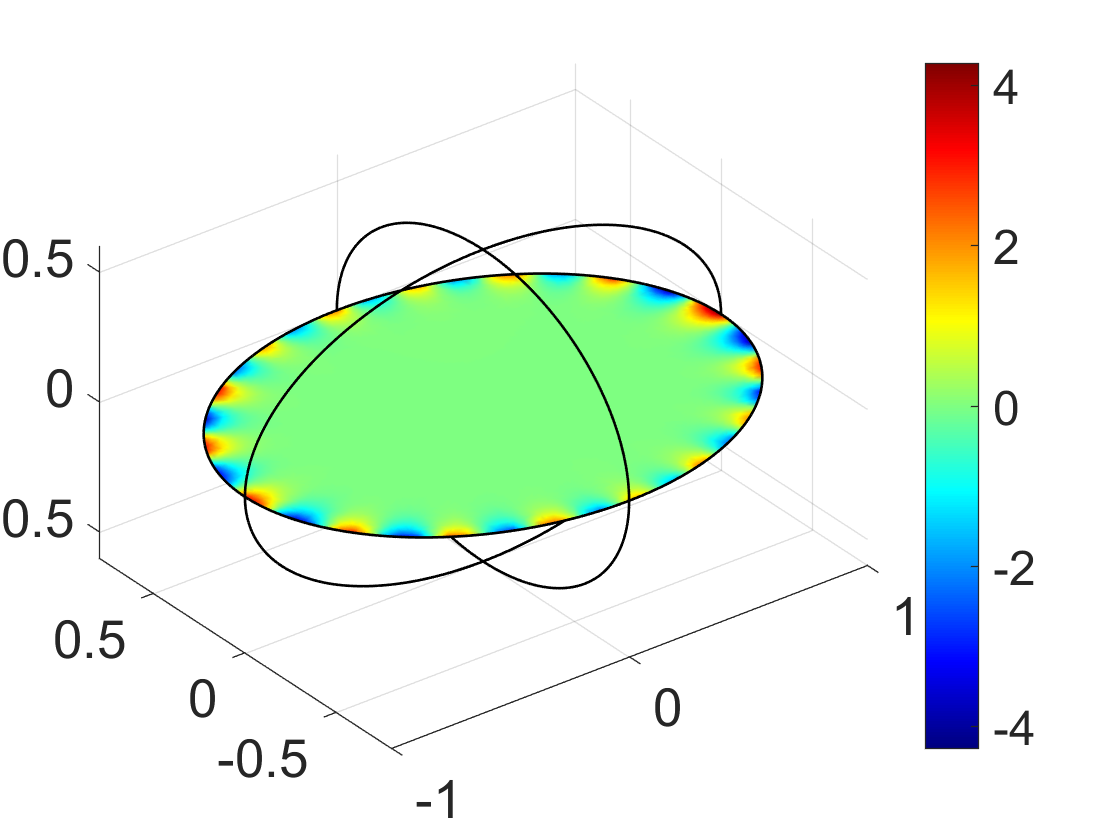}}
\caption{Transmission eigenfunctions to \eqref{eq:trans1} associated with $Q=8$ (or equivalently, $V=63$) for different $D$'s and $\kappa$'s. In Figures (a) and (b), the domain $D$ is a central disk of radius 1. In Figure (c), $D$ is an ellipsoid. Figure (d) is the slice plotting of Figure (c) at $x_3=0$ for $x=(x_j)_{j=1}^3\in\mathbb{R}^3$. }
\end{figure}
However, the discovery in \cite{CDHLW} is mainly based on numerics, though the case within radial geometry is verified rigorously \cite{CDHLW,DJLZ,JLZZ1} by using the analytic expressions of the transmission eigenfunctions via the Bessel and spherical harmonic functions. It is the aim of this paper to derive a theoretical understanding of this peculiar spectral phenomenon. First, we make essential use of the layer-potential operators, especially the so-called Neumann-Poincar\'e operators, which are used to reformulate the transmission eigenvalue problem \eqref{eq:trans1} as a spectral system associated with coupled boundary integral equations. Treating those potential operators as pseudo-differential operators and exploiting their quantitative properties, we introduce a certain generalized transmission eigenvalue problem which approximates the original transmission eigenvalue problem in a certain sense. It is emphasized that the set of generalized transmission eigenfunctions include the transmission eigenfunctions as a subset. Second, we show that the (local) Dirichlet energy of the generalised transmission eigenfunctions is localized around $\partial D$ with quantitative characterisations. Then via quantum ergodicity, we can also establish that the generalised transmission eigenfunctions are quantitatively localized around $\partial D$ almost surely. In establishing those quantitative results, generalized Weyl's law and certain novel ergodic properties of the coupled layer potential operators are explored.

{\color{black} Finally, we would like to make several remarks to highlight the novel contributions as well as the limitations of our work and also point out the potential extensions for future study. Throughout our study, it is assumed that $\partial D$ is $\mathcal{C}^\infty$ smooth and $Q^2=1+V$ is constant. 
However, we would like to emphasise that the numerics in \cite{CDHLW} indicate that the boundary-localizing property hold in more general scenarios with 
Lipscthiz domains and variable $Q$. 
The constant $Q$ is needed for reformulating the transmission eigenvalue problem into a system of coupled boundary integral equation via the layer-potential theory. The $C^\infty$-smooth $\partial D$ is needed for treating the involved boundary integral operators as pseudo-differential operators and applying the operator calculus. 
Though with those limitations mentioned above due to technical requirements, our study represents the first one in the literature on theoretically understanding the peculiar boundary-localizing behaviour of the transmission eigenfunctions. Moreover, even in the current setup, the mathematical analysis presents significant technical difficulties and challenges. Both the results and the mathematical techniques developed enrich the spectral theory of transmission eigenvalue problems and also open up a field of research for many potential developments. First, we believe the theoretical framework developed in this paper can be used to treat similar phenomena for transmission resonances arising in electromagnetic and elastic scattering. Second, relaxing the technical requirements discussed above represents a highly interesting direction of research, though it will present significant challenges. We shall investigate these and other possible developments in our future work. Finally, it is remarked that we mainly consider $d\geq 3$ in the current article, though in principle our study can be extended to cover the case $d=2$. However, the boundary integral operators in two dimensions is of a different type and possess distinct properties compared to the case $d\geq 3$, which require different calculations and analysis. Hence, in order to be focusing and concise in the exposition, we mainly consider the case $d\geq 3$ and choose to present the two-dimensional result somewhere else. }

The rest of the paper is organized as follows. In Section 2, we present the integral reformulation of the transmission eigenvalue problem \eqref{eq:trans1} as well as the quantitative properties of the layer-potential operators as pseudo-differential operators. In Section 3, we present the generalized transmission eigenvalue problem and discuss its properties. Section 4 is devoted to the main results on the surface localization as well as the corresponding proofs.

\section{Integral formulation and layer-potential operators}\label{sect:2}

\subsection{Preliminaries}

From this section and onward, let us only consider $D \subset \mathbb{R}^d$ with $\partial D \in \mathcal{C}^{\infty}$ and {$d \geq 3$}.  We discuss the layer potential formulism  (cf. \cite{CK_book,regularity5,kellog, folland,mcowen}).
For a given $\kappa\in\mathbb{R}_+$, we introduce the single- and double-layer potential operators as follows:
\beqn
    \mathcal{S}^\kappa_{\partial D} [\phi] ({x}) &:=&  \int_{\partial D} G_\kappa({x}-{y}) \phi(y) d \sigma(y) ,\\
  \mathcal{D}^\kappa_{\partial D} [\phi] (x) &:=& \int_{\partial D} \partial_{ \nu_y } G_\kappa(x-y) \phi(y) d \sigma(y) ,\label{eq:nn1}
\eqn
for $x \in \mathbb{R}^d$ and $\phi\in L^2(\partial D)$. Here and also in what follows, $\nu_y$ signifies the exterior unit normal vector at $y\in \partial D$ and $G_\kappa$ is the outgoing fundamental solution of the partial differential operator $\Delta+\kappa^2$ in $\mathbb{R}^d$ given by
\beqn
     G_\kappa (x-y) = \left( C_{d}   \kappa^{ d-2} \mathrm{i} \right)  ( \kappa |x-y| )^{- \frac{d-2}{2}} H^{(1)}_{\frac{d-2}{2}}(\kappa |x-y|),
    \label{fundamental2}
\eqn
where $C_d$ is some dimensional constant and ${H}^{(1)}_{\frac{d-2}{2}}$ is the Hankel function of the first kind and order $(d-2)/2$. {\color{black} In \eqref{eq:nn1}, $\partial_{\nu_y} G_\kappa(x-y):=-\nu_y\cdot\nabla G_\kappa(x-y)$. }
It is known that the single-layer potential $\mathcal{S}^\kappa_{\partial D}$ satisfies the following jump condition on $\partial D$:
\beqn
    \f{\p}{\p \nu} \left(  \mathcal{S}^\kappa_{\partial D} [\phi] \right)^{\pm} = (\pm \f{1}{2} I + {\mathcal{K}^\kappa_{\partial D}}^* )[\phi]\,,
    \label{jump_condition2}
\eqn
where the superscripts $\pm$ indicate the limits from outside and inside $D$ respectively, and
${\mathcal{K}^\kappa_{\partial D}}^*: L^2(\partial D) \rightarrow L^2(\partial D)$ is the Neumann-Poincar\'e operator defined by
\beqn
    {\mathcal{K}^\kappa_{\partial D}}^* [\phi] (x) := \int_{\partial D} \partial_{\nu_x} G_\kappa(x-y) \phi(y) d \sigma(y),\ \ \forall x\in \partial D, 
    \label{operatorK2}
\eqn
{\color{black} with $\partial_{\nu_x}G_\kappa(x-y):=\nu_x\cdot\nabla G_\kappa(x-y)$.} By an abuse of notations, whenever no confusions arise, we denote the restriction of the layer potential operator onto the boundary with the same notations, i.e. we write $ \mathcal{S}^\kappa_{\partial D}: L^2(\partial D) \rightarrow L^2(\partial D)$ and $  \mathcal{D}^\kappa_{\partial D} : L^2(\partial D) \rightarrow L^2(\partial D)$.

With the above preparation, we consider the transmission eigenvalue problem \eqref{eq:trans1} by taking
\begin{equation}\label{eq:tt1}
u = \mathcal{S}^{\kappa Q}_{\partial D} [\phi] \,,\, \quad v = \mathcal{S}^{\kappa}_{\partial D} [\varphi]\   \text{ on } D,
\end{equation}
where $(\phi, {\varphi}) \in L^2 (\partial D) \times L^2 (\partial D)$. Then by using \eqref{jump_condition2}, we can rewrite \eqref{eq:trans1} into the following boundary integral system:
\begin{eqnarray}
\begin{pmatrix}
\mathcal{S}^{\kappa Q}_{\partial D} &  - \mathcal{S}^{\kappa}_{\partial D} \\
 - \f{1}{2} I + {\mathcal{K}^{\kappa Q}_{\partial D}}^* &    - ( - \f{1}{2} I + {\mathcal{K}^{\kappa}_{\partial D}}^* )
\end{pmatrix}
\begin{pmatrix}
\phi \\
\varphi
\end{pmatrix}
=
\begin{pmatrix}
0  \\
0
\end{pmatrix} ,
\label{eq:operator_new}
\end{eqnarray}
or that
\begin{equation}\label{eq:tt2}
\varphi = \left[ \mathcal{S}^{\kappa}_{\partial D}  \right]^{-1}  \mathcal{S}^{\kappa Q}_{\partial D}  [\phi]   \, , \, \quad \left(  ( - \f{1}{2} I + {\mathcal{K}^{\kappa Q}_{\partial D}}^* ) - ( - \f{1}{2} I + {\mathcal{K}^{\kappa}_{\partial D}}^* ) \left[ \mathcal{S}^{\kappa}_{\partial D}   \right]^{-1}  \mathcal{S}^{\kappa Q}_{\partial D} \right)  [\phi]  = 0.
\end{equation}{\color{black} whenever $\kappa^2$ is not a Dirichlet eigenvalue of $D$, where in this case $ \mathcal{S}^{\kappa}_{\partial D} $ is invertable as an operator from $L^2(\partial D)$ to $H^{1} (\partial D)$ \cite{CK_book,AGJKLSW}, see also \cite{McLean} Theorem 7.5.} 
Instead of \eqref{eq:trans1}, we shall consider \eqref{eq:tt1}--\eqref{eq:tt2} in what follows.

We shall treat the layer potential operators in \eqref{eq:tt1}--\eqref{eq:tt2} as pseudo-differential operators in our subsequent analysis. For clarity and self-containedness, we briefly discuss the pseudo-differential operators and refer to \cite{Hor1,Hor2,Hor3,Hor4} for more relevant details. Throughout the rest of the paper, we let $h := \kappa^{-1}$ and $ \Phi \text{SO}^{m}_h$ denote the set of pseudo-differential operators with the action $\mathrm{Op}_{a,h} := \mathcal{F}^{-1}_h \circ \mathfrak{M}_a \circ\mathcal{F}_h$, where $\mathcal{F}_h$ is the semi-classical Fourier transform, $\mathfrak{M}_a$ is the action with multiplication by $a$ and $a$ belongs to the symbol class of order $m$.
{\color{black} For clarity, we first consider $\mathbb{R}^{d-1}$, we have for $f \in \mathcal{S}(\mathbb{R}^{d-1}) $ in the space of Schwartz  functions,
\begin{equation}\label{eq:ff1}
\mathcal{F}_h f(x):=(2\pi h)^{-(d-1)/2}\int_{\mathbb{R}^{d-1}} e^{-\mathrm{i}\frac{x\cdot y}{h}} f(y)\, dy,
\end{equation}
and for $f \in \mathcal{S}'(\mathbb{R}^{d-1}) $ in the space of tempered distribution, 
\begin{equation}\label{eq:ff2}
\langle \mathcal{F}_h f , g \rangle := \langle  f ,  \mathcal{F}_h^* g \rangle =  \langle  f ,  \mathcal{F}_h^{-1} g \rangle \quad \forall g \in \mathcal{S}(\mathbb{R}^{d-1})\,.
\end{equation}
Now for $a \in \mathcal{S}(\mathbb{R}^{2d-2}) $, we also define
\begin{equation}\label{eq:pp1}
\mathrm{Op}_{a,h}f(x):=(2\pi h)^{-(d-1)}\int_{\mathbb{R}^{d-1}}\int_{\mathbb{R}^{d-1}} e^{\mathrm{i}\frac{(x-y)\cdot\xi}{h}}a\left(x,\xi\right) f(y)\, dyd\xi.
\end{equation}
For $a\in \mathcal{S}'(\mathbb{R}^{2d-2})$, $\mathrm{Op}_{a,h} :  \mathcal{S}(\mathbb{R}^{d-1}) \rightarrow  \mathcal{S}'(\mathbb{R}^{d-1})$ defines a continuous operator such that 
\begin{eqnarray}\label{eq:pp2}
\langle \mathrm{Op}_{a,h}f , g \rangle 
& := &   \langle a ,  \mathcal{F}_{2,h} ( g \boxtimes f ) \rangle  \quad  \forall f, g \in \mathcal{S}(\mathbb{R}^{d-1})  
\end{eqnarray}
where we write
\begin{eqnarray*}
( g \boxtimes f )(x,y) &:=& g(x) \overline{ f(y) }  \in \mathcal{S}(\mathbb{R}^{2d-2}) \\
\mathcal{F}_{2,h} \, \phi (x,\xi)  &:=&  (2\pi h)^{-(d-1)} \int_{\mathbb{R}^{d-1}}  e^{- \mathrm{i}\frac{(x-y)\cdot\xi}{h}} \phi(x,y) \, dy  \in \mathcal{S}(\mathbb{R}^{2d-2}) \quad \forall \phi \in \mathcal{S}(\mathbb{R}^{2d-2}) \,.
\end{eqnarray*}
We now notice the map $a \mapsto  \mathrm{Op}_{a,h} $ is also continuous from $\mathcal{S}'(\mathbb{R}^{2d-2})$ to the set of continuous linear operators from $  \mathcal{S}(\mathbb{R}^{d-1}) $  to $  \mathcal{S}'(\mathbb{R}^{d-1}) $.
Moreover, from sequential density of $\mathcal{S}(\mathbb{R}^{2d-2})$ in $\mathcal{S}'(\mathbb{R}^{2d-2}) $ in the weak-$*$ topology, we may always approximate in the weak sense $ \mathrm{Op}_{a,h}$ with $a\in \mathcal{S}'(\mathbb{R}^{2d-2})$ by a sequence of operators $\{ \mathrm{Op}_{a_n,h} \}_{n \in \mathbb{N}}$ with $\{ a_n \}_{n \in \mathbb{N}} \in  \mathcal{S}(\mathbb{R}^{d-1})$ converging to $a \in \mathcal{S}'(\mathbb{R}^{2d-2})$.  In fact, we have
\begin{eqnarray*}
 \langle a_n ,  f \rangle \rightarrow \langle a ,  f \rangle \quad \forall f  \in \mathcal{S}(\mathbb{R}^{2d-2}) \, \Rightarrow \, \langle \mathrm{Op}_{a_n,h}f , g \rangle  \rightarrow \langle \mathrm{Op}_{a,h}f , g \rangle \quad \forall f,g  \in \mathcal{S}(\mathbb{R}^{d-1}) \,.
\end{eqnarray*}
}Next we briefly mention the following notations and definitions of spaces over $\partial D$ that are used in our study on for $Q > 0$:
{\color{black}
\beqnx
\begin{split}
 \bigcup_{i=1}^N U_i = \partial D \,,\quad & F_i : \pi^{-1} (U_i) \rightarrow U_i \times \mathbb{R}^{d-1}\,,  \quad  \sum_{i,j=1}^N \psi_{ij} = 1 \,, \quad { \mathrm{supp}(\psi_{ij}) } \subset U_i\, ; \\
\mathrm{supp}(\psi_{ij}) \bigcap \mathrm{supp}(\psi_{i'j'})&  \neq \emptyset  \Rightarrow \mathrm{supp}(\psi_{ij}) \bigcup \mathrm{supp}(\psi_{i'j'}) \subset U_i \bigcap U_{i'} \forall i,j,i',j' =1,...,N
\end{split}
\eqnx
}{\color{black} where
$\pi: T^*(\partial D) \rightarrow \partial D$ signifies the cotangent bundle of $\partial D$ with the bundle projection $\pi$ from $T^*(\partial D)$ onto the base space $\partial D$. Here, we write $\{ U_i \}_{i=1}^N$ as a locally trivialling (finite) open cover of $\partial D$ with associated local trivialization $F_i : \pi^{-1} (U_i) \rightarrow U_i \times \mathbb{R}^{d-1}$ and $\{ \psi_{ij} \}_{i,j=1}^N$ as a partition of unity subordinate to this cover $\{ U_i \}_{i=1}^N$ with the additional property such that $\mathrm{supp}(\psi_{ij}) \bigcap \mathrm{supp}(\psi_{i'j'}) \neq \emptyset \Rightarrow \mathrm{supp}(\psi_{ij}) \bigcup \mathrm{supp}(\psi_{i'j'}) \subset U_i \bigcap U_{i'} $ for all $ i,j,i',j' =1,...,N$ as specified above.
{\color{black} We then specify the symbol class under consideration
\beqnx
\begin{split}
\widetilde{S}^m (T^*(\partial D) )  :=&  \bigg \{ a: T^*(\partial D)   \rightarrow \mathbb{C} \, ;  a = \sum_{i,j=1}^N \psi_{ij} F_i^* \left(  a_i \right) , \,  a_i \in \widetilde{S}^m \left( U_i \times \mathbb{R}^{d-1}  \right)  \bigg \}; \\
\widetilde{S}^m \left( U_i \times \mathbb{R}^{d-1}  \right)  := & \bigg\{ 
  a: U_i \times \left( \mathbb{R}^{d-1} \backslash \bigcup_{c \in \{1,Q\} } \{|\xi|^2 = c^2 \} \right) \rightarrow \mathbb{C}  \, ; \\
& \quad { \color{black} \exists M(a) \in \mathbb{N}  \text{ such that } } \\
& \quad a (x,\xi) = \sum \limits_{\scriptsize \begin{matrix} l_c,m_c = - M(a)  \end{matrix} }^{M(a)} r_{ l_1 m_1 l_Q m_Q }  (x,\xi) \prod_{c \in \{1,Q\}} \left( \sqrt{c^2 - |\xi|^2 } \right)^{l_c}  \left( \overline{ \sqrt{c^2 - |\xi|^2 } } \right)^{m_c}  , \\
& \quad  \text{with } r_{ l_1 m_1 l_Q m_Q } \in \mathcal{C}^{\infty} (  U_i \times  \mathbb{R}^{d-1}   ) \text{ and }  \\
&\quad  | \partial_\xi^\alpha \partial_x^{\beta} r_{ l_1 m_1 l_Q m_Q } (x,\xi) | \\
& \quad  \quad \leq C_{\alpha, \beta, l_1, m_1, l_Q, m_Q} (1+ |\xi|^2)^{ \frac{ m - |\alpha| \left(1 - \sum_{ c \in \{1,Q\}} ( l_c + m_c )  \right) }{2}  }  \,, \\
&\quad  \forall x \in U_i, \xi \in \mathbb{R}^{d-1}, \alpha,\beta \in \mathbb{N}^{d-1}   \bigg\}\,,
\end{split}
\eqnx
where $z \mapsto z^{\frac{1}{2}} $ is always taken over the branch with angle between $0$ and $\pi$. 
We remark that for any $a \in \widetilde{S}^m \left( U_i \times \mathbb{R}^{d-1}  \right) $, 
we have $a \in \mathcal{C}^{\infty} \left(  U_i \times \left( \mathbb{R}^{d-1} \backslash \bigcup_{c \in \{1,Q\} } \{|\xi|^2 = c^2 \} \right)  \right) $.  Moreover, for any $\delta>0$ and for all $x \in U_i, \xi \in \{ \xi \in \mathbb{R}^{d-1} : | |\xi| - c | \geq \delta, c \in \{1,Q\} \} , \alpha,\beta \in \mathbb{N}^{d-1}$, we have
\beqnx
 | \partial_\xi^\alpha \partial_x^{\beta} a (x,\xi) | \leq C_{\alpha, \beta} \delta^{- 2 M(a) - \beta } (1+ |\xi|^2)^{\frac{m - |\alpha| }{2}} \,.
\eqnx
}We are now ready to write the definition of $\mathrm{Op}_{a,h} $ over a compact $\partial D$ to be such that
{\color{black}
\begin{eqnarray}
\mathrm{Op}_{a,h} f  &:=&  \sum_{i=1}^N \bigg( \sum_{\mathrm{supp}(\psi_{ij}) \bigcap \mathrm{supp}(\psi_{i'j'})\neq \emptyset  }   \psi_{i'j'} F_i^* \left(  \mathrm{Op}_{ (F_i^{-1})^* \left( a |_{T^*(U_i)} \right) ,h} [ (F_i^{-1})^*(\psi_{ij} f ) ] \right) \bigg) \nonumber \\
& & + \int_{\partial D} K_h(\cdot,y) f(y) d \sigma(y) 
\label{eq:pp3}
\end{eqnarray}
{\color{black}
where $K_h \in C^{\infty}(\partial D \times \partial D ) $  with $ |  \partial_x^\alpha \partial^\beta_y K_h | < C_{n \alpha \beta} \, h^{n}  $ for all $ \alpha,\beta \in \mathbb{N}^{d-1},  n  < N(M(a)) $ for some $N(M(a))$ depending on $M(a) \in \mathbb{N}$.} }
}It is noted that $\mathrm{Op}_{a,h}$ is uniquely defined modulus $h \Phi \text{SO}_h^{m-1}$ if $a \in \tilde{\mathcal{S}}^m(T^*(\partial D))$.
Given an operator $\mathscr{A}_h \in \Phi \text{SO}^{m}_h$, we also denote the symbol (under a given coordinate system):
\begin{equation}\label{eq:pp1}
p_{\mathscr{A}_h} (x,\xi) := a(x,\xi)\ \text{ if }\ \mathscr{A}_h = \mathrm{Op}_{a,h}.
\end{equation}
Notice that the principal symbol $p_{\mathscr{A}_h} (x,\xi) (\text{mod} \, h \tilde{\mathcal{S}}^{m-1} (T^*(\partial D)))$ is well-defined and is independent of the choice of the coordinate system. In \eqref{eq:pp1}, we define the operator via the right/Kohn-Nirenberg quantisation. It is remarked that one can also use the left or Weyl quantisation, but this will not affect our subsequent analysis since the principal symbol of the operator is independent of the choice of the quantisation (cf. \cite{ACLS}), and that the left, Weyl and right quantisations differ only by an action of $\exp\left(\pm i \frac{h}{2} \langle \partial_x ,  \partial_\xi \rangle \right)$. We remark that Weyl quantisation helps give a definition of sub-principal symbol, as the sub-principal symbol of Weyl quantistion is well-defined and independent of the choice of coordinate.
For notational sake, we also recall the Hormander symbol class of order $m$, $S^m (T^*(\partial D))$, as
{\color{black}
\beqnx
S^m (T^*(\partial D)) & := & \left \{ a: T^*(\partial D) \rightarrow \mathbb{C} \, ; \, a = \sum_{i,j=1}^N \psi_{ij} F_i^* \left( a_i \right)  \in S^m (U_i \times \mathbb{R}^{d-1} ) \right \}, \\
S^m (U_i \times \mathbb{R}^{d-1} ) & := & \bigg\{ a: U_i \times \mathbb{R}^{d-1} \rightarrow \mathbb{C} \, ; a \in \mathcal{C}^{\infty} (  U_i \times  \mathbb{R}^{d-1}   ) \\
& &\quad | \partial_\xi^\alpha \partial_x^{\beta} a (x,\xi) | \leq C_{\alpha, \beta} (1+ |\xi|^2)^{\frac{m - |\alpha| }{2}} \quad \forall x \in U_i, \xi \in \mathbb{R}^{d-1}, \alpha,\beta \in \mathbb{N}^{d-1}  \bigg\} \,.
\eqnx
We recall that, in this case, we shall further require that $K_h \in C^{\infty}(\partial D \times \partial D ) $ in \eqref{eq:pp3} to be such that { \color{black} $ |  \partial_x^\alpha \partial^\beta_y K_h | < C_{n \alpha \beta} \, h^{n} (1+|x-y|^2)^{-\frac{n}{2}} $ } for all $ \alpha,\beta \in \mathbb{N}^{d-1}, n \in \mathbb{N}$. (c.f. \cite{Zworski} Theorem 9.6)
The readers may also see \cite{Zworski} Chapters 9 and 14 for further details.
}

\subsection{Layer potential operators as $\Phi\mathrm{SO}$'s}

In this subsection, we compute the principal symbols of $\mathcal{S}^{\kappa Q}_{\partial D} $ and $ {\mathcal{K}^{\kappa Q}_{\partial D}}^*$, and also derive some quantitative properties of these operators.

First, we briefly introduce the geometric description of $D\subset\mathbb{R}^d$, and we also refer to \cite{ACL1} for a similar treatment.
Let $\mathbb{X}(x): x\in U\subset\mathbb{R}^{d-1}\mapsto\partial D\subset\mathbb{R}^d$ be a regular parametrization of the surface $\partial D$. We often write the vector $\mathbb{X}_j := \f{\p \mathbb{X}}{\p x_j}$, $j=1,2,\ldots, d-1$. One has that the the normal vector is given as $\nu := \times_{j=1}^{d-1} \mathbb{X}_j / |\times_{j=1}^{d-1} \mathbb{X}_j | $. Let $\bar{\nabla}$ denote the standard covariant derivative on the ambient space $\mathbb{R}^{d}$, and $\Pi$ be the second fundamental form defined on the tangent space $T(\partial D)$. Next, we introduce the following matrix $A_{ij} (x), x \in \partial D $ , defined as
\beqn
A(x) := ( A_{ij}(x) ) =  \langle \textbf{II}_x (\mathbb{X}_i ,\mathbb{X}_j), \nu_x \rangle \notag, \ \ i,j\in\{1,\ldots, d-1\}. 
\eqn
Let $g=(g_{ij})$ be the induced metric tensor and $(g^{ij})=g^{-1}$. We write $H(x), x \in \partial D$ as the mean curvature satisfying
$$ \text{tr}_{g(x)} (A(x)) :=  \sum_{i,j = 1}^{d-1} g^{ij}(x) A_{ij}(x) := (d-1) H (x) \, . $$
We write $\nabla$ as the Levi-Civita connection on $\partial D$. Let $\Gamma_{ij}^k$ and $\bar{\Gamma}_{ij}^k$ be the Christoffel symbols such that
\[
 \nabla_{\mathbb{X}_{i}} \mathbb{X}_{j}  := \sum_{k=1}^{d-1}  \Gamma^k_{ij} \mathbb{X}_{k}, \quad \bar{\nabla}_{e_{i}} e_{j}  := \sum_{k=1}^{d}  \bar{\Gamma}^k_{ij}  e_k\quad \mbox{for\ \ $i,j = 1,..., d-1$},
\]
for the basis $\{ e_i \}_{i=1,...,d-1} = \{ \mathbb{X}_{i} \}_{i=1,...,d-1}$ and $e_d = \nu$. It is noted that for a fixed $x\in \partial D$, if one takes the geodesic normal coordinate in a neighborhood of $x$ to give $z \in \text{Dom}(\exp_x) \subset  T_x( \partial D) \cong \mathbb{R}^{d-1} \mapsto \mathbb{X}(z) : = \exp_x (z) \in  \partial D$, then at the point $x$, we have that $g_{ij}(x) = \boldsymbol{\delta}_{ij}$ and $\Gamma_{ij}^l(x) = 0$ for $i,j,l = 1,...,d-1.$ Here and also in what follows, $\boldsymbol{\delta}$ signifies the Kronecker delta. Moreover, if we choose on the ambient space the semi-geodesic normal coordinate in a neighborhood of $x$, which is given in a neighbourhood of $x \in \partial D$ as $(\alpha,s)$, $\tilde{\mathbb{X}}(\alpha,s) = \exp_{x}(\alpha) + \, s \, \nu(\exp_{x}(\alpha))$, we have $\bar{\Gamma}^l_{ij}(x) = \Gamma_{ij}^l(x) =0 $ for $l\neq d$ and {$\bar{\Gamma}^d_{ij}(x) = A_{ij}(x) $}.
Now, for $y = \exp_x( \delta \, \omega)$ with $\delta\in\mathbb{R}_+$ and $\omega\in\mathbb{S}^{d-1}$,  {\color{black} we have via a direct computation of the corresponding derivatives from definitions up to the respective order, together with the respective orthogonal relations and the aforementioned defining properties of semi-geodesic normal coordinate, that (cf. \cite{DoCarmo} Chapter 6,  \cite{Jost} Chapter 3, \cite{Lee} Chapter 8): }
{\small
\begin{equation}\label{eq:cc3}
\begin{split}
y = &  x +  \delta \omega - \frac{1}{2}   \delta^2 \langle A(x) \omega , \omega \rangle \,  \nu(x) - \frac{1}{6}   \delta^3 \left[ \langle A(x) \omega , \omega \rangle A(x) \omega + \langle \partial_{\omega} A(x) \omega, \omega \rangle v(x) \right]  \,  + \mathcal{O}(\delta^4 )  \,,  \\
\nu(y) =& \nu(x) +  \delta  A(x) \omega + \frac{1}{2} \delta^2 \left[ \partial_{\omega} A(x) \omega  
- |A(x) \omega |^2 \nu(x) \right]    + \mathcal{O}(\delta^3 ) \,,  \\
  g_{ij}(y)   =&  \delta_{ij} - \frac{1}{3} \delta^2 \, \sum_{k,l = 1}^{d-1} { \text{Rm}_{x} }_{iklj} \omega_k \omega_l  + \mathcal{O}(\delta^3) \\
\sqrt{ \det ( g(y)) }   =&  1 - \frac{1}{6} \delta^2 \, \text{Ric}_{x}(\omega,\omega)  + \mathcal{O}(\delta^3)  \\
= & 1 - \frac{1}{6} \delta^2 \, \left[ (d-1) H(x)  \langle A(x) \omega , \omega \rangle - | A(x) \omega |^2   \right]  + \mathcal{O}(\delta^3)  \,,
\end{split}
\end{equation}
}where $\mathrm{Ric}$ and $\mathrm{Rm}$ signify respectively the Ricci and Riemann curvature tensors, {\color{black} and the last equality comes from Gauss-Codazzi Theorem.}

Next, for the fundamental solution $G_\kappa$ introduced in \eqref{fundamental2}, by using the analytic properties of the Hankel function {\color{black} (cf. \cite{AandS} P.358, P.360 and \cite{digital} Section 10.8.   See also \cite{Kor02}) i.e. for $n \in \mathbb{N}$ and $z \in \mathbb{C}$,
\begin{eqnarray*}
J_{\alpha} (z) &=& \sum_{k=0}^\infty \frac{(-1)^k}{k! \Gamma(k + \alpha +1) }  \left(\frac{z}{2} \right)^{2k + \alpha} \text{ for } \alpha = n \text{ or } n + \frac{1}{2} \\
Y_n (z) &=& - \pi^{-1} \left(\frac{z}{2} \right)^{-n} \sum_{k=0}^{n-1} \frac{(n-k-1)!}{k!} \left(\frac{z}{2} \right)^{2k} + \frac{2}{\pi} J_n (z) \log\left( \frac{z}{2}  \right) \\
& & - \pi^{-1} \left(\frac{z}{2} \right)^{n} \sum_{k=0}^{\infty} \frac{ [ \partial_t \log(\Gamma(t)) ]_{t =k+1} + [ \partial_t \log(\Gamma(t)) ]_{t =n+k+1}  }{k! (n+k)!} (-1)^k \left(\frac{z}{2} \right)^{2k} \\
Y_{n + \frac{1}{2}} (z)&=& (- 1)^{n-1} J_{- n - \frac{1}{2}} (z)  \\
H_\alpha^{1}(z) &=& J_\alpha(z) + i Y_\alpha(z) \text{ for } \alpha = n \text{ or } n + \frac{1}{2} 
\end{eqnarray*}
where $\Gamma(\cdot)$ is the Gamma function,} one can have by direct calculations that
{\color{black}
\begin{equation}\label{eq:cc1}
G_\kappa (x-y) =
\begin{cases}
 \kappa^{ d-2}  \bigg[ C_{d}   \kappa^{2-d}  |x-y|^{2-d}  + \tilde{C}_{d}  \kappa^{4 -d} |x-y|^{4 -d} + \tilde{\tilde{C}}_{0,d}  +  \tilde{\tilde{C}}_d \kappa^{6 -d} |x-y|^{6 -d} \\
\hspace*{3cm} +  \mathcal{O}\left( \kappa^{7 -d} |x-y|^{7 -d} \right) \bigg]\  \text{when}\ d > 2, d \neq 3,4,6,  \, ;\\
 \kappa  \bigg[ C_{3}   \kappa^{-1}  |x-y|^{-1} + \tilde{C}_{0,3} + \tilde{C}_{3}  \kappa  |x-y|  +  \tilde{\tilde{C}}_d \kappa^{2} |x-y|^{2} \\
\hspace*{3cm} +  \mathcal{O}\left( \kappa^{3} |x-y|^{3} \right) \bigg]\  \text{when}\ d = 3 \, ;\\
 \kappa^2  \bigg[ C_{4}   \kappa^{-2}  |x-y|^{-2}  + \tilde{C}_{4} \, \log( \kappa |x-y| ) + \tilde{\tilde{C}}_4 \kappa^{2}  |x-y|^{2}\\
\hspace*{3cm} +  \mathcal{O}\left( \kappa^{4} |x-y|^{4} \right) \bigg]\   \text{when}\ d = 4 \, ; \\
\kappa^4  \bigg[ C_{6}   \kappa^{-4}  |x-y|^{-4}  + \tilde{C}_{6} \kappa^{-2}  |x-y|^{-2}  \, + \tilde{\tilde{C}}_6  \log( k |x-y| )\\
\hspace*{3cm}  +  \mathcal{O}\left( \kappa^{2} |x-y|^{2} \right) \bigg]\  \text{when}\ d =6 \, ;
\end{cases}
\end{equation}
}
and when $d>2$:
\begin{equation}\label{eq:cc2}
\begin{split}
 &  \partial_{\nu_x} G_\kappa (x-y) =
 \kappa^{d-2}  \bigg[ C_{d} (2-d)  \,  \kappa^{2-d}  \langle \nu_x, x -y \rangle  |x-y|^{-d}  \\
&  \quad + \tilde{C}_{d} (4-d) \kappa^{4 -d}
 \langle \nu_x, x -y \rangle  |x-y|^{2-d}  +  \mathcal{O}\left( \kappa^{5 -d} |x-y|^{4 -d} \right) \bigg] \,,
\end{split}
\end{equation}
where and also in what follows $C_d, \tilde{C}_{0,d} \tilde{C}_d , \tilde{\tilde{C}}_{0,d}, \tilde{\tilde{C}}_d$ signify some dimensional constants and $\tilde{\tilde{C}}_{0,d} = 0$ when $d \neq 5$

Using \eqref{eq:cc1}, \eqref{eq:cc2} and \eqref{eq:cc3}, together with direct though tedious calculations, 
{\color{black} e.g.
\begin{eqnarray} \label{eq:cc4}
| y - x|^2 &=&  \delta^2  - \frac{1}{12}   \delta^4 \langle A(x) \omega , \omega \rangle ^2 \,  + \mathcal{O}(\delta^5 )  \nonumber \\
| y - x|^l  &=& \left( | y - x|^2 \right)^{\frac{l}{2}}  =  \delta^l  \left( 1 - \frac{l}{24}   \delta^2 \langle A(x) \omega , \omega \rangle ^2 \,  + \mathcal{O}(\delta^3 )   \right) \text{ for } l \neq 0
\end{eqnarray}
}one has when $d>2$:
{\color{black}
\begin{eqnarray}\label{eq:ddd1}
& &G_{\kappa Q}(x-y) \, d \sigma(y)\nonumber\\
&=& (\kappa Q)^{d-2} \bigg[ G_1\left( Q \kappa \delta (0,\omega) \right) + C_{d}   \kappa^{-2} Q^{2-d}  (\kappa\delta)^{4-d} \bigg(  \frac{d-2}{24}      \langle A(x) \omega , \omega \rangle^2 \label{eq:ddd1}\\
& &\quad  -  \frac{1}{6}  \, (d-1) H(x)  \langle A(x) \omega , \omega \rangle + \frac{1}{6} | A(x) \omega |^2 \bigg) + \mathcal{O}( \kappa^{-2} (\kappa \delta)^{5-d} )  \bigg]\,  dy; \nonumber
\end{eqnarray}
where  $G_1(\cdot)$ is $G_\kappa(\cdot)$ with $\kappa=1$, i.e. for small $\delta>0$:
\begin{equation}\label{eq:ddd2}
G_1\left( Q \kappa \delta (0,\omega) \right)  =
\begin{cases}
C_{d}  Q^{2-d}  (\kappa \delta)^{2-d}  + \tilde{C}_{d}  Q^{4 -d} (\kappa \delta)^{4-d}  +  \tilde{\tilde{C}}_{0,d}  + \tilde{\tilde{C}}_d  Q^{6-d} (\kappa \delta)^{6-d} + \mathcal{O}( (\kappa \delta)^{7-d} )  \\
\hspace*{3cm} \  \text{when}\ d > 2, d \neq 3,4,6,  \, ;\\
 C_{3}  Q^{-1}  (\kappa \delta)^{-1}  +  \tilde{C}_{0,3} + \tilde{C}_{3}  Q (\kappa \delta)   + \tilde{\tilde{C}}_3  (\kappa \delta)^{2} + \mathcal{O}( (\kappa \delta)^{3} ) \\
\hspace*{3cm}  \  \text{when}\ d = 3 \, ;\\
 C_{4}  Q^{-2}  (\kappa \delta)^{-2}  + \tilde{C}_{4}    \log( \kappa Q \delta )  +  \tilde{\tilde{C}}_4  Q^2 (\kappa \delta)^2  + \mathcal{O}( (\kappa \delta)^2 ) \\
\hspace*{3cm} \  \text{when}\ d = 4 \, ; \\
C_{6}  Q^{-4}  (\kappa \delta)^{-4}  + \tilde{C}_{6}  Q^{-2} (\kappa \delta)^{-2} + \tilde{\tilde{C}}_6  \log( \kappa Q \delta )  + \mathcal{O}( (\kappa \delta)^4 ) \\
\hspace*{3cm} \  \text{when}\ d =6 \, ;
\end{cases}
\end{equation}
}as well as that:
\begin{eqnarray}
&  & \partial_{\nu_x} G_{\kappa Q}(x-y) \, d \sigma(y)\nonumber \\
&=&
 \frac{ (\kappa Q)^{d-2} }{2} \bigg[ 
\, \big \langle z' \, , \, A(x) \partial_{z'} G_1 \big( (0,z') \big) \big \rangle  \bigg|_{ z' = Q \kappa \delta (0,\omega)} \nonumber \\
&  & +  \frac{1}{3} C_{d}  (2-d)  \kappa^{-1} Q^{2-d}  (\kappa\delta)^{3-d}  \langle \partial_{\omega} A(x) \omega, \omega \rangle     + \mathcal{O}( \kappa^{-1} (\kappa \delta)^{4-d} )  \bigg]\,  dy.   \label{eq:ddd3}
\end{eqnarray}
with $z' \in \mathbb{R}^{d-1}$, where
\begin{eqnarray}
&  &  \big \langle z' \, , \, A(x) \partial_{z'} G_1 \big( (0,z') \big) \big \rangle  \bigg|_{ z' = Q \kappa \delta (0,\omega)} \nonumber \\
&=&
C_{d}  (2-d) Q^{2-d}  (\kappa \delta)^{2-d}  \langle A(x) \omega , \omega \rangle + \tilde{C}_{d}  (4-d) Q^{4 -d}  (\kappa \delta)^{4-d}   \langle A(x) \omega , \omega \rangle + \mathcal{O}( (\kappa \delta)^{5-d} )  \,  \nonumber \\
& &  \label{eq:ddd4}
\end{eqnarray}

{\color{black} Next, we shall make frequent use of the Fourier transform \eqref{eq:ff1} with respect to $\mathfrak{q}:=\kappa\delta\omega=h^{-1}\delta \omega$. For clarity, we shall always write $\mathcal{F}_{\mathfrak{q}}$ to signify the Fourier transform in such a case.} Then using the results derived in \eqref{eq:ddd1}--\eqref{eq:ddd4}, one can obtain by direct though tedious calculations the following principal symbols around $x=y$ in the geodesic normal coordinate ({\color{black} here, we also refer to \cite{weyl2,weyl1,ACL1,ACL2} for related results in the literature and Example 1 in what follows for a concrete example}), when $d>2$:
{\color{black}
\begin{eqnarray}
&& p_{  \kappa  \mathcal{S}^{\kappa Q}_{\partial D}}(x,\xi)  \nonumber\\
&=& Q^{ d-2} \bigg[ 
 \mathcal{F}_{\mathfrak{q}} \left[ G_1 \left( Q \left( 0, \mathfrak{q} \right) \right)  \right] (\xi)  + C_{d}   h^{2} Q^{2-d} \mathcal{F}_{\mathfrak{q}} \bigg[  \frac{2-d}{24}   |\mathfrak{q}|^{-d}    \langle A(x) \mathfrak{q} , \mathfrak{q} \rangle^2\label{eq:eee1} \\
&& -  \frac{1}{6}  \, (d-1) H(x)  |\mathfrak{q}|^{2-d}    \langle A(x) \mathfrak{q} , \mathfrak{q} \rangle +  \frac{1}{6} |\mathfrak{q}|^{2-d}  | A(x) \mathfrak{q} |^2  \bigg] (\xi)  + \mathcal{O}( h^2 | \xi |^{-4} )  \bigg];\nonumber
\end{eqnarray}
where $(0, \mathfrak{q})\in\mathbb{R}^d$ with $\mathfrak{q} \in \mathbb{R}^{d-1}$, and:
\begin{eqnarray}
 &  & p_{   \kappa {\mathcal{K}^{\kappa Q}_{\partial D}}^* }(x,\xi)\nonumber   \\
& =  &
\frac{ Q^{d-2} }{2} \bigg[ 
 [\mathcal{F}_{\mathfrak{q}}   \big \langle Q \mathfrak{q} \, , \, A(x) \partial_{Q \mathfrak{q}} G_1 \big( (0, Q \mathfrak{q} ) \big) \big \rangle  ] ( \xi )  \nonumber \\
&  &  + \frac{1}{3} C_{d}  (2-d)  h Q^{2-d} \mathcal{F}_{\mathfrak{q}} \bigg[  |\mathfrak{q}|^{-d}    \langle \partial_{ \mathfrak{q} } A(x) \mathfrak{q} , \mathfrak{q} \rangle  \bigg](\xi)  + \mathcal{O}( h | \xi |^{-3} )  \bigg] \label{eq:eee4}  \,.
\end{eqnarray}
}

Finally, by letting $\mathfrak{q}=(\mathfrak{q}_j)_{j=1}^{d-1}$ and $\xi=(\xi_j)_{j=1}^{d-1}$, one can deduce the following relations by direct calculations {\color{black} (c.f. \cite{Hor1} Theorem 7.1.24, \cite{Galfand}, P.363-364):}
\begin{eqnarray}
&& \mathcal{F}_{\mathfrak{q}} \left[  |\mathfrak{q}|^{-d}  \right] (\xi)  = K_{0,d} |\xi |  \,,\ \   \mathcal{F}_{\mathfrak{q}} \left[  |\mathfrak{q}|^{2-d}  \right] (\xi)  = K_{1,d} |\xi |^{-1}  \,; \nonumber\\
&&\mathcal{F}_{\mathfrak{q}} \left[ \mathfrak{q}_i \mathfrak{q}_j |\mathfrak{q}|^{-d}  \right] (\xi)  = - K_{0,d} \partial_i \partial_j |\xi |  \,,  \ \ \
\mathcal{F}_{\mathfrak{q}} \left[ \mathfrak{q}_i \mathfrak{q}_j |\mathfrak{q}|^{2-d}  \right] (\xi)  = - K_{1,d} \partial_i \partial_j |\xi |^{-1}  \, ;\label{eq:ggg1} \\
&&\mathcal{F}_{\mathfrak{q}} \left[ \mathfrak{q}_i \mathfrak{q}_j \mathfrak{q}_k \mathfrak{q}_l |\mathfrak{q}|^{-d}  \right] (\xi)  = K_{0,d} \partial_i \partial_j \partial_k \partial_l |\xi |  \,,  \ \ \   \mathcal{F}_{\mathfrak{q}} \left[ \mathfrak{q}_i \mathfrak{q}_j \mathfrak{q}_k |\mathfrak{q}|^{-d}  \right] (\xi)  = i K_{0,d} \partial_i \partial_j \partial_k  |\xi |    \, ;\nonumber 
\end{eqnarray}
where $K_{j,d}$, $j=0,1$ are dimensional constants. 
{\color{black}
Next, we focus on the function $[\mathfrak{F} G_1 \big((0, z')\big)] (\xi )$, where $(0, z')\in\mathbb{R}^d$ with $z' \in \mathbb{R}^{d-1}$; and $\mathfrak{F}$ is standard Fourier transform, namely $\mathcal{F}_h$ in \eqref{eq:ff1} with $h=1$, but associated with $z'\in\mathbb{R}^{d-1}$.
In fact, by the definition that $G_1$ is the fundamental solution with outgoing radiation condition, we have
\begin{eqnarray*}
\left( \partial_{z_0 z_0}^2  +  \sum_{i=1}^{d-1} \partial_{z'_i z'_i}^2 + \left( 1 + 0^+ \mathrm{i} \right) \right) G_1 (z_0, z') = \delta_0( d z_0)   \delta_0( dz')
\end{eqnarray*}
where $ \delta_0( d z_0) ,  \delta_0( dz') $ are the dirac measure of the respective variables $z_0$ and $z'$ at $0$, and $ 0^+ $ is the regularizing limit in the resolvent $G_1$ (in the Fourier spectrum) that gives the outgoing radiation condition.  Therefore, via a Fourier transform, we directly have
\begin{eqnarray*}
\left( \partial_{z_0 z_0}^2  + \left( 1 - |\xi|^2 + 0^+ \mathrm{i} \right) \right) [\mathfrak{F} G_1 \big(( z_0, z')\big)] (\xi ) = \delta_0( d z_0)   
\end{eqnarray*}
which gives
\begin{eqnarray*}
[\mathfrak{F} G_1 \big(( z_0, z')\big)] (\xi )  = C_0 \frac{\mathrm{i} e^{ \mathrm{i} \sqrt{1 - |\xi|^2 } |z_0| }}{\sqrt{1 - |\xi|^2 }}
\end{eqnarray*}
for some constant $C_0$, where $z \mapsto z^{\frac{1}{2}} $ is always taken over the branch with angle between $0$ and $\pi$ (so that the above function decay as $| x_d | \rightarrow \infty$ when $|\xi| >  1$,) and therefore 
\begin{eqnarray*}
[\mathfrak{F} G_1 \big(( 0, z')\big)] (\xi )  =  \frac{ \mathrm{i} C_0 }{\sqrt{1 - |\xi|^2 }} \,.
\end{eqnarray*}
We now notice that 
\begin{eqnarray}
Q^{ d-2}  [\mathfrak{F} G_1 \big((0, Q z')\big)] ( \xi ) =  Q^{-1} [\mathfrak{F} G_1 \big((0, z')\big)] ( Q^{-1} \xi )  = \frac{\mathrm{i} C_0}{\sqrt{Q^2 - |\xi|^2 }}
\, ;\label{eq:ggg2}
\end{eqnarray}
which is now a well-defined tempered distribution in $\mathcal{S}'(\mathbb{R}^{d-1})$.  (In fact we recall the observation that the part inside $|\xi|<Q$ corresponds to the imaginary part of the fundamental solution which gives the outgoing radiation condition and is real analytic on $\mathbb{R}^{d-1}$, and the part outside $|\xi|>Q$ corresponds to the real part of the fundamental solution and its leading order at infinity provides the singularity type of the kernal at the origin.)
Now it is direct to see that whenever $|\xi|<Q$,  $ | [\mathfrak{F} G_1 \big((0, Q z')\big)] ( \xi ) | > C_0 Q^{-1}$, whereas
whenever $|\xi| > Q$, 
\begin{eqnarray}
Q^{ d-2}  [\mathfrak{F} G_1 \big((0, Q z')\big)] ( \xi )  \nonumber
&=& \frac{C_0}{|\xi| \sqrt{ 1- \frac{Q^2}{|\xi|^2}  }}  \nonumber \\
&=& C_0 \sum_{l=0}^\infty \frac{ \prod_{k=0}^{l-1} ( 2k + 1) }{ 2^l l! } Q^{2l}  |\xi|^{-2l-1}  \nonumber  \\
&=& \widehat{C}_{d}  |\xi |^{-1}  +  \widehat{\widehat{C}}_{d}  Q^{2} |\xi |^{-3} + \widehat{ \widehat{\widehat{C}} }_{d}  Q^{4} |\xi |^{-5}  + \mathcal{O}( | \xi |^{-7}  )
\, ;\label{eq:ggg3}
\end{eqnarray}
where $ \widehat{C}_{d}   ,  \widehat{\widehat{C}}_{d} ,  \widehat{ \widehat{\widehat{C}} }_{d}  $ are the same dimensional constants that we stated before, and the above series is uniform on $|\xi| \geq Q + \varepsilon$ for any $\varepsilon> 0$.  With this, we may also inspect that 
\begin{eqnarray*}
 [\mathfrak{F}  \big \langle z' \, , \, A(x) \partial_{z'} G_1 \big( (0,z') \big) \big \rangle  ] ( \xi ) &=& \sum_{i,j=1}^{d-1} A_{ij}(x)  \partial_j  \left( \frac{\mathrm{i} C_0  \xi_i}{\sqrt{1 - |\xi|^2 }} \right)  \\
& = & \mathrm{i} C_0    \left( \frac{ (d-1) H(x )}{\sqrt{1 - |\xi|^2 }}   +    \frac{  \langle  A(x) \xi, \xi \rangle }{\left( 1 - |\xi|^2 \right)^{\frac{3}{2}} } \right), 
\end{eqnarray*}
where $z \mapsto z^{\frac{3}{2}} $ is always taken over the branch with angle between $0$ and $3\pi$, and hence
\begin{eqnarray}
 Q^{ d-2}  [\mathfrak{F}  \big \langle Q z' \, , \, A(x) \partial_{Q z'} G_1 \big( (0, Q z') \big) \big \rangle  ] ( \xi ) & = & Q^{-1} [\mathfrak{F}  \big \langle z' \, , \, A(x) \partial_{z'} G_1 \big( (0, z') \big) \big \rangle  ] ( Q^{-1} \xi ) \nonumber \\
& = & \mathrm{i} C_0    \left( \frac{ (d-1) H(x)}{\sqrt{Q^2 - |\xi|^2 }}   +    \frac{  \langle  A(x) \xi, \xi \rangle }{\left( Q^2 - |\xi|^2 \right)^{\frac{3}{2}} } \right), 
\label{eq:ggg4}
\end{eqnarray}
and whenever $|\xi| > Q$, 
\begin{eqnarray}
&  & Q^{ d-2}  [\mathfrak{F}  \big \langle Q z' \, , \, A(x) \partial_{Q z'} G_1 \big( (0, Q z') \big) \big \rangle  ] ( \xi ) \nonumber \\
&=&  C_0   \sum_{i,j=1}^{d-1} A_{ij}(x)  \partial_i \partial_j \left( |\xi | \sqrt{ 1- \frac{Q^2}{|\xi|^2}  } \right) \nonumber \\
&=&  C_0   \sum_{l=0}^\infty \frac{ \prod_{k=0}^{l-2} ( 2k + 1)  }{ 2^l l! }  Q^{2l} \left(  \sum_{i,j=1}^{d-1} A_{ij}(x)  \partial_i \partial_j   |\xi|^{ 1 - 2l}  \right) \nonumber \\
& = & \widehat{C}_{d}  \sum_{i,j=1}^{d-1} A_{ij}(x) \partial_i \partial_j |\xi |  +  \widehat{\widehat{C}}_{d} Q^{2} \sum_{i,j=1}^{d-1}  A_{ij} (x)\partial_i \partial_j |\xi |^{-1} + \mathcal{O}( | \xi |^{-5}  )  ,
\label{eq:ggg5}
\end{eqnarray}
where  $ \widehat{C}_{d}   ,  \widehat{\widehat{C}}_{d} ,  \widehat{ \widehat{\widehat{C}} }_{d}  $ are (coincidentally) the same dimensional constants that we stated before.

By combining \eqref{eq:eee1}--\eqref{eq:ggg5}, we obtain:
\begin{eqnarray}
 &&p_{   \kappa \mathcal{S}^{\kappa Q}_{\partial D}}(x,\xi)  \nonumber \\
& =  &
Q^{-1} [\mathfrak{F} G_1 \big((0, z')\big)] ( Q^{-1} \xi )  \nonumber\\
& & +  C_d h^{2}  \bigg(  \frac{2-d}{24} K_{0,d} \sum_{i,j,k,l=1}^{d-1} A_{ij}(x) A_{kl}(x) \partial_i \partial_j \partial_k \partial_l |\xi | \label{eq:hhh1}\\
& &  + \frac{1}{6} K_{1,d}  \, (d-1) H(x) \sum_{i,j=1}^{d-1} A_{ij}(x) \partial_i \partial_j |\xi |^{-1}  + \frac{1}{6} K_{1,d} \, \sum_{i,j,k}^{d-1} A_{ik}(x) A_{kj}(x) \partial_i \partial_j |\xi |^{-1} \bigg)  \nonumber \\
&  & \quad + \mathcal{O}( h^{2} | \xi |^{-4} ) ,\nonumber
\end{eqnarray}
where $Q^{-1} [\mathfrak{F} G_1 \big((0, z')\big)] ( Q^{-1} \xi )  $ satisfies \eqref{eq:ggg2}-\eqref{eq:ggg3} and
\begin{eqnarray}
&& p_{   \kappa  {\mathcal{K}^{\kappa Q}_{\partial D}}^* }(x,\xi) \nonumber\\
& =  &
\frac{ 1}{2} \bigg[  Q^{-1} [\mathfrak{F}  \big \langle z' \, , \, A(x) \partial_{z'} G_1 \big( (0, z') \big) \big \rangle  ] ( Q^{-1} \xi  )
\nonumber \\
&  &   - \mathrm{i} C_d K_{0,d} (2-d)  h  
\sum_{i,j,k =1}^{d-1} \partial_k A_{ij}(x) \partial_i \partial_j \partial_k  |\xi |    + \mathcal{O}( h | \xi |^{-3} )  \bigg]\,,\label{eq:hhh2} 
\end{eqnarray}
where $Q^{-1} [\mathfrak{F}  \big \langle z' \, , \, A(x) \partial_{z'} G_1 \big( (0, z') \big) \big \rangle  ] ( Q^{-1} \xi ) $ satisfies \eqref{eq:ggg4}-\eqref{eq:ggg5}.
{\color{black}
Moreover careful tracing the series expansions from \eqref{eq:cc3}-\eqref{eq:cc4} that arrives to \eqref{eq:hhh1}-\eqref{eq:hhh2}, we obtain that the terms in \eqref{eq:hhh1}-\eqref{eq:hhh2}  only contain singularities induced from $Q^{-1} [ \mathfrak{F} \left( (z')^\alpha \partial_{z'}^{\beta} G_1 \big((0, z') \big) \right) ] ( Q^{-1} \xi) $ for some $\alpha,\beta \in \mathbb{N}^{d-1}$ where
\[
 [ \mathfrak{F} \left( (z')^\alpha \partial_{z'}^{\beta} G_1 \big((0, z') \big) \right) ] ( \xi) = \mathrm{i}^{\alpha - \beta}
  \partial_{\xi}^\alpha  \left( \xi^{\beta} [ \mathfrak{F} G_1  \big((0, z') \big) ] (\xi) \right) 
= \mathrm{i}^{\alpha - \beta +1} C_0 \partial_{\xi}^\alpha  \left( \xi^{\beta} \left( 1 - |\xi|^2 \right)^{-\frac{1}{2}}  \right) 
\]
which contains only singularity when $|\xi|^2 = 1$ or $Q^2$, and hence in $\widetilde{S}^m (T^*(\partial D) ) $ for some $m \in \mathbb{Z}$.
}
}

In summarizing the derivation and discussion so far in this subsection, we present the following lemma.

\begin{Lemma}
\label{lemma_SK}
The operators $ \mathcal{S}^{\kappa Q}_{\partial D}$ and $ {\mathcal{K}^{\kappa Q}_{\partial D}}^* $ can be decomposed as follows:
\begin{equation}\label{eq:dec1}
 \mathcal{S}^{\kappa Q}_{\partial D} = h \left(  \mathcal{S}^{\kappa Q}_{-1} + h^2  \mathcal{S}^{\kappa Q}_{-3}  \right) \quad\mbox{and}\quad {\mathcal{K}^{\kappa Q}_{\partial D}}^* = \frac{ h }{2} \left( \mathcal{K}^{\kappa Q}_{-1} + h \mathcal{K}^{\kappa Q}_{-2}  \right) \,,
\end{equation}
where
\beqnx
 \mathcal{S}^{\kappa Q}_{-1} \in \Phi\mathrm{SO}^{-1}_h \,,\,   \mathcal{S}^{\kappa Q}_{-3} \in \Phi\mathrm{SO}^{-3}_h \quad\mbox{and}\quad
\mathcal{K}^{\kappa Q}_{-1} \in \Phi\mathrm{SO}^{-1}_h \,,\,  \mathcal{K}^{\kappa Q}_{-2} \in \Phi\mathrm{SO}^{-2}_h \,,\,
\eqnx
and in the geodesic normal coordinate, we have { \color{black}
\begin{eqnarray}
 p_{ \mathcal{S}^{\kappa Q}_{-1}}(x,\xi) 
& = &  \frac{\mathrm{i} C_0}{\sqrt{Q^2 - |\xi|^2 }} , \label{eq:dec2}\\
 p_{ \mathcal{K}^{\kappa Q}_{-1}}(x,\xi) 
 &=&  \mathrm{i} C_0    \left( \frac{ (d-1) H(x)}{\sqrt{Q^2 - |\xi|^2 }}   +    \frac{  \langle  A(x) \xi, \xi \rangle }{\left( Q^2 - |\xi|^2 \right)^{\frac{3}{2}} } \right)  
 \label{eq:dec3} 
\end{eqnarray}
where $z \mapsto z^{\frac{l}{2}} $ is always taken over the branch with angle between $0$ and $l\pi$ for $l = 1,3$, and whenever $|\xi| > Q$, we have
\begin{eqnarray}
p_{ \mathcal{S}^{\kappa Q}_{-1}}(x,\xi)  
&  = &  \widehat{C}_{d}  |\xi |^{-1}  + \widehat{\widehat{C}}_{d} Q^{2} |\xi |^{-3} + \widehat{ \widehat{\widehat{C}} }_{d}  Q^{4} |\xi |^{-5} + \mathcal{O}( | \xi |^{-7}  ),  \label{eq:dec4}  \\
p_{ \mathcal{S}^{\kappa Q}_{-3}}(x,\xi) 
 &= & \frac{2-d}{24} C_d K_{0,d} \sum_{i,j,k,l=1}^{d-1} A_{ij}(x) A_{kl}(x) \partial_i \partial_j \partial_k \partial_l |\xi | + \frac{1}{6} C_d K_{1,d}  \, (d-1) H(x)\times \nonumber \\
&&  \sum_{i,j=1}^{d-1} A_{ij}(x) \partial_i \partial_j |\xi |^{-1}  + \frac{1}{6} C_d K_{1,d}  \, \sum_{i,j,k}^{d-1} A_{ik}(x) A_{kj}(x) \partial_i \partial_j |\xi |^{-1}+ \mathcal{O}( | \xi |^{-4} ),  \nonumber \\
& & \label{eq:dec5}  \\
p_{ \mathcal{K}^{\kappa Q}_{-1}}(x,\xi)  &=& \widehat{C}_{d}  \sum_{i,j=1}^{d-1} A_{ij}(x) \partial_i \partial_j |\xi |  +  \widehat{\widehat{C}}_{d}  Q^{2} \sum_{i,j=1}^{d-1}  A_{ij} (x)\partial_i \partial_j |\xi |^{-1} + \mathcal{O}( | \xi |^{-5}  )   \label{eq:dec6}  \\
 p_{ \mathcal{K}^{\kappa Q}_{-2}}(x,\xi)  
 &=&   - \mathrm{i} C_d K_{0,d} (2-d)  
\sum_{i,j,k =1}^{d-1} \partial_k A_{ij}(x) \partial_i \partial_j \partial_k  |\xi |    + \mathcal{O}( | \xi |^{-3} ) ,    \label{eq:dec7}  
\end{eqnarray}
}and $ C_0, C_{d}, \widehat{C}_{d}   ,  \widehat{\widehat{C}}_{d} ,  \widehat{ \widehat{\widehat{C}} }_{d}  $ and $K_{j,d}, j = 0,1$ are dimensional constants.

It is also noted that $ \mathcal{S}^{\kappa Q}_{-1}$ and $ \frac{1}{2} \mathcal{K}^{\kappa Q}_{-1}$ are the principal parts of the operators $  \kappa \mathcal{S}^{\kappa Q}_{\partial D}$ and $ \kappa {\mathcal{K}^{\kappa Q}_{\partial D}}^* $, respectively.
\end{Lemma}

{\color{black}
\begin{Remark}
{\color{black}
We would like to remark that the singularity type $\left( Q^2 - \| \xi \|^2_{g(x)} \right)^{-\frac{1}{2}}$ of $p_{ \mathcal{S}^{\kappa Q}_{-1}}(x,\xi)  $ in \eqref{eq:ggg2} at $ \| \xi \|_{g(x)}  = Q$  (for any $Q>0$) under a general coordinate corresponds to the following decay order of the kernal of $ \mathcal{S}^{\kappa Q}_{\partial D}$ that, 
as $|x-y| \rightarrow \infty$,
\[
  \left(  \kappa^{-1} \partial_y \right)^{\alpha}  \left(  \kappa^{-1} \partial_y \right)^{\beta}  G_{\kappa Q} (x-y)   \sim C_d \mathrm{i} ( \kappa Q )^{ d-2} e^{\mathrm{i} \kappa Q |x-y| } ( \kappa  Q |x-y| )^{- \frac{d-1}{2} - |\alpha| - |\beta|}  \quad \forall \alpha,\beta \in \mathbb{N}^{d-1} \,,
\]
which comes readily from the fact that for all $m \in \mathbb{N}$, we have (cf. \cite{AandS}, P. 361) 
\[
 \left(  r^{-1} \partial_r \right)^{m} [ r^{\frac{d-2}{2}}  H^{(1)}_{\frac{d-2}{2}}(r) ]  =  r^{\frac{d-2}{2} - m}  H^{(1)}_{\frac{d-2}{2} - m }(r) \,,
\]
and as $r \rightarrow \infty$, we have (cf. \cite{AandS}, P. 364) 
\[
H^{(1)}_{\frac{d-2}{2} - m }(r)   \sim C_d r^{-\frac{1}{2}} e^{\mathrm{i} r} \,.
\]
}{\color{black}
The readers may look into \cite{Galfand}, P.363 and P.365 for the Fourier transform of $\left(  c^2 - \| \xi \|^2_{g(x)}  \right)^{\frac{l_c}{2}}$ for $l_c \in \mathbb{Z} \backslash 2 \mathbb{N}$, $c = 1,Q$ in terms of Hankel functions and powers of the norm, 
as well as Gegenbauer's addition formula in \cite{Watson}, P. 365 
and the generating function (and reproducting kernal property) of Gegenbauer polynomials in \cite{stein}, Chapter IV, (also c.f. \cite{ACL1}, P.6 and P.35)
to obtain the Fourier transform of a product of $\left(  c^2 - \| \xi \|^2_{g(x)}  \right)^{\frac{l_c}{2}}$ and their conjugates via repeated convolution.
}

In fact, the above phenomenon reflects the fact that \eqref{eq:ggg2} indicates that the propagation of singularities of $ \mathcal{S}^{\kappa Q}_{\partial D}$ is described by (c.f. \cite{Tacy} (1.9) )
\[
\text{WF}_h( \mathcal{S}^{\kappa Q}_{\partial D}  u )  \backslash \text{WF}_h( u )  \subset \{ (x,\xi) \in T^*(\partial D) : \| \xi \|_{g(x)}^2 = Q^2 \} ,
\]
where $\text{WF}_h$ is the semi-classical wavefront set.  A similar phenomenon can be observed over ${\mathcal{K}^{\kappa Q}_{\partial D}}^* $; see \cite{Tacy} for more refined $L^2(\partial D)$ estimates of $ \mathcal{S}^{\kappa Q}_{\partial D}$ and ${\mathcal{K}^{\kappa Q}_{\partial D}}^* $ when $\partial D$ is non-smooth.

Meanwhile, when $\kappa$ is fixed, treating $\kappa^2 Q^2$ as a constant and taking $h = 1$, a similar calculation as in the above Lemma would lead us to the observation that $ p_{ \mathcal{S}^{\kappa Q}_{\partial D} }(x,\xi)  = \|\xi\|^{-1}_{g(x)} + \mathcal{O}( \|\xi\|^{-2}_{g(x)}  ) $ and hence $ \text{WF}( \mathcal{S}^{\kappa Q}_{\partial D}  u ) \subset \text{WF}( u )  $. Therefore $\mathcal{S}^{\kappa Q}_{\partial D}$ is bounded from $L^{2}(\partial D)$ to $H^{1}(\partial D)$ for any fixed $\kappa > 0$.  
This is also reflected by the fact that for a compact $\partial D$, with a fixed $\kappa > 0$, the decay order of $G_{\kappa Q} (x-y)$ (at infinity) would not be observed on $\partial D$.

\end{Remark}
}

{\color{black}
\noindent \textbf{Example 1.}  
Let us consider 
 $D \subset \mathbb{R}^3$ such that 
\[
\left\{(x,y,z) \in \mathbb{R}^3 : (x,y) \in B_1(0) \,, \,  z = \frac{x^2+y^2}{2} \right\} \subset \partial D.
\]
Consider
{
\begin{eqnarray*}
R : [0,1] \rightarrow [0,\infty) \quad r \mapsto \int_0^r \sqrt{1+s^2} ds \quad \text{and} \quad
r : B_1(0) \rightarrow [0,\infty) \quad (x_1,x_2) \mapsto \sqrt{x_1^2 + x_2^2 }. \\
\end{eqnarray*}
}Then there exists $r_0 >0$ such that the following coordinate patch
\[
\Phi  : B_{r_0} (0) \rightarrow \partial D \quad  (x_1, x_2 ) \mapsto \left( \frac{x_1R^{-1} ( r(x_1,x_2)  ) }{r(x_1,x_2) },  \frac{x_2R^{-1} ( r(x_1,x_2)  ) }{r(x_1,x_2) } , \frac{R^{-1} ( r(x_1,x_2) ) }{2}  \right)
\]
is in fact the expoential map $\exp_0 (\cdot )$  at $0 \in \partial D$.  We also notice that
{
\begin{eqnarray*} 
\partial_s R^{-1} (s) = \frac{1}{\left( 1+ x^2 \right)^{\frac{1}{2}}} \bigg|_{ R^{-1} (s) } &\,, \,& \partial^2_s R^{-1} (s) = \frac{x}{\left( 1+ x^2 \right)^2 }\bigg|_{R^{-1} (s) }  \,,\,\\
 \partial_s^3 R^{-1} (s) = \frac{ 3 x^2 -1 }{\left( 1+ x^2  \right)^{\frac{7}{2}} }\bigg|_{R^{-1} (s) } &\,,\, &
\partial_s^4 R^{-1} (s) = \frac{ \left( 13 - 15 x^2  \right) x }{\left( 1+ x^2  \right)^{5} }\bigg|_{R^{-1} (s) }  \,,\,\\
 \partial_s^5 R^{-1} (s) = \frac{ 105 x^4  - 162 x^2 +13  }{\left( 1+ x^2  \right)^{ \frac{13}{2}} } \bigg|_{R^{-1} (s) } & \,,\, &
\partial_s^6 R^{-1} (s) = \frac{ \left( -945 x^4 + 2202 x^2 - 493  \right) x }{\left( 1+ x^2  \right)^{5} }\bigg|_{R^{-1} (s) } \,,\,
\end{eqnarray*}
}and when $s >0 $ is sufficiently small, in a uniform manner
{
\begin{eqnarray*}
R^{-1} (s) = s - \frac{1}{6} s^3 + \frac{13}{120} s^5 + O(s^7) & \,,\, &
\frac{ R^{-1} (s) }{s} = 1 - \frac{1}{6} s^2 + \frac{13}{120} s^4 + O(s^6) \,,\, \\
\partial_s \left( \frac{ R^{-1} (s) }{s}  \right) = - \frac{1}{3} s + \frac{13}{30} s^3 + O(s^5) & \,,\, &
\partial_s^2 \left( \frac{ R^{-1} (s) }{s}  \right) = - \frac{1}{3} + \frac{13}{10} s^2 + O(s^4) \,,\,
\end{eqnarray*}
}and hence $\partial_s^i \big( \frac{ R^{-1} (s) }{s} \big) $ exists in a neighborhood of $s=0$.  We are now ready to compute the first and second fundamental forms, the normal and volume as follows:
{
\begin{eqnarray*}
\partial_{x_1} \Phi (x_1,x_2) &=& (1,0,x_1) \frac{ R^{-1} (s) }{s}\bigg|_{r(x_1,x_2)} +  (x_1,x_2,0) \partial_s \left( \frac{ R^{-1} (s) }{s}  \right)  \bigg|_{r(x_1,x_2)},  \\
\partial_{x_2} \Phi (x_1,x_2) &=& (0,1,x_2) \frac{ R^{-1} (s) }{s}\bigg|_{r(x_1,x_2)} +  (x_1,x_2,0) \partial_s \left( \frac{ R^{-1} (s) }{s}  \right)  \bigg|_{r(x_1,x_2)} , \\
g_{11}  (x_1,x_2) &=&  \left( 1 + x_1^2 \right)  \left(  \frac{ R^{-1} (s) }{s} \right)^2  \bigg|_{r(x_1,x_2)} +  2 x_1  \frac{ R^{-1} (s) }{s}\bigg|_{r(x_1,x_2)}  \partial_s \left( \frac{ R^{-1} (s) }{s}  \right)  \bigg|_{r(x_1,x_2)}  \\
&  &  + \left( r(x_1,x_2) \right)^2  \left(  \partial_s \left( \frac{ R^{-1} (s) }{s}  \right)  \right)^2 \bigg|_{r(x_1,x_2)} , \\
g_{12}  (x_1,x_2) &=&  x_1 x_2 \left(  \frac{ R^{-1} (s) }{s} \right)^2  \bigg|_{r(x_1,x_2)} +  \left(  x_1 + x_2 \right)  \frac{ R^{-1} (s) }{s}\bigg|_{r(x_1,x_2)}  \partial_s \left( \frac{ R^{-1} (s) }{s}  \right)  \bigg|_{r(x_1,x_2)}  \\
&  &  + \left( r(x_1,x_2) \right)^2 \left(  \partial_s \left( \frac{ R^{-1} (s) }{s}  \right)  \right)^2 \bigg|_{r(x_1,x_2)} , \\
g_{22}  (x_1,x_2) &=&  \left( 1 + x_2^2 \right)  \left(  \frac{ R^{-1} (s) }{s} \right)^2  \bigg|_{r(x_1,x_2)} +  2 x_2  \frac{ R^{-1} (s) }{s}\bigg|_{r(x_1,x_2)}  \partial_s \left( \frac{ R^{-1} (s) }{s}  \right)  \bigg|_{r(x_1,x_2)}  \\
&  & + \left( r(x_1,x_2) \right)^2  \left(  \partial_s \left( \frac{ R^{-1} (s) }{s}  \right)  \right)^2 \bigg|_{r(x_1,x_2)} , \\
\sqrt{\det(g(x_1,x_2))} &=&    \left(  \frac{ R^{-1} (s) }{s} \right) \bigg|_{r(x_1,x_2)} \bigg(   \left( 1+ \left( r(x_1,x_2) \right)^2 \right) \left(  \frac{ R^{-1} (s) }{s} \right)^2  \bigg|_{r(x_1,x_2)}  \\
&  &  \quad \quad +2( x_1 + x_2 )  \frac{ R^{-1} (s) }{s}\bigg|_{r(x_1,x_2)}  \partial_s \left( \frac{ R^{-1} (s) }{s}  \right)  \bigg|_{r(x_1,x_2)}  \\
& &  \quad \quad + \left( r^4 + (1 - 2x_1x_2 ) r^2  + 2 x_1 x_2 \right)  \left( r(x_1,x_2) \right)^2  \left(  \partial_s \left( \frac{ R^{-1} (s) }{s}  \right)  \right)^2 \bigg|_{r(x_1,x_2)}   \bigg)^{\frac{1}{2}}, \\
\nu(x_1,x_2) &=& \frac{1}{\sqrt{\det(g(x_1,x_2))}} \bigg( (-x_1, -x_2, 1)  \frac{ R^{-1} (s) }{s}\bigg|_{r(x_1,x_2)}   \\
&  &   \quad \quad  - \left( (x_1-x_2) x_2, (x_2 - x_1) x_1, (x_1 + x_2 ) \right)  \partial_s \left( \frac{ R^{-1} (s) }{s}  \right)  \bigg|_{r(x_1,x_2)}  \bigg), \\
\partial_{x_1x_1}^2 \Phi (x_1,x_2) &=& (0,0,1) \frac{ R^{-1} (s) }{s}\bigg|_{r(x_1,x_2)} +  \left(1 + \frac{x_1}{r(x_1,x_2)} ,0 ,\frac{x_1^2}{r(x_1,x_2)} \right) \partial_s \left( \frac{ R^{-1} (s) }{s}  \right)  \bigg|_{r(x_1,x_2)}  \\
& &  + \frac{x_1}{r(x_1,x_2)}(x_1,x_2,0) \partial_s^2 \left( \frac{ R^{-1} (s) }{s}  \right)  \bigg|_{r(x_1,x_2)} , \\
\partial_{x_1 x_2}^2 \Phi (x_1,x_2) &=& \frac{1}{2} \left(1 + \frac{x_2}{r(x_1,x_2)} ,1 + \frac{x_1}{r(x_1,x_2)} ,\frac{x_1 x_2}{r(x_1,x_2)} \right) \partial_s \left( \frac{ R^{-1} (s) }{s}  \right)  \bigg|_{r(x_1,x_2)}  \\
&  &  + \frac{x_1 + x_2}{2 r(x_1,x_2)}(x_1,x_2,0) \partial_s^2 \left( \frac{ R^{-1} (s) }{s}  \right)  \bigg|_{r(x_1,x_2)},  \\
\partial_{x_2 x_2}^2 \Phi (x_1,x_2) &=& (0,0,1) \frac{ R^{-1} (s) }{s}\bigg|_{r(x_1,x_2)} +  \left( 0 , 1 + \frac{x_2}{r(x_1,x_2)} ,\frac{x_2^2}{r(x_1,x_2)} \right) \partial_s \left( \frac{ R^{-1} (s) }{s}  \right)  \bigg|_{r(x_1,x_2)}  \\
&  &  + \frac{x_2}{r(x_1,x_2)}(x_1,x_2,0) \partial_s^2 \left( \frac{ R^{-1} (s) }{s}  \right)  \bigg|_{r(x_1,x_2)} , \\
A_{11}  (x_1,x_2)  &= & \left(  \frac{ R^{-1} (s) }{s} \right)^2 \bigg|_{r(x_1,x_2)}    +  \left(r (x_1,x_2) x_1 + x_2\right) \frac{ R^{-1} (s) }{s}\bigg|_{r(x_1,x_2)}  \partial_s \left( \frac{ R^{-1} (s) }{s}  \right)  \bigg|_{r(x_1,x_2)}  \\
& &  + \left( (x_1 - x_2) x_2 (1 + \frac{x_1}{r (x_1,x_2) }) + \frac{x_1^2}{r (x_1,x_2)} (x_1 + x_2) \right)   \left(  \partial_s \left( \frac{ R^{-1} (s) }{s}  \right)  \right)^2 \bigg|_{r(x_1,x_2)} , \\
A_{12}  (x_1,x_2) &=&  - \frac{1}{2} \left( x_1+ x_2 + \frac{x_1 x_2}{r(x_1,x_2)}  \right) \frac{ R^{-1} (s) }{s}\bigg|_{r(x_1,x_2)}  \partial_s \left( \frac{ R^{-1} (s) }{s}  \right)  \bigg|_{r(x_1,x_2)}  \\
&  & + \frac{1}{2} \left( (x_1 -  x_2)^2 + \frac{x_1 x_2 (x_1 + x_2)}{r(x_1,x_2)} \right)   \left(  \partial_s \left( \frac{ R^{-1} (s) }{s}  \right)  \right)^2 \bigg|_{r(x_1,x_2)}  \\
&  & - \frac{r(x_1,x_2) (x_1 + x_2)}{2}  \frac{ R^{-1} (s) }{s}\bigg|_{r(x_1,x_2)}   \partial_s^2 \left( \frac{ R^{-1} (s) }{s}  \right)  \bigg|_{r(x_1,x_2)}  , \\
A_{22}  (x_1,x_2)  &= & \left(  \frac{ R^{-1} (s) }{s} \right)^2 \bigg|_{r(x_1,x_2)}    +  \left(r (x_1,x_2) x_2 + x_1\right) \frac{ R^{-1} (s) }{s}\bigg|_{r(x_1,x_2)}  \partial_s \left( \frac{ R^{-1} (s) }{s}  \right)  \bigg|_{r(x_1,x_2)} \\
& &  + \left( (x_2 - x_1) x_1 (1 + \frac{x_2}{r (x_1,x_2) }) + \frac{x_2^2}{r (x_1,x_2)} (x_1 + x_2) \right)  \left(  \partial_s \left( \frac{ R^{-1} (s) }{s}  \right)  \right)^2 \bigg|_{r(x_1,x_2)} . \\
\end{eqnarray*}
}We now notice the semi-geodesic normal coordinate patch is given by, for a small $\tilde{r}$:
\[
\tilde{ \Phi }  : B_{r_0} (0) \times (- \tilde{r}, \tilde{r} ) \rightarrow \partial D \quad  (x_1, x_2, x_3) \mapsto \Phi(x_1,x_2) + x_3 \nu(x_1,x_2).
\]
Hence considering
\[
G_{\kappa Q} (x-y) = \frac{e^{\mathrm{i} \kappa Q |x-y|}}{|x-y|}  = \frac{1}{4 \pi} \sum_{l=0}^\infty \frac{1}{l!} \kappa^l  Q^l \mathrm{i}^l |x-y|^{l-1},
\]
and letting $x = \Phi(0)$, $y = \Phi(\delta \omega) = \left(\omega_1 R^{-1}(\delta), \omega_2 R^{-1}(\delta) , \frac{\left( R^{-1}(\delta) \right)^2 }{2} \right)$, $|\omega| = 1$, then we have
{
\begin{eqnarray*}
g_{ij}(0) &=& A_{ij}(0) = \delta_{ij},  \\
y  &=& x   +   R^{-1}(\delta) \omega - \frac{\left( R^{-1}(\delta) \right)^2 }{2} \nu(0) 
= x+ \delta \omega  - \frac{1}{2} \delta^2 \nu(0) + \frac{1}{6} \delta^3  \omega  + O (\delta^4), \\
|y-x|^2 & =& \left( R^{-1}(\delta) \right)^2   \left( 1+   \frac{ \left( R^{-1}(\delta) \right)^2}{4}   \right)^2 \\
& =& \delta^2  \left( 1 - \frac{1}{6} \delta^2 + O(\delta^4) \right)^2   \left( 1 + \frac{1}{4} \delta^2 + O(\delta^4) \right)^2 \\
& = & \delta^2  \left( 1 - \frac{1}{12} \delta^2 + O(\delta^4) \right),  \\
|y-x|^{l} & =&  \left( |y-x|^2 \right)^{\frac{l}{2}}  =  \delta^l  \left( 1 - \frac{l}{24} \delta^2 + O(\delta^4) \right) , \\
\sqrt{ \det(g(y)) } & =& 1 - \frac{1}{6} \delta^2 + O(\delta^3) ,\\
G_{\kappa Q} (x-y) d \sigma(y) &=&  \frac{1}{4 \pi} \left(  \sum_{l=0}^\infty \frac{1}{l!} k^l Q^l \mathrm{i}^l  \delta^{l-1}  \left( 1 - \frac{l-1}{24} \delta^2 + O(\delta^4) \right)  \right) \left( 1 - \frac{1}{6} \delta^2 + O(\delta^3) \right) dy \\
&=&  \frac{1}{4 \pi} (\kappa Q )  \left(   \sum_{l=0}^\infty \frac{\mathrm{i}^l}{l!}   \left( Q \kappa \delta \right)^{l-1} + O \left(  \left( \kappa \delta \right)^{4} \right) + \kappa^{-2} Q^{-1} \left[ - \frac{1}{8} \left( \kappa \delta \right) +  O \left(  \left( \kappa \delta \right)^{2} \right) \right ] \right) dy  \\
&=&  \frac{1}{4 \pi} (\kappa Q )  \left(  G_1(  Q  \kappa \delta ) + \kappa^{-2} Q^{-1} \left[ - \frac{1}{8} \left( \kappa \delta \right) +  O \left(  \left( \kappa \delta \right)^{2} \right) \right ] \right) dy .
\end{eqnarray*}
}
Now with  \eqref{eq:ggg1}-\eqref{eq:ggg2}, we have
\begin{eqnarray*}
p_{   \kappa \mathcal{S}^{\kappa Q}_{\partial D}}(x,\xi) = \frac{\mathrm{i} C_0}{\sqrt{Q^2 - |\xi|^2 }} + h^2 p_{ \mathcal{S}^{\kappa Q}_{-3}}(x,\xi) ,
\end{eqnarray*}
and when $ | \xi |$ is large,
\begin{eqnarray*}
p_{ \mathcal{S}^{\kappa Q}_{-3}}(x,\xi) = - \frac{1}{8}  | \xi |^{-3}  + \mathcal{O}( | \xi |^{-4} )  \,, \nonumber
\end{eqnarray*}
which coincides with \eqref{eq:dec2} and \eqref{eq:dec5} considering the fact that $ g_{ij}(0) = A_{ij}(0) = \delta_{ij} $. 
}

\subsection{Layer potential operators associated with a surface in $D$ }\label{sect:2.3}

For the subsequent use, we shall also need to consider layer potential operators defined on surfaces in $D$, which correspond to the boundaries of certain interior domains. To that end, for a given $R\in\mathbb{R}_+$, let us consider a surface $\Gamma_R \subset D$ which has a global $C^\infty$ diffeomorphism $F_R: \partial D  \rightarrow \Gamma_R $ and satisfies $\mathrm{dist}(\Gamma_R, \partial D) \geq R $.  Consider a smooth vector field $X$ in $D \subset \mathbb{R}^d$, which may coincide with either a tangent vector field or a normal vector field $\nu$ when restricted on $\Gamma_R$.  Let us also write $X(f) = df(X) = \sum_{i=1}^d X_i(x) \partial_i f(x)$.
Then we consider the following operator $ \widetilde{\mathcal{S}^\kappa_{\partial D}}  : L^2(\partial D) \rightarrow L^2(\partial D)$ and $ \widetilde{ X \circ \mathcal{S}^\kappa_{\partial D}}  : L^2(\partial D) \rightarrow L^2(\partial D)$:
\begin{equation}\label{eq:kkk1}
\widetilde{\mathcal{S}^\kappa_{\partial D}} [\phi] (x) := \left( \mathcal{S}^\kappa_{\partial D} [\phi ] \right) (F_R (x)) \,,\, \quad \left(  \widetilde{X \circ \mathcal{S}^\kappa_{\partial D}} \right) [\phi] (x) := \left[ X \circ \left( \mathcal{S}^\kappa_{\partial D} [\phi ] \right) \right]  (F_R (x))
\end{equation}
for $x \in \partial D$ and $\nu_{F_R (x)}$ is a normal of $\Gamma_R$ at the point $F_R(x)$. For a simple illustration, if $D$ is strictly convex and contains the origin $\mathbf{0}\in\mathbb{R}^d$ with $\mathrm{dist}(\partial D, \mathbf{0}) > 0$, then for any $0< R < \mathrm{dist}(\partial D,\mathbf{0})$ we can choose $\Gamma_{R} = \left( 1 - \frac{R}{\mathrm{dist}(\partial D, \mathbf{0})} \right) \partial D $ and $F_R(x) = \left(1 - \frac{R}{\mathrm{dist}(\partial D, \mathbf{0})} \right) x$.

\begin{Lemma}
\label{lemma_SK_tilde}
Given $R\in\mathbb{R}_+$ and a surface $\Gamma_R \subset D$ with $\mathrm{dist}(\Gamma_R, \partial D) \geq R $ and a smooth vector field $X$ in $D$, the operators $  \widetilde{\mathcal{S}^{\kappa Q}_{\partial D}} $ and $  \widetilde{ X \circ \mathcal{S}^\kappa_{\partial D}}  $ can be decomposed as follows:
\begin{eqnarray}
\widetilde{\mathcal{S}^{\kappa Q}_{\partial D}}  &=& Q^{\frac{d-1}{2}}  \sum_{\alpha \geq 0}  h^{\frac{d-1}{2} + |\alpha|}  \widetilde{\mathcal{S}^{\kappa Q}}_{\alpha, -d-1-|\alpha|} \,, \\
\widetilde{ X \circ \mathcal{S}^{\kappa Q}_{\partial D}}  &=& Q^{\frac{d-1}{2}}  \sum_{\alpha \geq 0}  h^{\frac{d-1}{2} + |\alpha|}  \widetilde{ ( X \circ  \mathcal{S}^{\kappa Q})}_{\alpha, -d-1-|\alpha|}  \,,
\end{eqnarray}
where
\begin{equation}
\begin{split}
\widetilde{\mathcal{S}^{\kappa Q}}_{\alpha, -d-1-|\alpha|}\,,\,\widetilde{ ( X \circ  \mathcal{S}^{\kappa Q})}_{\alpha, -d-1-|\alpha|}   \in \Phi\mathrm{SO}^{-d-1-|\alpha|}_h  \,,\, \\
p_{\widetilde{\mathcal{S}^{\kappa Q}_{\alpha, -d-1-|\alpha|}} } ,  p_{ \widetilde{ ( X \circ  \mathcal{S}^{\kappa Q})}_{\alpha, -d-1-|\alpha|} } \in S^{-d-1-|\alpha|} (T^*(\partial D))
\end{split}
\end{equation}
and
\begin{equation}
\begin{split}
\big|p_{\widetilde{\mathcal{S}^{\kappa Q}}_{\alpha, -d-1-|\alpha|}} (x,\xi) \big| \leq C_{d,\alpha} R^{- \frac{d-3}{2} - |\alpha|} |\xi|^{-d-1-|\alpha|} \,,\\
 \big|   p_{\widetilde{ ( X \circ  \mathcal{S}^{\kappa Q})}_{\alpha, -d-1-|\alpha|} } (x,\xi)   \big| \leq C_{d,\alpha,X} R^{- \frac{d-1}{2} - |\alpha|} |\xi|^{-d-1-|\alpha|},
\end{split}
\end{equation}
for some constants $C_{d,\alpha}$ and $C_{d,\alpha,X}$.
\end{Lemma}

\begin{proof}
First, noting that $G(x-y)$ is real analytic outside the region $|x-y| \geq R$, together with the fact that
\[
H^{(1)}_{\frac{d-2}{2}}(x) \sim \sqrt{ \frac{2}{\pi x} } e^{\mathrm{i} \left(x - \frac{2d-5}{4} \pi \right)} \quad \mbox{as $|x| \rightarrow \infty$,}
\]
we have that
\begin{eqnarray}
G_\kappa (F_R(x)-y) &=& \kappa^{d-2} \sum_{\alpha \geq 0} \frac{1}{\alpha !} \partial^{\alpha}_z \left[ G_1 (\kappa( F_R(x)-z ) ) \right] \mid_{z=x}  (y - x)^{\alpha}\nonumber \\
&=& \kappa^{d-2} \sum_{\alpha \geq 0} \frac{1}{\alpha !}  \left[  (\partial^{\alpha} G_1) (\kappa( F_R(x)-x ) ) \right]   \kappa^\alpha (y - x)^{\alpha}\label{eq:b1}
\end{eqnarray}
where for $\kappa\gg 1$, we have
\begin{equation}\label{eq:b2}
\| (\partial^{\alpha} G_1) (\kappa( F_R(x)-x ) )  \|_{L^\infty} \leq C_d   \kappa^{- \frac{d-3}{2} - |\alpha|} R^{- \frac{d-3}{2} - |\alpha|}.
\end{equation}
{Similarly, we have
\begin{eqnarray}
&&( X_{F_R(x)} \circ G_\kappa ) (F_R(x)-y)\nonumber\\
 &=& - \kappa^{d-2} \sum_{\alpha \geq 0} \sum_{i=1}^d  \frac{1}{\alpha !} \partial^{\alpha}_z \left[ X_{i} (F_R (x)) \partial_{z_i}  \left[ G_1 (\kappa ( F_R(x)-z ) ) \right] \right] \mid_{z=x}  (y - x)^{\alpha}\label{eq:b3} \\
&=& - \kappa^{d-1} \sum_{\alpha \geq 0}  \sum_{i=1}^d \frac{1}{\alpha !} X_{i} (F_R (x)) \left[  ( \partial_{i} \partial^{\alpha} G_1) (\kappa( F_R(x)-x ) ) \right]   \kappa^\alpha (y - x)^{\alpha},\nonumber
\end{eqnarray}
}where for $\kappa \gg 1$, we have
\begin{equation}\label{eq:b4}
\| (\partial_i \partial^{\alpha} G_1) (\kappa( F_R(x)-x ) )  \|_{L^\infty} \leq C_d   \kappa^{- \frac{d-1}{2} - |\alpha|} R^{- \frac{d-1}{2} - |\alpha|}.
\end{equation}
Using \eqref{eq:b1}--\eqref{eq:b4}, one can show that
\[
\widetilde{\mathcal{S}^{\kappa Q}_{\partial D}}  = Q^{\frac{d-1}{2}}  \sum_{\alpha \geq 0}  h^{\frac{d-1}{2} + |\alpha|}  \widetilde{\mathcal{S}^{\kappa Q}}_{\alpha, -d-1-|\alpha|} \,, \quad
\widetilde{ X \circ \mathcal{S}^{\kappa Q}_{\partial D}}  = Q^{\frac{d-1}{2}}  \sum_{\alpha \geq 0}  h^{\frac{d-1}{2} + |\alpha|}  \widetilde{ ( X \circ  \mathcal{S}^{\kappa Q})}_{\alpha, -d-1-|\alpha|} ,
\]
where $\widetilde{\mathcal{S}^{\kappa Q}_{\alpha, -d-1-|\alpha|}}, \widetilde{ ( X \circ  \mathcal{S}^{\kappa Q})}_{\alpha, -d-1-|\alpha|}  \in \Phi \text{SO}^{-d-1-|\alpha|}_h $ with $p_{\widetilde{\mathcal{S}^{\kappa Q}_{\alpha, -d-1-|\alpha|}} } ,  p_{ \widetilde{ ( X \circ  \mathcal{S}^{\kappa Q})}_{\alpha, -d-1-|\alpha|} } \in S^{-d-1-|\alpha|} (T^*(\partial D)) $, via checking with {\color{black} Beal's criterion in \cite{beal} (see also  \cite{Zworski} Theorem 9.12) with an application of commutators $\text{ad}_{\ell_i (x,h \partial)} $ over $\widetilde{\mathcal{S}^{\kappa Q}}_{\alpha, -d-1-|\alpha| }$, where $\ell_1,...,\ell_N \in S^1 (T^*(\partial D))$ are linear over $\xi$:}
{\beqnx
& &   \| \text{ad}_{\ell_1 (x,h \partial)} \circ \dots \circ \text{ad}_{\ell_N (x,h \partial)} \, \widetilde{\mathcal{S}^{\kappa Q}}_{\alpha, -d-1-|\alpha| } \|_{L^2 \rightarrow L^2 }\\
&& +  \| \text{ad}_{\ell_1 (x,h \partial)} \circ \dots \circ \text{ad}_{\ell_N (x,h \partial)} \, \widetilde{ ( X \circ  \mathcal{S}^{kQ})}_{\alpha, -d-1-|\alpha|}   \|_{L^2 \rightarrow L^2 }\\
& \leq  & C_1 \kappa^{ \frac{d-3}{2} + |\alpha|} \sum_{|\beta| = N} \|  \kappa^{-N} \partial_x^\beta [ (\partial^{\alpha} G_1) (\kappa( F_R(x)-x ) ) ] \|_{L^\infty}  \\
&   &+ C_2 \kappa^{ \frac{d-1}{2} + |\alpha|}   \sum_{|\beta| = N}   \sum_{i = 1}^d \|  \kappa^{-N} \partial_x^\beta [ X_i (F_R (x))  ( \partial_i \partial^{\alpha} G_1) (\kappa( F_R(x)-x ) ) ] \|_{L^\infty} \\
&\leq& C   h^{2 N} \left( R^{- \frac{d-3}{2} - |\alpha|}  \sum_{|\beta| = N} \| \partial^\beta_x F_R(x) \|^N  R^{- \frac{d-1}{2} - |\alpha|}  + \sum_{|\beta| = N+4} \| \partial^\beta_x F_R(x) \|^{2N} \| \partial^\beta_x X(x) \|^{N} \right ) \\
&=& O(  h^{2 N} ) \leq O(h^N) \,.
\eqnx
}{\color{black} Here we remark that $ \text{ad}_{\ell_i (x,h \partial)} (\cdot) := [ \ell_i (x,h \partial), (\cdot) ] = \ell_i (x,h \partial) (\cdot) - (\cdot)  \ell_i (x,h \partial) $ is the adjoint action of the operator $\ell_i (x,h \partial) \in \Phi \text{SO}_h^{1}$.}

The proof is complete.
\end{proof}

\section{Generalized/approximate transmission eigenvalue problem }\label{sect:3}

We first recall the integral formulation of the transmission eigenvalue problem \eqref{eq:tt1}--\eqref{eq:tt2}. Using the results derived in the previous section, we can show that
\begin{Lemma} \label{transmission_operator_equation}
{\color{black} Suppose $\kappa^2$ is not a Dirichlet eigenvalue of $D$.}
There exists a symmetric operator $\mathcal{A}^{\kappa,Q}_{\partial D}$ such that we have a non-trivial solution to \eqref{eq:operator_new} if and only if there exists a solution to the following (eigenvalue) problem for $\mathcal{A}^{\kappa,Q}_{\partial D}$:
\begin{equation}
0 =  \left( \mathcal{A}^{\kappa,Q}_{\partial D}  -1 \right)  [\phi] ,\quad \phi\in L^2(\partial D),
\label{Eigenpair}
\end{equation}
where
{\color{black}
\begin{equation}
\mathcal{A}^{\kappa,Q}_{\partial D}  =
\mathcal{A}^{\kappa}_{-2} + h  \mathcal{A}^{\kappa}_{-1,1} + h^2 \mathcal{A}^{\kappa}_{-1,2},
\label{Decomposition}
\end{equation}
with $ \mathcal{A}^{\kappa}_{-2} \in \Phi\mathrm{SO}^{-2}_h $ and for $l = 1,2$, $ \mathcal{A}^{\kappa}_{-1,l} \in \Phi\mathrm{SO}^{-1}_h $,} and in the geodesic normal coordinate, we have
\beqn
p_{\mathcal{A}^{\kappa}_{-2} } (x,\xi) & = & {1  -  \bigg| \frac{ 2  |\xi|^2  }{\left( Q^2 -1 \right) } \left( \frac{ \sqrt{1 - |\xi|^2 } }{ \sqrt{Q^2 - |\xi|^2 } }  - 1 \right)  \bigg|^2  } \,, \label{Hamiltonian} \\
p_{ \mathcal{A}^{\kappa}_{-1,1} } (x,\xi) & = & \, \Re \left( 
  \frac{  4 i C_0 |\xi|^2 \langle  A(x) \xi, \xi \rangle  }{\left( Q^2 -1 \right)  \left( 1 - |\xi|^2 \right) \left( Q^2 - |\xi|^2 \right)^{\frac{3}{2}}  } \left( \frac{ \sqrt{1 - |\xi|^2 } }{ \sqrt{Q^2 - |\xi|^2 } }  - 1 \right) 
 \right)  \label{Remainder} 
\eqn
and whenever $|\xi| > \max\{1,Q\}$, we have
\beqnx
p_{\mathcal{A}^{\kappa}_{-2} } (x,\xi) 
&=&  - \left(  \frac{ 3 Q^{2}  + 1 }{ 2 }  \right)  |\xi |^{-2}    + \mathcal{O}( | \xi |^{-4}  ) \,. \\
p_{ \mathcal{A}^{\kappa}_{-1,1} } (x,\xi) & = & - \, \Re \left(  
 \,    C_0  \sum_{i,j=1}^{d-1}  A_{ij} (x) \,  |\xi|^2 \partial_i \partial_j |\xi |^{-1}  \right)  +  \mathcal{O}(  |\xi|^{-3}  )  \\
p_{ \mathcal{A}^{\kappa}_{-1,2} } (x,\xi) & = &   \mathcal{O}(  |\xi|^{-1}   )   \,.
\eqnx
Here, $\mathcal{A}^{\kappa}_{-2} $ is the principal part of the operator $\mathcal{A}^{\kappa,Q}_{\partial D} $ and $\mathcal{A}^{\kappa,Q}_{\partial D}  \leq 1 $.
\end{Lemma}

\begin{proof}
Using Lemma~\ref{lemma_SK}, {\color{black} and the fact that $\kappa^2$ is not a Dirichlet eigenvalue of $D$,} we can calculate in the geodesic normal coordinate that
{\color{black} whenever $|\xi| > \max\{1,Q\}$, we have
\begin{eqnarray}
&  & p_{  \left[ \mathcal{S}^{\kappa}_{\partial D}   \right]^{-1}   \mathcal{S}^{\kappa Q}_{\partial D}  }(x,\xi)\nonumber \\
& = & \left(   1  - \frac{  \widehat{\widehat{C}}_{d} }{  \widehat{C}_{d} } |\xi |^{-2}  - \frac{ \widehat{\widehat{\widehat{C}}}_{d} \widehat{C}_{d}  -  \widehat{\widehat{C}}_{d}^2 }{ \widehat{C}_{d} ^2}  |\xi |^{-4} + \mathcal{O}( | \xi |^{-6}  )  \right) \nonumber \\
&  &  \times  \left(  1 + \frac{  \widehat{\widehat{C}}_{d}  }{ \widehat{C}_{d}  } Q^{2} |\xi |^{-2} + \frac{\widehat{\widehat{\widehat{C}}}_{d} }{ \widehat{C}_{d}  } Q^{4} |\xi |^{-4}  + O( | \xi |^{-6}  )   \right) \nonumber \\
&   & + h^2  \bigg( 1  + \mathcal{O}( | \xi |^{-2}  )  \bigg)  \left( \frac{ |\xi |  p_{ \mathcal{S}^{\kappa}_{-1}}(x,\xi)  }{\widehat{C}_{d} }  \right)  \left( 1 -   \bigg(1 + O( | \xi |^{-2}  )   \bigg) \bigg( 1  + \mathcal{O}( | \xi |^{-2}  )  \bigg)  \right) \nonumber \\
&  & + \mathcal{O}( h^2 | \xi |^{-3}  )   \label{eq:cd1} \\
& = &  1 + \frac{  \widehat{\widehat{C}}_{d} }{  \widehat{C}_{d}  } (Q^{2} -1) |\xi |^{-2}  
+ \frac{ \left(  \widehat{\widehat{\widehat{C}}}_{d} \widehat{C}_{d} (Q^2 +1)  -  \widehat{\widehat{C}}_{d} ^2 \right) ( Q^{2} -1)  }{ \widehat{C}_{d} ^2}  |\xi |^{-4} + \mathcal{O}( | \xi |^{-6}  ) +  \mathcal{O}( h^2 | \xi |^{-3}  ) \,,\nonumber 
\end{eqnarray}
where we have used the fact that $p_{ \mathcal{S}^{\kappa Q}_{-3}}(x,\xi)  - p_{ \mathcal{S}^{\kappa }_{-3}}(x,\xi) = \mathcal{O}(  | \xi |^{-4}  )   $ in the first line, and  $p_{ \mathcal{S}^{\kappa }_{-3}}(x,\xi)  =  \mathcal{O}(  | \xi |^{-3}  )  $ in the second line.
This then yields that, whenever $|\xi| > \max\{1,Q\}$, we have
\beqn
&  & p_{  ( - \f{1}{2} I + {\mathcal{K}^{\kappa Q}_{\partial D}}^* ) - ( - \f{1}{2} I + {\mathcal{K}^{\kappa}_{\partial D}}^* )  \left[ \mathcal{S}^{\kappa}_{\partial D}   \right]^{-1}  \mathcal{S}^{\kappa Q}_{\partial D}   } (x,\xi)\nonumber \\
&=& - \frac{1 }{2} +  \frac{ h \, }{2} \bigg( \widehat{C}_{d}   \sum_{i,j=1}^{d-1} A_{ij}(x) \partial_i \partial_j |\xi |  +   \widehat{\widehat{C}}_{d}  Q^{2} \sum_{i,j=1}^{d-1}  A_{ij} (x)\partial_i \partial_j |\xi |^{-1}  + \mathcal{O}( | \xi |^{-5}  ) \bigg)    \nonumber\\
&&   + \left(  \frac{1}{2}  -  \frac{ h }{2}  \bigg(  \widehat{C}_{d}  \sum_{i,j=1}^{d-1} A_{ij}(x) \partial_i \partial_j |\xi|  +  \widehat{\widehat{C}}_{d}  \sum_{i,j=1}^{d-1}  A_{ij} (x)\partial_i \partial_j |\xi |^{-1} + O( | \xi |^{-5}  ) \bigg)    \right)\nonumber \\
&  &  \times \left( 
1 + \frac{  \widehat{\widehat{C}}_{d} }{  \widehat{C}_{d}  } (Q^{2} -1) |\xi |^{-2}  
+ \frac{ \left(  \widehat{\widehat{\widehat{C}}}_{d} \widehat{C}_{d} (Q^2 +1)  -  \widehat{\widehat{C}}_{d} ^2 \right) ( Q^{2} -1)  }{ \widehat{C}_{d} ^2}  |\xi |^{-4} + \mathcal{O}( | \xi |^{-6}  )   
 \right)\nonumber \\
&   &  + \frac{h^2}{2}  \left( p_{ \mathcal{K}^{\kappa Q}_{-2}}(x,\xi) -  p_{ \mathcal{K}^{\kappa }_{-2}}(x,\xi) \right)   +  \mathcal{O}( h^2 | \xi |^{-3}  )   \nonumber \\
&=&   \frac{  \widehat{\widehat{C}}_{d} }{  2 \widehat{C}_{d}  } (Q^{2} -1) |\xi |^{-2}  
+ \frac{ \left(  \widehat{\widehat{\widehat{C}}}_{d} \widehat{C}_{d} (Q^2 +1)  -  \widehat{\widehat{C}}_{d} ^2 \right) ( Q^{2} -1)  }{ 2 \widehat{C}_{d} ^2}  |\xi |^{-4} + \mathcal{O}( | \xi |^{-6}  )\nonumber  \\
&  & - \frac{ h}{2}  ( Q^{2} -  1)  \bigg( \widehat{\widehat{C}}_{d}   \sum_{i,j=1}^{d-1}  A_{ij} (x)\partial_i \partial_j |\xi |^{-1}   + O(  |\xi|^{-5}  )  \bigg) + \mathcal{O}(h^2 |\xi|^{-3})   \label{eq:cd2} \, ,
\eqn
where we have used the fact that $ p_{ \mathcal{K}^{\kappa Q}_{-2}}(x,\xi) -  p_{ \mathcal{K}^{\kappa }_{-2}}(x,\xi)  = \mathcal{O}(  | \xi |^{-3}  )   $ in the last line.
Hence, whenever $|\xi| > \max\{1,Q\}$,
\beqn
&  & p_{  \frac{2  \widehat{C}_{d}  }{  \widehat{\widehat{C}}_{d}  }  ( Q^{2} -  1)^{-1}  \left( - h^2 \Delta_{\partial D} + h^2 \right) \left(  ( - \f{1}{2} I + {\mathcal{K}^{\kappa Q}_{\partial D}}^* ) - ( - \f{1}{2} I + {\mathcal{K}^{\kappa}_{\partial D}}^* )  \left[ \mathcal{S}^{\kappa}_{\partial D}   \right]^{-1}  \mathcal{S}^{\kappa Q}_{\partial D}  \right)  } (x,\xi) \nonumber\\
&=& 1 + \left( \frac{   \widehat{\widehat{\widehat{C}}}_{d} \widehat{C}_{d} (Q^2 +1)  -  \widehat{\widehat{C}}_{d} ^2 }{ \widehat{C}_{d}   \widehat{\widehat{C}}_{d}   } \right) |\xi |^{-2} + \mathcal{O}( | \xi |^{-4}  ) \nonumber \\
&  & - h \,   \bigg(  C_{d}  \sum_{i,j=1}^{d-1}  A_{ij} (x) \,  |\xi|^2 \partial_i \partial_j |\xi |^{-1}    + \mathcal{O}(  |\xi|^{-3}  )  \bigg) + \mathcal{O}(h^2  |\xi|^{-1} )  \label{eq:cd3} \\
&=&   1  +   \left(  \frac{ 3 Q^{2}  + 1 }{ 4 }  \right)  |\xi |^{-2} + O( | \xi |^{-4}  ) - h \,   \bigg(  C_0  \sum_{i,j=1}^{d-1}  A_{ij} (x) \,  |\xi|^2 \partial_i \partial_j |\xi |^{-1}    + \mathcal{O}(  |\xi|^{-3}  )  \bigg) + \mathcal{O}(h^2  |\xi|^{-1} ) , \nonumber
\eqn
where we obtained the last line by recalling that
\begin{equation*}
 \widehat{C}_d = C_0 \,,\, \quad   \widehat{\widehat{C}}_d  = \frac{C_0}{2 }  \,,\, \quad \widehat{\widehat{\widehat{C}}}_{d} = \frac{3 C_0}{8 }  \,.
\end{equation*}
}

By tracing the computations in \eqref{eq:cd1}--\eqref{eq:cd3}, and their original expressinon in \eqref{eq:dec1}- \eqref{eq:dec7}
we can obtain a more precise expression of terms in the last expansion in \eqref{eq:cd3} whenever $|\xi| > \max\{1,Q\}$ as follows:
\begin{eqnarray}
 1  +   \left(  \frac{ 3 Q^{2} + 1 }{ 4 }  \right) |\xi |^{-2} + O( | \xi |^{-4}  )   & =  &
 \frac{ 2  |\xi|^2  }{\left( Q^2 -1 \right) } \left( \frac{ \sqrt{1 - |\xi|^2 } }{ \sqrt{Q^2 - |\xi|^2 } }  - 1 \right)  \label{eq:gg3} \\
 C_0  \sum_{i,j=1}^{d-1}  A_{ij} (x) \,  |\xi|^2 \partial_i \partial_j |\xi |^{-1}    + \mathcal{O}(  |\xi|^{-3}  ) & = & - \frac{ i C_0  |\xi|^2  \langle  A(x) \xi, \xi \rangle   }{\left( 1 - |\xi|^2 \right) \left( Q^2 - |\xi|^2 \right)^{\frac{3}{2}}  }  \label{eq:gg4} 
\end{eqnarray}
where the terms at the right handside are now the corresponding symbols for all $\xi$ with $|\xi| \neq 1,Q$.

{\color{black}
Next we denote
\[
\mathcal{B}^{\kappa,Q}_{\partial D} := I -  \frac{2  \widehat{C}_{d}  }{  \widehat{\widehat{C}}_{d}  }   ( Q^{2} -  1)^{-1}  \left( - h^2 \Delta_{\partial D} + h^2  \right) \left(  ( - \f{1}{2} I + {\mathcal{K}^{\kappa Q}_{\partial D}}^* ) - ( - \f{1}{2} I + {\mathcal{K}^{\kappa}_{\partial D}}^* )  \left[ \mathcal{S}^{\kappa}_{\partial D}   \right]^{-1}  \mathcal{S}^{\kappa Q}_{\partial D}  \right) \,, 
\]
and notice that $( - \Delta_{\partial D} + 1  ) f = 0$ if and only if $f = 0$.  Hence, we have that whenever $h^{-2}$ is not a Dirichlet eigenvalue of $D$, there exists a non-trivial solution to \eqref{eq:operator_new} if and only if there exists a solution to
\begin{equation}\label{eq:new1}
0 =  \left( \mathcal{B}^{\kappa,Q}_{\partial D}  - 1  \right)  [\phi] \quad \text{ with }\quad
\mathcal{B}^{\kappa,Q}_{\partial D}  = \mathcal{B}^{\kappa}_{-2} + h \mathcal{B}^{\kappa}_{-1,1}  +  h^2 \mathcal{B}^{\kappa}_{-1,2} ,
\end{equation}
where
\begin{eqnarray}
\mathcal{B}^{\kappa}_{-2} &:=& \mathrm{Op}_{ \left( 1 -  \frac{ 2  |\xi|^2 }{\left(Q^2 -1\right)} \left( \frac{ \sqrt{1 - |\xi|^2 } }{ \sqrt{Q^2 - |\xi|^2 } }  - 1 \right)  \right)  ,h}   \in \Phi\mathrm{SO}^{-2}_h \;   \label{eq:new31}  \\ 
\mathcal{B}^{\kappa}_{-1,1} &:=& \mathrm{Op}_{  \frac{  i C_0 |\xi|^2   \langle A(x) \xi, \xi \rangle }{\left( 1 - |\xi|^2 \right) \left( Q^2 - |\xi|^2 \right)^{\frac{3}{2}}  }   ,h} \in \Phi\mathrm{SO}^{-1}_h   \label{eq:new32}\\
 \mathcal{B}^{\kappa}_{-1,2} &:=& h^{-2} \left( \mathcal{B}^{\kappa,Q}_{\partial D}  - \mathcal{B}^{\kappa}_{-2} - h  \mathcal{B}^{\kappa}_{-1,1} \right) \in \Phi\mathrm{SO}^{-1}_h  \label{eq:new3} \,,\,
\end{eqnarray}
}which in turn holds if and only if there exists a solution to
\begin{equation}\label{eq:new4}
0 = \left( \mathcal{B}^{\kappa,Q}_{\partial D}  - 1  \right)^*\left(\mathcal{B}^{\kappa,Q}_{\partial D}  -  1 \right)  [\phi] .
\end{equation}
Writing
\begin{equation}\label{eq:new5}
\mathcal{A}^{\kappa,Q}_{\partial D} :=  \mathcal{B}^{\kappa,Q}_{\partial D}  + \left( \mathcal{B}^{\kappa,Q}_{\partial D} \right)^*  - \left( \mathcal{B}^{\kappa,Q}_{\partial D} \right)^* \mathcal{B}^{\kappa,Q}_{\partial D} \,,
\end{equation}
which is symmetric, and noticing that
\begin{eqnarray}\label{eq:new6}
p_{\mathcal{A}^{\kappa,Q}_{\partial D} }(x,\xi) = 1 -  \left|  p_{\mathcal{B}^{\kappa}_{-1,1} }   (x,\xi) -1  \right|^2 + 2 h \, \Re \left(   \left(  p_{\mathcal{B}^{\kappa}_{2}  } (x,\xi) -1 \right) \overline{ p_{\mathcal{B}^{\kappa}_{-1,1}  } (x,\xi) }   \right)   + \mathcal{O}( h^2 |\xi|^{-1} )  \,,  \quad 
\end{eqnarray}
one can readily get \eqref{Hamiltonian}-\eqref{Remainder}, and this completes the proof.

\end{proof}

{\color{black}
\begin{Corollary} \label{transmission_operator_corollary}
 {\color{black} Fix $\kappa>0$. Suppose $\kappa^2$ is not a Dirichlet eigenvalue of $D$,} then $\mathcal{A}^{\kappa,Q}_{\partial D}$ is a compact self-adjoint operator in $L^2(\partial D, d \sigma)$.
\end{Corollary}

\begin{proof}
 {\color{black} Whenever $\kappa^2$ is not a Dirichlet eigenvalue of $D$,}
{\color{black} it is evident that the operator $  ( - \f{1}{2} I + {\mathcal{K}^{\kappa Q}_{\partial D}}^* ) - ( - \f{1}{2} I + {\mathcal{K}^{\kappa}_{\partial D}}^* )  \left[ \mathcal{S}^{\kappa}_{\partial D}   \right]^{-1}  \mathcal{S}^{\kappa Q}_{\partial D}   $ is bounded from $L^2(\partial D, d \sigma)$ to  $H^{-1}(\partial D, d \sigma)$,} and therefore $\mathcal{A}^{\kappa,Q}_{\partial D}$ is bounded from $L^2(\partial D, d \sigma)$ to $H^{-3}(\partial D, d \sigma)$.
Now we consider $\chi_Q \in C^{\infty}(T^*(\partial D))$ to be such that in the geodesic normal coordinate,
\begin{eqnarray*}
\begin{cases}
\chi_Q(x,\xi) = 1  & \text{ when } |\xi| \leq 2Q \,, \\
0 \leq \chi_Q(x,\xi) \leq 1  & \text{ when } 2Q \leq |\xi| \leq 4Q \,, \\
\chi_Q(x,\xi) = 0  & \text{ when } 4 Q \leq |\xi| \,. \\
\end{cases}
\end{eqnarray*}
Then we have that $\chi_Q \in S^{-\infty} (T^*(\partial D))$ is in the Hormander symbol class and hence for each $\kappa$, $\mathrm{Op}_{\chi_Q,h}$ is bounded from $H^{-3}(\partial D, d \sigma)$ to $H^{1}(\partial D, d \sigma)$ (as well as from $L^2(\partial D, d \sigma)$ to $L^2(\partial D, d \sigma)$.)    
Therefore $\mathrm{Op}_{\chi_Q,h} \, \mathcal{A}^{\kappa,Q}_{\partial D} \,\mathrm{Op}^*_{\chi_Q,h} $ is bounded from $L^2(\partial D, d \sigma)$ to $H^{1}(\partial D, d \sigma)$, and is therefore compact in $L^2(\partial D, d \sigma)$.
Now, it is clear that 
\begin{eqnarray}
\mathcal{A}^{\kappa,Q}_{\partial D} := \mathrm{Op}_{\chi_Q,h} \, \mathcal{A}^{\kappa,Q}_{\partial D} \, \mathrm{Op}^*_{\chi_Q,h}  + \mathcal{E}^{\kappa,Q}_{\partial D} \,,
\label{error_difference}
\end{eqnarray}
where the operator $\mathcal{E}^{\kappa,Q}_{\partial D}$ is such that
\begin{eqnarray*}
p_{\mathcal{E}^{\kappa,Q}_{\partial D} } 
&=& \left( 1 -  \chi_Q^2 \right) p_{\mathcal{A}^{\kappa}_{-2} }   + h \left( 1 -  \chi_Q^2 \right)  p_{\mathcal{A}^{\kappa}_{-1,1} }   \\
& & + h^2 \left( \left( 1 -  \chi_Q^2 \right)  p_{\mathcal{A}^{\kappa}_{-1,2} }  + \langle \partial_\xi p_{A,-2} \, , \, \partial^2_{xx} \, \chi_Q \partial_\xi \chi_Q \rangle \right) + \mathcal{O}( h^3 |\xi|^{-2}  )
\end{eqnarray*}
with 
$  ( 1 -  \chi_Q^2 ) p_{\mathcal{A}^{\kappa}_{-2} }    \in  S^{-2} (T^*(\partial D)) $, $  ( 1 -  \chi_Q^2 ) p_{\mathcal{A}^{\kappa}_{-1,1} }    \in  S^{-1} (T^*(\partial D)) $ and  $  \left( 1 -  \chi_Q^2 \right)  p_{\mathcal{A}^{\kappa}_{-1,2} }  + \langle \partial_\xi p_{A,-2} \, , \, \partial^2_{xx} \, \chi_Q \partial_\xi \chi_Q \rangle    \in  S^{-1} (T^*(\partial D)) $  being all now in the H\"ormander symbol class.
Therefore $\mathcal{E}^{\kappa,Q}_{\partial D}$ is bounded from $L^2(\partial D, d \sigma)$ to $H^{1}(\partial D, d \sigma)$, is hence also compact in $L^2(\partial D, d \sigma)$.
Now since $\mathrm{Op}_{\chi_Q,h} \, \mathcal{A}^{\kappa,Q}_{\partial D} \,\mathrm{Op}^*_{\chi_Q,h} $ and $\mathcal{E}^{\kappa,Q}_{\partial D}$ are both compact in $L^2(\partial D, d \sigma)$, the symmetric operator $\mathcal{A}^{\kappa,Q}_{\partial D}$ is a self-adjoint operator.
The proof is complete.
\end{proof}
}

\begin{Remark}\label{rem:i1}
Since $\mathcal{A}^{\kappa,Q}_{\partial D} $ is compact self-adjoint in $L^2(\partial D, d \sigma)$ {\color{black} for each $\kappa > 0$ such that $\kappa^2$ is not a Dirichlet eigenvalue of $D$,} we write a sequence of eigen-pairs $ \{ ( \lambda_j(h), \phi_j(h) ) \}_{j=1}^{\infty} $ where, fixing $h=\kappa^{-1} > 0$, the set of eigenfunctions form an orthonormal frame and the set of eigenvalues converges to zero as $j$ goes to infinity. It is easily seen that $\kappa$ is a transmission eigenvalue if and only if there exists $j$ such that $\lambda_j(\kappa^{-1}) = 1$. It can also be directly inferred from the boundedness of $\mathcal{A}^{\kappa,Q}_{\partial D} $ from $L^2(\partial D, d \sigma)$ to $H^{1}(\partial D, d \sigma)$ and the Sobolev embedding that $\phi_j(h) \in \mathcal{C}^{\infty} (\partial D)$.
\end{Remark}

Next, we introduce the definition of $\varepsilon$-almost transmission eigen-pairs, which for terminological convenience shall be referred to as the generalized transmission eigen-pairs in what follows. {\color{black} Let $\varepsilon$ be sufficiently small} and $0\leq\delta\leq\varepsilon$. Consider the following transmission eigenvalue problem:
\begin{equation}
 \left( \mathcal{A}^{\kappa,Q}_{\partial D}  -1 + \delta \right)  [\phi_\delta]=0,\quad \|\phi_\delta\|_{L^2(\partial D)}=1.
\label{almost_eigenpair}
\end{equation}
{\color{black} Suppose $\kappa^2$ is not a Dirichlet eigenvalue of $D$.} Set
\begin{equation}\label{eq:ab2}
\varphi_\delta =  \left(  \mathcal{S}^{\kappa}_{\partial D}  \right)^{-1} \mathcal{S}^{\kappa Q}_{\partial D} [\phi_\delta]\ \ \mbox{on}\ \ \partial D.
\end{equation}
One can directly verify that the pair  $(\phi_\delta, \varphi_\delta) \in L^2 (\partial D) \times L^2 (\partial D)$  satisfies
\begin{equation}\label{eq:ee1}
\begin{cases}
\mathcal{S}^{\kappa Q}_{\partial D} [\phi_\delta]  - \mathcal{S}^{\kappa}_{\partial D} [\varphi_\delta]  = 0 ,\medskip \\
{ \| ( - \f{1}{2} I + {\mathcal{K}^{\kappa Q}_{\partial D}}^* )[\phi_\delta] - ( - \f{1}{2} I + {\mathcal{K}^{\kappa}_{\partial D}}^* ) [\varphi_\delta] \|_{H^2}^2 \leq  \varepsilon \kappa^4  \,\frac{|\tilde{C}_{d}|^2}{   4 |C_{d}|^2  }  ( Q^{2} -  1)^{2} } \,.
\end{cases}
\end{equation}
Hence, if we let $u_\delta,v_\delta \in H^1(D)$ be defined as follows:
\beqn
u_\delta = \mathcal{S}^{\kappa Q}_{\partial D} [\phi_\delta] \,,\, \quad v_\delta = \mathcal{S}^{\kappa}_{\partial D} [\varphi_\delta]   & \text{ in } D.
\label{eq:almost_eigenfunctions}
\eqn
Then they approximately satisfy \eqref{eq:trans1} in a sense that
\beqn
\begin{cases}
& \Delta u_\delta+\kappa^2 Q^2 u_\delta=0\hspace*{.82cm}\mbox{in}\ \ D,\medskip\\
& \Delta v_\delta+\kappa^2 v_\delta=0 \hspace*{1.4cm}\mbox{in}\ \ D,\medskip\\
& u_\delta=v_\delta, \quad {  \| \partial_\nu  u_\delta - \partial_\nu v_\delta \|_{H^2(\partial D)}^2 \leq  \varepsilon \kappa^4  \,\frac{|\tilde{C}_{d}|^2}{   4 |C_{d}|^2  }  ( Q^{2} -  1)^{2} }     \quad\mbox{on}\ \ \partial D\, .
\end{cases}
\label{eq:trans2}
\eqn

%

\begin{Definition}
\label{almost_transmission}
{\color{black} Suppose $\kappa^2$ is not a Dirichlet eigenvalue of $D$.}  {\color{black} Given $\varepsilon > 0 $ sufficiently small,} we call {$\kappa > 0$} a generalised transmission eigenvalue {\color{black} (with respect to $\varepsilon$) } if for such $\kappa$, there exists $(\delta, \phi_\delta)$ with $0\leq \delta \leq \varepsilon$ such that \eqref{almost_eigenpair} holds.  The pair $(u_\delta,v_\delta)$ defined in \eqref{eq:almost_eigenfunctions} are called the $\varepsilon$-almost transmission eigenfunctions, in a sense that they satisfy \eqref{eq:trans2}. For simplicity, they are also referred to as the generalised transmission eigenfunctions.
We also denote the generalised transmission eigenspace as follows:
\begin{equation}\label{eq:gs1}
\begin{split}
\mathbb{E} (\kappa,\varepsilon) :=& \bigcup_{0\leq \delta \leq \varepsilon} \mathrm{Span}\bigg \{  \left( \mathcal{S}^{\kappa Q}_{\partial D} [\phi_\delta] , \mathcal{S}^{\kappa}_{\partial D} [\varphi_\delta] \right):\\
&\qquad\qquad \phi_\delta\ \mbox{and}\ \varphi_\delta \ \mbox{are given in \eqref{almost_eigenpair} and \eqref{eq:ab2} respectively} \bigg \}.
\end{split}
\end{equation}
We also refer to the following quantity as the multiplicity of a generalised transmission eigenvalue $\kappa$:
\beqn
\mathfrak{m}(\kappa,\varepsilon) := \mathrm{dim} \bigcup_{ 0\leq \delta \leq \varepsilon} \left \{ \phi \in L^2(\partial D) \,:\,  \left( \mathcal{A}^{\kappa,Q}_{\partial D}  -1 + \delta \right)  [\phi]  = 0 \right \} \,.
\label{Definition_m}
\eqn
\end{Definition}

\begin{Remark}\label{rem:def1}
 We notice that following this definition, a candidate $\phi_\delta$ that satisfies \eqref{almost_eigenpair} for some $0\leq \delta \leq \varepsilon$ will sit in the subspace
\[
\mathbb{V}_\varepsilon := \overline{ \bigcup \left \{  \mathbf{V} \subset L^2(\partial D) \,,\, \left \|  \left(\mathcal{B}^{\kappa,Q}_{\partial D}  -  1 \right)\big|_{\mathbf{V}} \right \|_{\mathcal{L}(L^2(\partial D), L^2(\partial D))} \leq \varepsilon \right\} },
\]
since it is immediate to check that \eqref{almost_eigenpair} implies
\beqnx
 \left\| \left(\mathcal{B}^{\kappa,Q}_{\partial D}  -  1 \right)  [\phi_\delta] \right\|^2_{L^2(\partial D)} = \delta \| \phi_\delta \|^2_{L^2}  \leq \varepsilon \,.
\eqnx
With the above interpretation, we realize that the concept of the generalised transmission eigenvalue in Definition~\ref{almost_transmission} is intimately related to the concept of pseudospectrum of the operator $\mathcal{B}^{\kappa,Q}_{\partial D}$.
\end{Remark}

\begin{Remark}\label{rem:def2}
{\color{black} Whenever $\kappa^2$ is not a Dirichlet eigenvalue of $D$,} we see from Corollary \ref{transmission_operator_corollary} that the eigenvalues of the operator $\mathcal{A}^{\kappa,Q}_{\partial D}$ are real and have a cluster point $0$, as an approximation to a finite rank operator. Hence, the union of all the eigenspaces for any fixed threshold away from $0$ is finite dimensional. In particular, $\mathbb{E}(\kappa,\varepsilon)$ in \eqref{eq:gs1} is finite dimensional and $\mathfrak{m}(\kappa,\varepsilon)$ is a finite number.
\end{Remark}

\begin{Remark}
It is easily seen that an exact transmission eigenvalue is a generalised transmission eigenvalue. Hence, the generalised eigenspace $\mathbb{E}(\kappa,\varepsilon)$ contains the exact transmission eigenfunctions. Indeed, $\mathbb{E}(\kappa, 0)$ is an exact transmission eigenspace. For the numerical finding in \cite{CDHLW} as well as the examples presented in Fig.~1, considering the numerical errors, they are actually certain generalised transmission eigenfunctions.
\end{Remark}

{
\begin{Remark}\label{rem:i0}
We notice that in fact a similar equation can be obtained if instead that $\kappa^2 Q^2$ is not a Dirichlet eigenvalue of $D$, and in this case we can consider
\begin{equation}\label{eq:tt22}
\phi = \left[  \mathcal{S}^{\kappa Q}_{\partial D}  \right]^{-1}   \mathcal{S}^{\kappa}_{\partial D}   [ \varphi ]   \, , \, \quad \left(  ( - \f{1}{2} I +  {\mathcal{K}^{\kappa}_{\partial D}}^* ) - ( - \f{1}{2} I +  {\mathcal{K}^{\kappa Q}_{\partial D}}^* ) \left[  \mathcal{S}^{\kappa Q}_{\partial D}  \right]^{-1}   \mathcal{S}^{\kappa}_{\partial D} \right)  [ \varphi]  = 0.
\end{equation}
Similarly, if instead $\kappa^2$ is not a Neumann eigenvalue, we can consider
\begin{equation}\label{eq:tt2n}
\varphi = \left[ - \f{1}{2} I + {\mathcal{K}^{\kappa}_{\partial D}}^*  \right]^{-1}    ( - \f{1}{2} I + {\mathcal{K}^{\kappa Q}_{\partial D}}^* )   [\phi]   \, , \, \left( \mathcal{S}^{\kappa}_{\partial D}  -  \mathcal{S}^{\kappa Q}_{\partial D}   \left[ - \f{1}{2} I + {\mathcal{K}^{\kappa}_{\partial D}}^*  \right]^{-1}   ( - \f{1}{2} I + {\mathcal{K}^{\kappa Q}_{\partial D}}^* )  \right)  [\phi]  = 0\,;
\end{equation}
and
if $\kappa^2 Q^2$ instead is not a Neumann eigenvalue, we can look into
\begin{equation}\label{eq:tt22n}
\phi = \left[ - \f{1}{2} I + {\mathcal{K}^{\kappa Q}_{\partial D}}^*  \right]^{-1} ( - \f{1}{2} I + {\mathcal{K}^{\kappa}_{\partial D}}^* )  [\varphi]   \, , \,  \left( \mathcal{S}^{\kappa Q}_{\partial D}  - \mathcal{S}^{\kappa}_{\partial D}  \left[  - \f{1}{2} I + {\mathcal{K}^{\kappa Q}_{\partial D}}^*  \right]^{-1} ( - \f{1}{2} I + {\mathcal{K}^{\kappa}_{\partial D}}^* )  \right)  [\varphi]  = 0. 
\end{equation}
It is noted that alternative definitions of generalized transmission eigenvalues in Definition \ref{almost_transmission} can be similarly given with \eqref{eq:tt22}, \eqref{eq:tt2n} or \eqref{eq:tt22n}, and the main results in Theorems \ref{thm:main1} and \ref{thm:main2} can be shown to hold by following completely similar proofs with either of these alternative definitions.

In view of the above, the only technical restriction of our results in this work comes from the situation when $\kappa^2$ and $\kappa^2 Q^2$ are simultaneously both the Dirichlet and Neumann eigenvalues of $D$, and yet at the same time $\kappa^2$ is a transmission eigenvalue of $D$, which generically does not occur. In fact, even if it occurs, due to the countability of the eigenvalues (for a fixed $D$), one can vary $Q$ (with uncountably many choices) to make $\kappa^2Q^2$ not a Dirichlet/Neumann eigenvalue so that our results hold. Moreover, to further corroborate this point, we refer to \cite{CDHLW} where the analysis evidently shows that for a radial domain, a $\kappa^2$ cannot even be simultaneously a Dirichlet and transmission eigenvalue. 
 In the rest of the paper, for the sake of exposition and clarity, we only focus on \eqref{eq:tt2} with the condition that $\kappa^2$ is not a Dirichlet eigenvalue of $D$.
\end{Remark}
}

\section{Surface concentration of generalised transmission eigenfunctions}

In this section, we present the main results on the surface concentration of the generalised transmission eigenfunctions. 
By Remark~\ref{rem:i1}, for any fixed $h=\kappa^{-1}>0$ {\color{black} with $\kappa^2$ not a Dirichlet eigenvalue of $D$}, we let
\begin{equation}\label{eq:ppn1}
\{ ( \lambda_j(h), \phi_j(h) ) \}_{j=1}^{\infty}\ \ \mbox{with}\ \ \lim_{j\rightarrow\infty}\lambda_j(h)=0\ \mbox{and}\ \|\phi_j(h)\|_{L^2(\partial D)}=1,
\end{equation}
be the sequence of eigen-pairs of $\mathcal{A}^{\kappa,Q}_{\partial D} $.  {\color{black} Whenever $\kappa^2$ is not a Dirichlet eigenvalue of $D$,} it is noted that $\kappa\in\mathbb{R}_+$ is a transmission eigenvalue if and only if $\lambda_{j_0}(h)=1$ for a certain $j_0\in\mathbb{N}$, or equivalently $\mathfrak{m}(\kappa, 0)>0$, whereas $1-\varepsilon\leq\lambda_{j_0}(h)\leq 1$ corresponds to a generalised transmission eigenvalue according to Definition~\ref{almost_transmission}. In what follows, we set
${\varphi}_j (h) := \left[ \mathcal{S}^{\kappa}_{\partial D}  \right]^{-1} \mathcal{S}^{\kappa Q}_{\partial D}  \, [ \phi_j (h) ] $ and write
{\color{black}
\beqn
u_j(h) := \mathcal{S}^{\kappa Q}_{\partial D}  \, [ \phi_j (h) ] \quad \text{ in } D \,, \, \quad   \partial_\nu u_j(h) =   ( - \f{1}{2} I + {\mathcal{K}^{\kappa Q}_{\partial D}}^* ) \, [ \phi_j (h) ] \quad \text{ on } \partial D\,; \label{def_u_i} \\
v_j(h) := \mathcal{S}^{\kappa}_{\partial D}  \, [ {\varphi}_j (h) ]  \quad \text{ in } D \,, \, \quad   \partial_\nu v_j(h) =   ( - \f{1}{2} I + {\mathcal{K}^{\kappa}_{\partial D}}^* ) \, [ {\varphi}_j (h) ] \quad \text{ on } \partial D \label{def_v_i} \, .
\eqn
}

\subsection{Surface concentration in an averaging sense}

The first main result of this section is stated as follows.

\begin{Theorem}\label{thm:main1}
{\color{black} 
Suppose that $D$ is a 
$C^\infty$ domain.
Given $ \varepsilon > 0$ sufficiently small,}
{\color{black} whenever $\kappa^{2}$ is not a Dirichlet eigenvalue of $D$,} let us define the averaging functionals $\mathscr{I}^\ell_\zeta(\kappa,x,\varepsilon)$, $\zeta=u, v$ and $\ell=0, 1$ as follows:
\beqn
\mathscr{I}_\zeta^0(\kappa,x,\varepsilon) &:=& \frac{ \sum \limits_{1-\varepsilon \leq  \lambda_j (\kappa^{-1}) \leq 1}   \Big|  \big[\zeta_j(\kappa^{-1})\big] (x) \Big|^2 }{   \# \big\{ 1-\varepsilon \leq  \lambda_j (\kappa^{-1}) \leq 1 \big\} },\quad \zeta=u\ \mbox{or}\ v;\medskip \label{index0}  \\
\mathscr{I}_{\zeta}^1(\kappa,x,\varepsilon) &:=& \frac{ \sum \limits_{1-\varepsilon \leq  \lambda_j (\kappa^{-1}) \leq 1}   \Big|  \nabla \big[\zeta_j(\kappa^{-1})\big] (x) \Big|^2 }{   \# \big\{ 1-\varepsilon \leq  \lambda_j (\kappa^{-1}) \leq 1 \big\} },\quad \zeta=u\ \mbox{or}\ v;  \label{index1},
\eqn
where $\#$ counts the number of elements of a given set. Let $\Gamma_R\subset D$ be a smooth closed surface such that $\mathrm{dist}(\Gamma_R, \partial D):=R\geq 0$. Then for any bump function $\gamma(x) \in \mathcal{C}^{\infty} (\Gamma_R) $, we have as $\kappa\rightarrow \infty$:
\beqn
\int_{\Gamma_R} \gamma(x) \mathscr{I}_{\zeta}^0(\kappa,x,\varepsilon)\, d \sigma(x) &\sim& { \kappa^{-2}}  \quad\hspace*{2.5cm} \text{ if }\ \mathrm{supp}(\gamma) \subset \partial D,\ \ \zeta=u, v;  \label{haha1} \\
\int_{\Gamma_R} \gamma(x) \mathscr{I}_{u}^0(\kappa,x,\varepsilon)\, d \sigma(x) &=& \mathcal{O}(Q^{d-1} R^{3-d} \kappa^{1-d}) \quad \text{ if }\ \mathrm{supp}(\gamma)\cap \partial D = \emptyset; \label{haha2} \\
\int_{\Gamma_R} \gamma(x) \mathscr{I}_{v}^0(\kappa,x,\varepsilon)\, d \sigma(x) &=& \mathcal{O}(R^{3-d} \kappa^{1-d}) \quad\hspace*{.85cm} \text{ if }\ \mathrm{supp}(\gamma) \cap \partial D = \emptyset  \,; \label{haha3}
\eqn
and
\beqn
\int_{\Gamma_R} \gamma(x) \mathscr{I}_{\zeta}^1(\kappa,x,\varepsilon)\, d \sigma(x) &\sim& 1  \quad\hspace*{2.9cm} \text{ if }\ \mathrm{supp}(\gamma) \subset \partial D,\ \ \zeta=u, v;  \label{haha4} \\
\int_{\Gamma_R} \gamma(x) \mathscr{I}_{u}^1(\kappa,x,\varepsilon)\, d \sigma(x) &=& \mathcal{O}(Q^{d-1} R^{1-d} \kappa^{1-d}) \quad \text{ if }\ \mathrm{supp}(\gamma)  \cap \partial D = \emptyset;  \label{haha5} \\
\int_{\Gamma_R} \gamma(x) \mathscr{I}_{v}^1(\kappa,x,\varepsilon)\, d \sigma(x) &=& \mathcal{O}(R^{1-d} \kappa^{1-d}) \quad\hspace*{.9cm} \text{ if }\ \mathrm{supp}(\gamma) \cap \partial D = \emptyset \,, \label{haha6}
\eqn
where the asymptotic constants in the RHS terms of the above relations depend on $\varepsilon$ and $\|\gamma\|_{C(\Gamma)}$.
\end{Theorem}

\begin{Remark}
By Theorem~\ref{thm:main1}, the surface concentration of the generalised transmission eigenfunctions can be observed as follows. By \eqref{haha1} and \eqref{haha4}, we readily see that, {\color{black} whenever $\kappa^2$ is not a Dirichlet eigenvalue of $D$,} the generalised transmission eigenfunctions $u$ and $v$ are highly oscillatory around $\partial D$ (in an averaging sense as described by the averaging functionals). Indeed, by multiplying a normalisation factor $\kappa^2$, namely considering $\kappa\cdot(u, v)$, we see that the gradient fields blow up as $\kappa\rightarrow \infty$. Considering $\kappa\cdot(u, v)$, and by \eqref{haha1}--\eqref{haha3} and \eqref{haha4}--\eqref{haha6}, we readily see that in particular when $d\geq 4$, $\kappa\cdot u(x)$ and $\kappa\cdot v(x)$ decay rapidly when $x$ leaves away from $\partial D$ (again in the averaging sense) for $\kappa$ sufficiently large. 
{Even in the case with $d=3$, we can also see that $\kappa\cdot u$ decays when leaving away from $\partial D$ if $0<Q<1$}. According to our discussion in Section~\ref{sect:3}, the generalised transmission eigenfunctions contain the exact transmission eigenfunctions as subsequences, and hence it is unobjectionable to claim that concentration property in Theorem~\ref{thm:main1} holds also for the exact transmission eigenfunctions (in an averaging sense).
%

\end{Remark}

We proceed to give the proof of Theorem~\ref{thm:main1}. First, we discuss the generalised Weyl's law, which shall be needed in our proof. To that end, we introduce the following Hamiltonian:
\begin{equation}\label{eq:hf1}
\mathcal{H}(x, \xi) :=  p_{\mathcal{A}^{\kappa}_{-2}  } (x,\xi): T^*(\partial D)  \rightarrow \mathbb{R} \,,
\end{equation}
{\color{black}
where we have explicitly that
\begin{equation}\label{eq:hf1.1}
\mathcal{H}(x, \xi) = \rho_Q \left( \| \xi \|_{g(x)}^2 \right) \,,
\end{equation}
with $\|\xi\|_g(x) := \sum_{ij}^{d-1} g^{ij}(x) \xi_i \xi_j$ under a general coordinate system, and $\rho_Q $ being defined as
\begin{eqnarray}\label{eq:hf1.2}
\rho_Q : [0, + \infty] &\rightarrow& [-\infty,1]; \\
\rho_Q (t) &=& 
\begin{cases}
1 - \frac{4 t^2 \left( Q^2 +1-2t - 2 \sqrt{(1-t)(Q^2-t)} \right) }{(Q^2-1)^2 (Q^2-t) }   & \text{if } t \leq  \min\{1,Q^2\} \text{ or } \max\{1,Q^2\} \leq t ,\\
1 - \frac{4 t^2}{(Q^2-1) (Q^2-t) }  & \text{if } \min\{1,Q^2\} \leq t \leq  \max\{1,Q^2\},
\end{cases}
\nonumber
\end{eqnarray}
is differentiable except at $t= 1,Q^2$, 
and we recall that
\begin{eqnarray*}\label{eq:hf1.3}
\rho_Q (t) &=& 1 -  \frac{4}{Q^2 (Q^2+1)} t^2 - \frac{4}{Q^4 (Q+1)} t^3  + \mathcal{O}( t^{4}  )  \quad  \text{ as } t \rightarrow 0^+  \,, \\
\rho_Q (t) &=& - \frac{3Q^2 + 1}{2} t^{-1} + \mathcal{O}( t^{-2}  )  \quad  \text{ as } t \rightarrow + \infty  \,.
\end{eqnarray*}
We remark that
\begin{eqnarray}\label{eq:hf1.4}
\mathcal{A}^{\kappa}_{-2}    = \rho_Q( - h^2 \Delta_{\partial D} )  \quad   (\text{mod} \, h \Phi \text{SO}^{-3}_h),
\end{eqnarray}
where $\rho_Q( - h^2 \Delta_{\partial D} )  $ is now defined via functional calculus over the operator $- h^2 \Delta_{\partial D}$ whenever $Q^2 \kappa^2$ is not an eigenvalue of $-  \Delta_{\partial D}$ on the surface $\partial D$.}

Now, we are ready to state the Weyl's law for our purpose.
{\color{black}
The precise statement for Weyl's law for the special case when the operator is the Laplace or Schrodinger operator is given in \cite{Zworski} , Theorem 15.3. Its form for our compact self-adjoint operator comes directly by following the proof of Theorem 15.3 and Theorems 14.8 to 14.10 in \cite{Zworski}, as well as by noticing $| p_{\mathcal{A}^{\kappa}_{-2} }(x,\xi)^{-1} + i | \geq C (1+|\xi|^2) $ for some $C > 0$, which gives the desired 
resolvent estimate when we utilize the Helffer-Sjostrand formula; see also \cite{ACL2} from Lemma 3.3 to Proposition 3.5 and \cite{ACLS} from Lemma 4.3 to Proposition 4.5 for more details.
A precise statement of pointwise Weyl's law for a general elliptic pseudo-differential operator can be found in \cite{Hor5} Theorem 5.1.}

\begin{Proposition} \cite{trace,erg1,Zel,erg3,erg2,Zworski,Hor5}\label{prop:w1}
Assume that $D$ is a 
$C^\infty$ domain
and {\color{black} $\kappa^{2}$ is not a Dirichlet eigenvalue of $D$.} {\color{black} Given $a \in S^m (T^*(\partial D))$,} fixing $\varepsilon > 0$ sufficiently small, we have as $h\rightarrow 0^+$,
\begin{equation}\label{weyl}
\begin{split}
& (2 \pi h )^{ (d-1)}  \sum_{1-\varepsilon \leq  \lambda_j (h) \leq 1} \big\langle  \mathrm{Op}_{a,h}  \, \phi_j(h) , \phi_j(h) \big\rangle_{L^2(\partial D)}\\
=&  \int_{ \{ 1-\varepsilon \leq \mathcal{H} \leq  1\} } a  \, d \sigma \otimes d \sigma^{-1} + o_{\varepsilon}(1),
\end{split}
\end{equation}
where $\phi_j(h)$'s are given in \eqref{eq:ppn1} and the little-$o$ depends on $\varepsilon$.
\end{Proposition}

\begin{Corollary} Assuming that $D$ is a
$C^\infty$ domain and {\color{black} $\kappa^{2}$ is not a Dirichlet eigenvalue of $D$}, and fixing $\varepsilon > 0$ sufficiently small, we have
\beqn
\mathfrak{m}(\kappa,\varepsilon) =  \# \{ 1-\varepsilon \leq  \lambda_j (h) \leq 1 \} =  (2 \pi  h)^{ 1-d } \,
  \left(  \int_{ \{ 1-\varepsilon \leq \mathcal{H} \leq  1\} }   \, d \sigma \otimes d \sigma^{-1} + o_{\varepsilon}(1)  \right)  \,.
\label{weyl2}
\eqn
{where $\mathfrak{m}(\kappa,\varepsilon)$ is defined as in \eqref{Definition_m}.}
Hence, for any given $\varepsilon > 0$ sufficiently small, there exits a $\kappa_\varepsilon$ such that any $\kappa > \kappa_\varepsilon$ {\color{black} with $\kappa^{2}$ being not a Dirichlet eigenvalue of $D$} is a generalised transmission eigenfunction with multiplicity $\mathfrak{m}(\kappa,\varepsilon) \sim \kappa^{d-1} $ according to Definition~\ref{almost_transmission}.
\end{Corollary}
\begin{proof}
The first equality in \eqref{weyl2} comes from the immediate observation that, counting multiplicity,
\[
 \# \{ 1-\varepsilon \leq  \lambda_j (h) \leq 1 \}  = \mathrm{dim} \bigcup_{ 1-\varepsilon \leq  \lambda_j (h) \leq 1 } \left \{ \phi \in L^2(\partial D) \,:\,  \left( \mathcal{A}^{\kappa,Q}_{\partial D}  - \lambda_j (h)  \right)  [\phi]  = 0 \right \} = \mathfrak{m}(\kappa,\varepsilon) \,.
\]
The second equality in \eqref{weyl2} is the classical Weyl's law, which can be easily obtained from Proposition~\ref{prop:w1} by taking $a=1$ in \eqref{weyl}. {The last conclusion comes from the observation that, given $\varepsilon >0$, we check that $(2 \pi  h)^{ 1-d } $ is sufficiently large for $h = \kappa^{-1}>0$ sufficiently small and hence $\mathfrak{m}(\kappa,\varepsilon) $ is sufficiently larger than $1$. Therefore by definition, we can always find $ \kappa_\varepsilon$ such that for all $\kappa > \kappa_\varepsilon$, we have $\mathfrak{m}(\kappa,\varepsilon)  > 1$ and therefore $\kappa $ is a generalised transmission eigenfunction with multiplicity $\mathfrak{m}(\kappa,\varepsilon) \sim \kappa^{d-1} $.
}
\end{proof}

{\color{black}

\noindent We quickly remark that gazing at the expression \eqref{eq:hf1.1} of $\mathcal{H}$ and \eqref{eq:hf1.2} of $\rho_Q$, whenever $\varepsilon$ is sufficiently small with
$0 < \varepsilon < \min \left\{ 1, Q^4, \frac{Q^8 (Q+1)^2 }{64} \right\}$, we have
\beqnx
 \int_{ \{ 1-\varepsilon \leq \mathcal{H} \leq  1\} }   \, d \sigma \otimes d \sigma^{-1} =  \mathcal{O} \left(  \int_0^{ \frac{ |Q^2-1|^{\frac{1}{2}}  \varepsilon^{\frac{1}{4}} }{2}  } s^{d-2} ds  \right) =  |Q^2-1|^{\frac{d-1}{2}}  \, \mathcal{O} \left(  \varepsilon^{\frac{d-1}{4}}  \right) \,.
\eqnx

}

%
%
%
%
\begin{proof}[Proof of Theorem~\ref{thm:main1}]
By choosing $a(x,\xi) $ as a smooth non-negative bump function $\gamma \in \mathcal{C}^\infty$ either on $\partial D$ or $\Gamma_R$  multiplied with appropriate choices of symbols, we can obtain the averaged results for $\zeta_j(h)$ and $ \nabla \zeta_j(h) $ ($\zeta = u \text{ or } v$) with pointwise concentration. Our argument is inspired by that developed in \cite{ACL2} in a different context. In what follows, we denote a global $C^\infty$ diffeomorphism $F_R: \partial D  \rightarrow \Gamma_R $ (cf. Section~\ref{sect:2.3}).  
We also recall $ \gamma $ is a smooth nonnegative bump function compactly supported on $\Gamma_R$.

First, noting the fact that
\[
u_j(h) = \mathcal{S}^{\kappa Q}_{\partial D}  \, [ \phi_j (h) ]  =  h \left(  \mathcal{S}^{\kappa Q}_{-1} + h^2  \mathcal{S}^{\kappa Q}_{-3}  \right) [ \phi_j (h) ]  \,,
\]
we can make a choice of symbol $ a(x,\xi) = \gamma(x) |p_{ \mathcal{S}^{\kappa Q}_{-1}}(x,\xi) |^2$ in \eqref{weyl}, together with \eqref{weyl2} to obtain as $h\rightarrow+0$:
\beqn
& &\frac{ \sum \limits_{1-\varepsilon \leq  \lambda_j (h) \leq 1}   \int_{\partial D} \gamma(x)  |  [u_j(h)] (x) |^2 d \sigma(x) }{   \# \{ 1-\varepsilon \leq  \lambda_j (h) \leq 1 \} }  \notag \\
&= & h^2 \frac{ \sum \limits_{1-\varepsilon \leq  \lambda_j (h) \leq 1}   \left \langle \mathrm{Op}_{\gamma(x) |p_{ \mathcal{S}^{\kappa Q}_{-1}}(x,\xi) |^2 ,h }   \phi_j (h) , \phi_j (h)   \right \rangle_{L^2(\partial D)} }{   \# \{ 1-\varepsilon \leq  \lambda_j (h) \leq 1 \} } + \mathcal{O}(h^3)  \notag \\
&=& h^2 \left(  \frac{     \int \limits_{ \{ 1-\varepsilon \leq \mathcal{H} \leq 1 \} }  \gamma(x)  | p_{ \mathcal{S}^{\kappa Q}_{-1}}(x,\xi) |^2  d \sigma \otimes d \sigma^{-1}    }{    \int \limits_{ \{ 1-\epsilon \leq \mathcal{H} \leq 1 \} }   d \sigma \otimes d \sigma^{-1}  } + o_{\varepsilon, \gamma}(1) \right)  \, ,
\label{concentration1u}
\eqn
which readily gives \eqref{haha1} with $\zeta = u$.

In a similar manner, from
\[
\partial_\nu u_j(h) =  ( - \f{1}{2} I + {\mathcal{K}^{\kappa Q}_{\partial D}}^* ) \, [ \phi_j (h) ] =    \left( - \f{1}{2} I  +
 \frac{ h }{2}  \mathcal{K}^{\kappa Q}_{-1} + \frac{h^2}{2} \mathcal{K}^{\kappa Q}_{-2}   \right ) \, [ \phi_j (h) ],
\]
we can make a choice of symbol $ a(x,\xi) = \frac{1}{4} \gamma(x) $ in \eqref{weyl}, together with \eqref{weyl2}, to obtain as $h\rightarrow+0$:
\beqn
\frac{ \sum \limits_{1-\varepsilon \leq  \lambda_j (h) \leq 1}   \int_{\partial D} \gamma (x) |   [ \partial_\nu u_j(h) ] (x) |^2 d \sigma(x) }{   \# \{ 1-\varepsilon \leq  \lambda_j (h) \leq 1 \} } =  \frac{1}{4} \frac{     \int \limits_{ \{ 1-\varepsilon \leq \mathcal{H} \leq 1 \} }  \gamma(x)    d \sigma \otimes d \sigma^{-1}   }{     \int \limits_{ \{ 1-\epsilon \leq \mathcal{H} \leq 1 \} }      d \sigma \otimes d \sigma^{-1}   } + o_{\varepsilon, \gamma}(1) . \label{concentration2ua}
\eqn
Furthermore, noting that
\[
\partial_{x_i} u_j(h) = \frac{ \mathrm{Op}_{ \xi_i ,h } }{h} \circ \mathcal{S}^{\kappa Q}_{\partial D}  \, [ \phi_j (h) ]  \,,
\]
we can make a choice of symbol $ a(x,\xi) = \gamma(x) |\xi |^2  | p_{ \mathcal{S}^{\kappa Q}_{-1}}(x,\xi) |^2 $ in  \eqref{weyl}, together with \eqref{weyl2} to obtain as $h\rightarrow+0$:
\beqn
&&\frac{ \sum \limits_{1-\varepsilon \leq  \lambda_j (h) \leq 1}   \int_{\partial D} \gamma(x) |  [\partial_x u_j(h)] (x) |^2 d \sigma(x) }{   \# \{ 1-\varepsilon \leq  \lambda_j (h) \leq 1  \} }\notag\\
 &=&  \frac{       \int \limits_{ \{ 1-\varepsilon \leq \mathcal{H} \leq 1 \} }  \gamma(x) |\xi |^2  | p_{ \mathcal{S}^{\kappa Q}_{-1}}(x,\xi) |^2   d \sigma \otimes d \sigma^{-1}    }{    \int \limits_{ \{ 1-\epsilon \leq \mathcal{H} \leq 1 \} }    d \sigma \otimes d \sigma^{-1}    } + o_{\varepsilon, \gamma}(1).
\label{concentration2ub}
\eqn
By combining \eqref{concentration2ua} and \eqref{concentration2ub}, we readily have \eqref{haha4} with $\zeta = u$.

Next, recalling the diffeomorphism $F_R: \partial D\rightarrow\Gamma_R$, we notice that for $x \in \partial D$, it holds that
\beqnx
u_j(h) (F_R(x)) = \widetilde{\mathcal{S}^{\kappa Q}_{\partial D}}  \, [ \phi_j (h) ] (x) \,,\, \quad [ X \circ u_j(h) ] (F_R(x)) = \widetilde{ X \circ \mathcal{S}^{\kappa Q}_{\partial D}}  \, [ \phi_j (h) ] (x) \,, \\
v_j(h) (F_R(x)) = \widetilde{\mathcal{S}^{\kappa}_{\partial D}}  \, [ {\varphi}_j (h) ] (x) \,,\, \quad [ X \circ v_j(h) ] (F_R(x)) = \widetilde{ X \circ \mathcal{S}^{\kappa}_{\partial D}}  \, [ {\varphi}_j (h) ] (x) \,.
\eqnx
Considering $u_j(h) (F_R(x))$, we can make use of the explicit expressions in Lemma~\ref{lemma_SK_tilde}, as well as the respective choice of symbols $ a(x,\xi)  = \mathrm{det} (D F_R)^{-1}(F_R(x)) \, | p_{ \widetilde{\mathcal{S}^{\kappa Q}}_{0, -d-1} }(x,\xi) |^2 $  and $ a(x,\xi)  = \mathrm{det} (D F_R)^{-1}(F_R(x)) \, | p_{\widetilde{ ( X \circ  \mathcal{S}^{\kappa Q})}_{0, -d-1} } |^2 $ in \eqref{weyl}, together with a change of variable and \eqref{weyl2} to obtain as $h\rightarrow+0$:
\beqn
\frac{ \sum \limits_{1-\varepsilon \leq  \lambda_j (h) \leq 1}  \int_{\Gamma_R} \gamma(x) | u_j(h) (x)  |^2 d \sigma(x)  }{   \# \{ 1-\varepsilon \leq  \lambda_j (h) \leq 1  \} }  &=& \mathcal{O}_{\varepsilon, \gamma}  ( Q^{d-1} R^{3-d} h^{d-1} ),
\label{concentration3u} \\
\frac{ \sum \limits_{1-\varepsilon \leq  \lambda_j (h) \leq 1}   \int_{\Gamma_R} \gamma(x) |[ X \circ u_j(h) ] (x)  |^2 d \sigma(x) }{   \# \{ 1-\varepsilon \leq  \lambda_j (h) \leq 1  \} }
 &=& \mathcal{O}_{\varepsilon, \gamma, X }  ( Q^{d-1} R^{1-d} h^{d-1} ) \,.
\label{concentration4u}
\eqn
Now, one can directly verify that \eqref{concentration3u} gives \eqref{haha2}, and taking $X = e_j$ in \eqref{concentration4u} and summing them all up gives \eqref{haha5}.

Finally, with a quick observation that
\[
p_{  \left[ \mathcal{S}^{\kappa}_{\partial D}   \right]^{-1}   \mathcal{S}^{\kappa Q}_{\partial D}  }(x,\xi) = \frac{p_{ \mathcal{S}^{\kappa Q}_{-1}}(x,\xi)}{ p_{ \mathcal{S}^{k}_{-1}}(x,\xi)} + \mathcal{O}(h^2 |\xi|^{2})   \,,
\]
we obtain all the $\zeta = v$ counterparts by multiplying the symbol $a$ in all of the previous five choices to obtain as $h\rightarrow +0$:
\beqn
&&\hspace*{-.5cm}\frac{ \sum \limits_{1-\varepsilon \leq  \lambda_j (h) \leq 1}   \int_{\partial D} \gamma (x) |  [v_j(h)] (x) |^2 d \sigma(x) }{   \# \{ 1-\varepsilon \leq  \lambda_j (h) \leq 1  \} }\notag\\
 &&= h^2 \left(  \frac{     \int \limits_{ \{ 1-\varepsilon \leq \mathcal{H} \leq 1 \} }   \gamma(x) | p_{ \mathcal{S}^{\kappa Q}_{-1}}(x,\xi) |^2   d \sigma \otimes d \sigma^{-1}    }{   \int \limits_{ \{ 1-\epsilon \leq \mathcal{H} \leq 1 \} }     d \sigma \otimes d \sigma^{-1}    }  + o_{\varepsilon, \gamma}(1) \right), \notag \\
&&\hspace*{-.5cm}\frac{ \sum \limits_{1-\varepsilon \leq  \lambda_j (h) \leq 1}   \int_{\partial D} \gamma(x) |   [ \partial_\nu v_j(h) ] (x) |^2 d \sigma(x) }{   \# \{ 1-\varepsilon \leq  \lambda_j (h) \leq 1  \} }=  \frac{      \int \limits_{ \{ 1-\varepsilon \leq \mathcal{H} \leq 1 \} }   \gamma(x) \frac{|p_{ \mathcal{S}^{\kappa Q}_{-1}}(x,\xi)|^2}{ |p_{ \mathcal{S}^{\kappa}_{-1}}(x,\xi)|^2}    d \sigma \otimes d \sigma^{-1}   }{    \int \limits_{ \{ 1-\epsilon \leq \mathcal{H} \leq 1 \} }      d \sigma \otimes d \sigma^{-1}   }  + o_{\varepsilon, \gamma}(1),
\notag \\
&&\hspace*{-.5cm}\frac{ \sum \limits_{1-\varepsilon \leq  \lambda_j (h) \leq 1}   \int_{\partial D} \gamma (x) |  [\partial_x v_j(h)] (x) |^2 d \sigma(x) }{   \# \{ 1-\varepsilon \leq  \lambda_j (h) \leq 1  \} }= \frac{     \int \limits_{ \{ 1-\varepsilon \leq \mathcal{H} \leq 1 \} }   \gamma(x) |\xi |^2  | p_{ \mathcal{S}^{\kappa Q}_{-1}}(x,\xi) |^2   d \sigma \otimes d \sigma^{-1}    }{    \int \limits_{ \{ 1-\epsilon \leq \mathcal{H} \leq 1 \} }     d \sigma \otimes d \sigma^{-1}    }  + o_{\varepsilon, \gamma}(1) \,,  \notag
\eqn
and
\beqn
\frac{ \sum \limits_{1-\varepsilon \leq  \lambda_j (h) \leq 1}  \int_{\Gamma_R} \gamma(x) | v_j(h) (x)  |^2 d \sigma(x)  }{   \# \{ 1-\varepsilon \leq  \lambda_j (h) \leq 1 \} }  &=& \mathcal{O}_{\varepsilon, \gamma }  ( R^{3-d} h^{d-1} ),\notag \\
\frac{ \sum \limits_{1-\varepsilon \leq  \lambda_j (h) \leq 1}   \int_{\Gamma_R} \gamma (x) |[ X \circ v_j(h) ] (x)  |^2 d \sigma(x) }{   \# \{ 1-\varepsilon \leq  \lambda_j (h) \leq 1 \} }
 &=& \mathcal{O}_{\varepsilon, \gamma, X }  ( R^{1-d} h^{d-1} ). \notag
\eqn

The proof is complete.
\end{proof}

{\color{black}
\begin{Remark}
In fact we may improve our result to obtain a pointwise estimate with the local/pointwise Weyl's law given in \cite{Hor5} Theorem 5.1.  We defer this to a future work for the sake of simplicity.
\end{Remark}
}

\subsection{Quantum ergodicity and surface concentration almost surely }

In this subsection, we move onto obtaining another characterisation of the surface concentration of $\phi_j(h)$ and hence $u_j(h)$, $v_j(h)$.

{Before we continue, {\color{black} whenever $\kappa^2$ is not a Dirichlet eigenvalue of $D$,} let us recall the following solution under a Hamiltonian flow:
\begin{equation}\label{eq:hf2}
\begin{cases}
\frac{\partial}{\partial t} a_{x,\xi} (t) &=  \{ \mathcal{H} , a_{x,\xi}  (t) \}, \medskip\\
a_{x,\xi} (0)  &= a_0 (x,\xi) \in S^m (T^*(\partial D)),
\end{cases}
\end{equation}
where
$ \{ \cdot , \cdot \}$ is the Poisson bracket given by $\{ f , g\} := X_f \, g = - \omega(X_f, X_g)$
with $X_f$ being the symplectic gradient vector field given by $
\iota_{X_f} \, \omega =  d f $.
Hence, we have $\frac{\partial}{\partial t} a_{x,\xi}  = X_{ \mathcal{H} } a_{x,\xi}  $, and
$ a_{x,\xi}  (t) = a_0 ( \gamma(t), p (t)) $
where
\beqnx
\begin{cases}
 \frac{\partial}{\partial t}  (\gamma(t), p(t) ) &= X_{ \mathcal{H}} (\gamma(t), p(t) ),\medskip  \\
(\gamma(0), p(0) ) & = (x, \xi ) \in T^* (\partial D).
\end{cases}
\eqnx

Next, we recall the Heisenberg's picture and the lift to the operator level via Stone's and Egorov's theorems.
{\color{black}We refer to \cite{Zworski} Theorem 11.1 for an exact theorem statement, and \cite{Zworski} Theorem 11.12 for the long time Egorov theorem up to Ehrenfest time. We also refer to \cite{Egorov_exact_1} Theorem 1.2 to 1.5 and \cite{Egorov_exact_2} Theorem 2.2 for an improved estimate up to the Ehrenfest time. It is also referred to \cite{Ego} for the original statement (in Russian).
}

\begin{Proposition} \cite{Ego,Zworski} \label{Egorov}
The following operator evolution equation
\begin{equation}\label{eq:ohf1}
\begin{cases}
\frac{\partial}{\partial t} A_h(t) = \frac{\mathrm{i}}{ h } \left[ \mathrm{Op}_{\mathcal{H},h} , A_h(t) \right],\medskip\\
A_h(0) = \mathrm{Op}_{a_0,h},
\end{cases}
\end{equation}
defines a unique Fourier integral operator (up to $h^{\infty} \, \Phi \mathrm{SO}_h^{-\infty} $) for $t < C \log (h)$:
\beqnx
A_h(t) = e^{- \frac{\mathrm{i} t}{h}  \mathrm{Op}_{\mathcal{H}, h}  } \, A_h(0) \, e^{ \frac{\mathrm{i} t}{h}  \mathrm{Op}_{ \mathcal{H}, h}  }  + O(h \, \Phi \mathrm{SO}_h^{\,m-1} ) .
\eqnx
Moreover,
\beqnx
A_h(t) =  \mathrm{Op}_{a_{x,\xi} (t) ,h} + \mathcal{O}(h \, \Phi \mathrm{SO}_h^{\,m-1} )  \,,
\eqnx
or that $ p_{A_h(t)}(x,\xi) = a_{(x,\xi)}(t)  + \mathcal{O}( h |\xi|^{m-1})\,.$
\end{Proposition}

With the notion of $X_\mathcal{H}$ at hand, we now consider for $0 \leq \delta  \leq \varepsilon$ for sufficiently small $\varepsilon > 0$, the set $M_{X_{\mathcal{H}}} (  \{ \mathcal{H} =1 - \delta \}  )$ as the set of invariant measures on $\{ \mathcal{H} = 1 - \delta \} $ and also $M_{X_{\mathcal{H}},\text{erg}} (  \{ \mathcal{H} = 1 -\delta \}  )$ as the set of ergodic measures with respect to the Hamiltonian flow generated by $X_\mathcal{H}$ on $\{ \mathcal{H} = 1 - \delta \} $.
The Choquet's theorem can be applied to obtain the classical ergodic decomposition theorem:
\[
\sigma_{\{\mathcal{H}=1 - \delta \}}  = \sigma_{\{\mathcal{H}=1 - \delta \}} (\{ \mathcal{H} = 1 - \delta \}) \int_{M_{X_{\mathcal{H}},\text{erg}} (  \{ \mathcal{H} = 1 - \delta \}  ) }  \mu \,  d \nu_{1-\delta} \left(\mu\right) \,,
\]
where $\sigma_{\{\mathcal{H}=1 - \delta \}}$ is the Liouville measure on the surface $\{ \mathcal{H} = 1 - \delta \} $
and hence the disintegration theorem gives
\[
\begin{split}
&\int_{1-\varepsilon \leq \mathcal{H} \leq 1} f(x,\xi)  ( d \sigma \otimes d \sigma^{-1} ) (x,\xi)\\
 =&
 \int_{0}^\varepsilon  \sigma_{\{\mathcal{H}=1 - \delta \}} (\{ \mathcal{H} = 1 - \delta \})  \int \limits_{M_{X_{\mathcal{H}},\text{erg}} (  \{ \mathcal{H} = 1 - \delta \}  ) }  \int \limits_{\mathrm{supp}(\mu) }\frac{f(x,\xi)}{| (\partial_x \mathcal{H}, \partial_\xi \mathcal{H} ) (x,\xi)|}d \mu(x,\xi) \,  d \nu_{1-\delta} \left(\mu \right)  d \delta \,.
 \end{split}
\]
The above conclusion generalizes a related result that was established in our earlier work \cite{ACL2}. 
{\color{black}
The only difference is that in this work, $\mathcal{H}$ is no longer homogenous with respect to $\xi$.
On the other hand, gazing at \eqref{eq:hf1.1}-\eqref{eq:hf1.2}, whenever $\| \xi \|_{g(x)}^2 \neq 1,Q^2$, we have
\begin{eqnarray*}
X_{ \mathcal{H} }  = \rho_Q'(\|\xi\|_{g(x)}^2) X_{ \| \xi \|_{g(x)}^2 } .
\end{eqnarray*}
From the fact that $ \rho_Q'(\|\xi\|_{g(x)}^2) < 0 $ when $0 < \delta < \varepsilon$ whereas $ \rho_Q'(0) = 0$ when $\delta = 0$, the action of $X_\mathcal{H}$ on each invariant set level surface $\{ \mathcal{H} = 1 - \delta \} $ differs by only a rescaling constant $\rho_Q'\circ \rho_Q^{-1} (1-\delta) $ for each $0 < \delta < \varepsilon$, whereas $X_\mathcal{H}$ leaves the set $\{ \mathcal{H} = 1 \} $ unmoved.
Therefore 
$M_{X_{\mathcal{H}},\text{erg}} (  \{ \mathcal{H} = 1 - \delta \}  )  = M_{X_{\| \xi \|_{g(x)}^2},\text{erg}} (  \{ \| \xi \|_{g(x)}^2 = 1 \}  )$ is the same for each $0 < \delta < \varepsilon$. In fact, the Hamiltonian flow given by $\mathcal{H}$ (and hence any ergodic property) on each $\{ \mathcal{H} = 1 - \delta \} $ is the same as the geodesic flow on $  \{ \| \xi \|_{g(x)}^2 = 1 \}$.
}
With the above ergodic decomposition, similar to \cite{ACL2}, we can obtain the following lemma.

\begin{Lemma}
 \label{birkroff}
For any $\varepsilon > 0$ sufficiently small and all $a_0 \in \mathcal{S}^m(T^*(\partial D))$,
we have
\beqnx
\frac{1}{T}\int_0^T a_{x,\xi}(t) dt \rightarrow_{a.e. d \sigma \otimes d \sigma^{-1} \text{ and } L^2(\{ 1- \varepsilon \leq H \leq 1 \}, d \sigma \otimes d \sigma^{-1} ) } \bar{a}(x,\xi)\text{ as $T \rightarrow \infty$,}
\eqnx
for some $\bar{a} \in L^2 (\{ 1- \varepsilon \leq H \leq 1 \}, d \sigma \otimes d \sigma^{-1} ) $, and a.e. $ d \nu_{1-\delta} \,  d \delta $, we have
\beqnx
\bar{a} (x,\xi ) = \int_{\{H = 1- \delta \} } a_0   \, d \mu  \quad   \text{ a.e. } d \mu .
\eqnx
\end{Lemma}
}Now, we can utilize Egorov's lift in Proposition \ref{Egorov} and our definition of $\bar{a}$ in Lemma \ref{birkroff}, and follow a similar argument as in \cite{ACL2,ACLS}
 to obtain the following quantum ergodicity theorem with some generalization compared to the classical results in  \cite{erg1,Zel,erg3,erg2,erg_a,erg_b,trace,Ego,Zworski}.
\begin{Theorem}\label{thm:ergo1}
Fixing  $\varepsilon > 0$ sufficiently small, {\color{black} whenever $h^{-2}$ is not a Dirichlet eigenvalue of $D$,} we have the following (variance-like) estimate as $h\rightarrow+0$,
\begin{equation}\label{eq:est1}
\frac{1}{  \# \{ 1-\varepsilon \leq  \lambda_i (h) \leq 1 \}  }
 \sum_{ 1-\varepsilon \leq  \lambda_i(h) \leq 1}  \left|   \left \langle \mathrm{Op}_{ a - \bar{a}, h}   \,  \phi_i(h) ,  \phi_i(h) \right \rangle_{L^2}   \right|^2 \rightarrow 0.
 \end{equation}
\end{Theorem}

As an important consequence of Theorem~\ref{thm:ergo1}, by using Chebeychev's trick and a diagonal argument, one can readily derive the following quantum ergodicity result.  
{\color{black}
Before we state the result, for notational sake, from now on, given $\varepsilon$ and $h>0$, we will always write the following set
\[
 J(h) := \{j \in \mathbb{N} : 1-\varepsilon  \leq \lambda_j (h) \leq 1   \} \,.
\]
Now the quantum ergodicity result can be stated as follows.}

\begin{Corollary} \label{quantum_ergodicity}
Given $\varepsilon > 0$ sufficiently small, {\color{black} whenever $h^{-2}$ is not a Dirichlet eigenvalue of $D$,} there exists a subsequence $S(h) \subset J(h)$ of density $1$, i.e.
\[
\frac{\sum_{j \in S(h) } 1 }{ \sum_{j \in J(h) } 1 } = 1 + o_{\varepsilon}(1),
\]
such that whenever $h\rightarrow +0$ one has:
\beqnx
\max_{j \in S(h)} \left|   \left \langle  \mathrm{Op}_{a - \bar{a}, h}   \,  \phi_j(h) ,  \phi_j(h) \right \rangle_{L^2} \right| = o_{\varepsilon}(1)  \,.
\eqnx
\end{Corollary}

By using the quantum ergodicity result, we can arrive at the following lemma.

\begin{Lemma} \label{list_of_result}
Let $\Gamma_R$ and $\gamma\in C^\infty(\Gamma)$ be given in Theorem~\ref{thm:main1} with $R = \mathrm{dist}(\Gamma_R, \partial D)$. Given $\varepsilon >0$, {\color{black} whenever $h^{-2}$ is not a Dirichlet eigenvalue of $D$,}  there exists $S(h) \subset J(h)$ with $\, \frac{\sum_{j \in S(h) } 1 }{ \sum_{j \in J(h) } 1 } \sim 1 $ such that, as $h\rightarrow +0$, the following results hold simultaneously:
\beqn
\max_{j \in S(h)} \bigg| \int_{\partial D} \gamma(x) \left  | u_j(h) (x)  \right|^2 d \sigma(x) -  h^2  \left \langle  \mathrm{Op}_{\overline{a_1}, h}   \,  \phi_j(h) ,  \phi_j(h) \right \rangle_{L^2(\partial D)} \bigg| = o_{\varepsilon}(  {h^2 } ) , \label{ergodicityu1} \\
\max_{j \in S(h)} \bigg| \int_{\partial D} \gamma(x) \left  | \nabla u_j(h) (x)  \right|^2 d \sigma(x) -  \left \langle  \mathrm{Op}_{\overline{a_2}, h}   \,  \phi_j(h) ,  \phi_j(h) \right \rangle_{L^2(\partial D)} \bigg| = o_{\varepsilon}(1) , \label{ergodicityu2}
\eqn
and
\beqn
\hspace*{-3mm}\max_{j \in S(h)} \bigg| \int_{\Gamma_R} \gamma(x) \left  | u_j(h) (x)  \right|^2 d \sigma(x)  -  Q^{d-1} h^{d-1}  \left \langle  \mathrm{Op}_{\overline{a_3}, h}   \,  \phi_j(h) ,  \phi_j(h) \right \rangle_{L^2(\partial D)} \bigg| = o_{\varepsilon}({h^{d-1} }) , \label{ergodicityu3} \\
\max_{j \in S(h)} \bigg| \int_{\Gamma_R} \gamma(x) \left  | \nabla u_j(h) (x)  \right|^2 d \sigma(x)  -  Q^{d-1} h^{d-1}  \left \langle  \mathrm{Op}_{\overline{a_4}, h}   \,  \phi_j(h) ,  \phi_j(h) \right \rangle_{L^2(\partial D)} \bigg| = o_{\varepsilon}({h^{d-1} }) , \label{ergodicityu4}
\eqn
as well as
\beqn
\max_{j \in S(h)} \bigg| \int_{\partial D} \gamma(x) \left  | v_j(h) (x)  \right|^2 d \sigma(x)  -  h^2  \left \langle  \mathrm{Op}_{\overline{a_5}, h}   \,  \phi_j(h) ,  \phi_j(h) \right \rangle_{L^2(\partial D)} \bigg| = o_{\varepsilon}(1) , \label{ergodicityv1} \\
\max_{j \in S(h)} \bigg| \int_{\partial D} \gamma(x) \left  | \nabla v_j(h) (x)  \right|^2 d \sigma(x)  - \left \langle  \mathrm{Op}_{\overline{a_6}, h}   \,  \phi_j(h) ,  \phi_j(h) \right \rangle_{L^2(\partial D)} \bigg| = o_{\varepsilon}({h^{2} }) , \label{ergodicityv2}
\eqn
and
\beqn
 \max_{j \in S(h)} \bigg| \int_{\Gamma_R} \gamma(x) \left  | v_j(h) (x)  \right|^2 d \sigma(x) -  h^{d-1}  \left \langle  \mathrm{Op}_{\overline{a_7 }, h}   \,  \phi_j(h) ,  \phi_j(h) \right \rangle_{L^2(\partial D)} \bigg| = o_{\varepsilon}({h^{d-1} }) , \label{ergodicityv3} \\
\max_{j \in S(h)} \bigg| \int_{\Gamma_R} \gamma(x) \left  | \nabla v_j(h) (x)  \right|^2 d \sigma(x) -   h^{d-1}  \left \langle  \mathrm{Op}_{\overline{ a_8 }, h}   \,  \phi_j(h) ,  \phi_j(h) \right \rangle_{L^2} \bigg| = o_{\varepsilon}({h^{d-1} }) , \label{ergodicityv4}
\eqn
with
\beqnx
a_1(x,\xi) &:=& \gamma(x) | p_{ \mathcal{S}^{\kappa Q}_{-1}}(x,\xi) |^2, \\
a_2(x,\xi) &:=& \gamma(x)  + \varphi(x) |\xi |^2  | p_{ \mathcal{S}^{\kappa Q}_{-1}}(x,\xi) |^2 , \\
a_3(x,\xi) &:=&  \gamma(F_R(x))  \mathrm{det} (D F_R)^{-1}(F_R(x)) \,  |p_{ \widetilde{\mathcal{S}^{\kappa Q}}_{0, -d-1} }(x,\xi) |^2 = \mathcal{O}(R^{3-d}  { |\xi|^{-d-2} } ) ,\\
a_4(x,\xi) &:=& \gamma(F_R(x))   \mathrm{det} (D F_R)^{-1}(F_R(x)) \,  \left( \sum_{j=1}^d |p_{ \widetilde{ ( e_j \circ  \mathcal{S}^{\kappa Q})}_{0, -d-1} }(x,\xi) |^2 \right) = \mathcal{O}(R^{1-d}  { |\xi|^{-d-2} }), \\
a_5(x,\xi) &:=& \gamma(x) | p_{ \mathcal{S}^{\kappa Q}_{-1}}(x,\xi) |^2, \\
a_6(x,\xi) &:=&  \gamma(x)  \frac{|p_{ \mathcal{S}^{\kappa Q}_{-1}}(x,\xi)|^2}{ |p_{ \mathcal{S}^{\kappa}_{-1}}(x,\xi)|^2}   + \gamma(x) |\xi |^2  | p_{ \mathcal{S}^{\kappa Q}_{-1}}(x,\xi) |^2,   \\
a_7(x,\xi) &:=&   \gamma(F_R(x))  \mathrm{det} (D F_R)^{-1}(F_R(x)) \,  \frac{|p_{ \mathcal{S}^{\kappa Q}_{-1}}(x,\xi)|^2}{ |p_{ \mathcal{S}^{\kappa}_{-1}}(x,\xi)|^2}  \, |p_{ \widetilde{\mathcal{S}^{\kappa}}_{0, -d-1} }(x,\xi) |^2 = \mathcal{O}(R^{3-d}  {|\xi|^{-d-2} }), \\
a_8(x,\xi) &:=&  \gamma(F_R(x))   \mathrm{det} (D F_R)^{-1}(F_R(x)) \,  \frac{|p_{ \mathcal{S}^{\kappa Q}_{-1}}(x,\xi)|^2}{ |p_{ \mathcal{S}^{\kappa}_{-1}}(x,\xi)|^2} \left(  \sum_{j=1}^d |p_{ \widetilde{ ( e_j \circ  \mathcal{S}^{\kappa})}_{0, -d-1} }(x,\xi) |^2  \right)\\
& =& \mathcal{O}(R^{1-d} { |\xi|^{-d-2} }).  \\
\eqnx
\end{Lemma}
\begin{proof}
We obtain the result by choosing $a(x,\xi) $ as a smooth non-negative bump function $\gamma \in \mathcal{C}^\infty$ either on $\partial D$ or $\Gamma_R$ with $\mathrm{dist}(\Gamma_R, \partial D) = R > 0$ multiplied with appropriate symbols, and then applying Corollary \ref{quantum_ergodicity}.
For instance, we obtain the descriptions of $u_j$ as follows: \eqref{ergodicityu1} is obtained by choosing the symbol $ | p_{ \mathcal{S}^{\kappa Q}_{-1}}(x,\xi) |^2 $.
\eqref{ergodicityu2} comes from choosing the symbols $1 + | \xi |^2 | p_{ \mathcal{S}^{\kappa Q}_{-1}}(x,\xi) |^2 $.
\eqref{ergodicityu3} is resulted from taking $ \, \text{det} (D F_R)^{-1}(F_R(x)) \, | p_{ \widetilde{\mathcal{S}^{\kappa Q}}_{0, -d-1-0} }(x,\xi) |^2 $ and then apply a change of variable formula.
\eqref{ergodicityu4} comes from taking $ \text{det} (D F_R)^{-1}(F_R(x)) \, |p_{\widetilde{ ( X \circ  \mathcal{S}^{\kappa Q})}_{0, -d-1-0} } (x,\xi) |^{2}$ with $X$ being one of the constant coordinate vectors $\{ e_k \}_{k=1}^d$, and then summed over all symbols resulting from $X = e_k$.
Their $v$ counterparts \eqref{ergodicityv1}-\eqref{ergodicityv4} are obtained similary with a specific choice $Q = 1$.
The proof is complete.
\end{proof}

By using the above results, we can derive the following theorem which indicates that there are ``many" generalised transmission eigenfunctions which are localized around $\partial D$.

\begin{Theorem}  \label{thm:main2}
{\color{black} Let $ \varepsilon > 0$ be sufficiently small and $D$ be a 
$C^\infty$ domain.} Then, for any smooth closed surface $\Gamma_R\subset D$ {\color{black} such that $\mathrm{dist}(\Gamma_R, \partial D):=R\geq 0$} and any bump function $\gamma(x) \in \mathcal{C}^{\infty} (\Gamma_R)$, {\color{black} whenever $\kappa^{2}$ is not a Dirichlet eigenvalue of $D$,}  there exists $S(\kappa^{-1}) \subset J(\kappa^{-1}) $ with $\, \frac{\sum_{j \in S(\kappa^{-1}) } 1 }{ \sum_{j \in J(\kappa^{-1}) } 1 } \sim 1 $ such that {\color{black} whenever $j\in S(\kappa^{-1}) $}, we have, as $\kappa\rightarrow \infty$:
\beqnx
 \int_{\Gamma_R} \gamma(x) \left  | u_j(\kappa^{-1}) (x)  \right|^2 d \sigma(x)  &=& \mathcal{O}(  \kappa^{-2}  ) \hspace*{2.25cm}  \quad \mbox{if}\ \ \mathrm{supp}(\gamma) \subset \partial D,  \\
 \int_{\Gamma_R} \gamma(x) \left  | u_j(\kappa^{-1}) (x)  \right|^2 d \sigma(x)  &=& \mathcal{O}(Q^{d-1} R^{d-3}  \kappa^{1-d} )  \qquad \mbox{if}\ \ \mathrm{supp}(\gamma) \cap \partial D = \emptyset, \\
 \int_{\Gamma_R} \gamma(x) \left  |  v_j(\kappa^{-1}) (x)  \right|^2 d \sigma(x) &=& \mathcal{O}(\kappa^{-2} )  \hspace*{2.3cm}\quad \mbox{if}\ \ \mathrm{supp}(\gamma) \subset \partial D,   \\
\int_{\Gamma_R} \gamma(x) \left  | v_j(\kappa^{-1}) (x)  \right|^2 d \sigma(x)  &=& \mathcal{O}( R^{d-3} \kappa^{1-d})  \hspace*{1.3cm}\quad \mbox{if}\ \ \mathrm{supp}(\gamma) \cap \partial D = \emptyset,
\eqnx
and
\beqnx
 \int_{\Gamma_R} \gamma(x) \left  | \nabla u_j(\kappa^{-1}) (x)  \right|^2 d \sigma(x)  &=& \mathcal{O}(Q^{d-1} R^{d-1} \kappa^{1-d} )   \qquad \mbox{if}\ \ \mathrm{supp} (\gamma) \cap \partial D = \emptyset, \\
 \int_{\Gamma_R} \gamma(x) \left  | \nabla v_j(\kappa^{-1}) (x)  \right|^2 d \sigma(x) &=& \mathcal{O}( R^{d-1} \kappa^{1-d}) \hspace*{1.25cm}\quad \mbox{if}\ \ \mathrm{supp}(\gamma) \cap \partial D = \emptyset.
\eqnx
{\color{black}
Moreover, supposing $ \mathrm{supp}(\gamma) \subset \partial D$ and writing $\zeta = u,v$, there exist $C>0 $ and $\kappa_C >0 $ such that, for all $\kappa > \kappa_C$, there is another subsequence $\widetilde{S_\zeta (\kappa^{-1})} \subset J(\kappa^{-1})$ with the property that
{\color{black} whenever $j\in \widetilde{S_\zeta(\kappa^{-1})} $}, we have
\beqnx
\left| \int_{\Gamma_R} \gamma(x) \left  | \nabla \zeta_j(\kappa^{-1}) (x)  \right|^2 d \sigma(x) \right| &\geq&C  \,.
\eqnx
}In all of the above relations, the asymptotic constants in the RHS terms depend on $\varepsilon, \|\gamma\|_{C(\Gamma)}$.
\end{Theorem}

\begin{proof}

\noindent From the fact that
\[
 | \left \langle  \mathrm{Op}_{\overline{a}, h}   \,  \phi_i(h) ,  \phi_i(h) \right \rangle_{L^2(\partial D)}  | \leq \| \mathrm{Op}_{\overline{a}, h}\|_{\mathcal{L}(L^2(\partial D), L^2(\partial D))} \| \phi_i(h) \|^2_{L^2(\partial D)} \leq C_{a} \,,
\]
(bearing in mind that $ \| \phi_i(h) \|^2_{L^2} = 1$)
together with Lemma \ref{list_of_result}, we can arrive at the first $6$ conclusions of the theorem.  As an example, recalling $\kappa = h^{-1}$, we show directly from \label{ergodicityu1} that
\beqnx
&  & \max_{j \in S(\kappa^{-1})} \bigg| \int_{\partial D} \gamma(x) \left  |  u_j(\kappa^{f-1}) (x)  \right|^2 d \sigma(x) \bigg| \\
&\leq&  \max_{j \in S(\kappa^{-1})} | \kappa^{-2}  \left \langle  \mathrm{Op}_{\overline{a_1}, h}   \,  \phi_j(h) ,  \phi_j(h) \right \rangle_{L^2(\partial D)} | \\ &  & \quad +
\max_{j \in S(\kappa^{-1})} \bigg| \int_{\partial D} \gamma(x) \left  | u_j(\kappa^{-1}) (x)  \right|^2 d \sigma(x) -  \kappa^{-2}  \left \langle  \mathrm{Op}_{\overline{a_1}, h}   \,  \phi_j(h) ,  \phi_j(h) \right \rangle_{L^2(\partial D)} \bigg|  \\
&\leq& C_{a_1} \,  \kappa^{-2}  + o_{\varepsilon}(  {\kappa^{-2} } ) \\
&=&  \mathcal{O}(  \kappa^{-2}  )  .
\eqnx
The other five conclusions can be obtained in a similar manner.

The last two conclusions come from applying a pigeonhole principle to the sums \eqref{haha4} with $\zeta = u,v$ respectively.
{\color{black}
In fact, given $\mathrm{supp}(\gamma) \subset \partial D$, if we suppose otherwise that for all $C>0$ and $\kappa_C>0$, there exists $\kappa > \kappa_C$ with 
\beqnx
\max_{j \in S(\kappa^{-1})}  \left| \int_{\Gamma_R} \gamma(x) \left  | \nabla \zeta_j(\kappa^{-1}) (x)  \right|^2 d \sigma(x) \right| < C \,,
\eqnx
then by choosing $C := l^{-1}$ and $\kappa_C := \max \{ \kappa_{l-1}, l \}$ for $l \in \mathbb{N}$, we may iteratively find $\kappa_{l} > \max \{ \kappa_{l-1}, l \}$ such that 
\beqnx
\max_{j \in S(\kappa_l^{-1})}  \left| \int_{\Gamma_R} \gamma(x) \left  | \nabla \zeta_j(\kappa_l^{-1}) (x)  \right|^2 d \sigma(x) \right| < l^{-1} \,.
\eqnx
Therefore, we can create a sequence $\{ \kappa_l \}_{l=0}^\infty$ such that as $l \rightarrow \infty$ that
\beqnx
 \kappa_l \rightarrow \infty \quad 
\text{ and } \quad
\max_{j \in S(\kappa_l^{-1})}  \left| \int_{\Gamma_R} \gamma(x) \left  | \nabla \zeta_j(\kappa_l^{-1}) (x)  \right|^2 d \sigma(x) \right| \rightarrow 0 \,.
\eqnx
With this, one can directly verify that for such $\{ h_l  \}_{l=0}^\infty := \{ \kappa_l^{-1} \}_{l=0}^\infty$, we have as $l \rightarrow \infty$ that
\[
h_l \rightarrow 0 \quad 
\text{ and } \quad
\left| \frac{ \sum \limits_{1-\varepsilon \leq  \lambda_j (h_l) \leq 1}   \int_{\partial D} \gamma(x)  |  [\zeta_j(h_l)] (x) |^2 d \sigma(x) }{   \# \{ 1-\varepsilon \leq  \lambda_j (h_l) \leq 1 \} } \right|  \rightarrow 0 \,,
\]
and this contradicts to \eqref{haha4}, and therefore the result follows.
}

The proof is complete.
\end{proof}

%

\section*{Acknowledgment}
The work of Y. Deng was supported by NSF grant of China No. 11971487 and NSF grant of Hunan No. 2020JJ2038. The work of H Liu was supported by Hong Kong RGC General Research Funds (project numbers, 11300821, 12301420 and 12302919) and the NSFC/RGC Joint Research Fund, N\_CityU101/21. { \color{black} The authors are grateful to the very helpful discussion with K.L. Lee, W.T. Leung and the two anonymous referees for their tremendously helpful suggestions. }

\section*{Data Availability Statement}

This is a theoretical analysis work and has no associated data.

\end{document}